\documentclass[11pt]{article}
\usepackage{amssymb,graphics,amsmath,amsthm,graphicx,epsfig}
\usepackage[all]{xy}
\newtheorem{defn}{Definition}[section]
\newtheorem{theorem}[defn]{Theorem}
\newtheorem{prop}[defn]{Proposition}
\newtheorem{lem}[defn]{Lemma}
\newtheorem{fact}[defn]{Fact}
\newtheorem{cor}[defn]{Corollary}
\newtheorem{example}[defn]{Example}

\newtheorem{conj}[defn]{Conjecture}	
\newtheorem{que}[defn]{Question}
\newtheorem{prob}[defn]{Problem}

\def\zhs{{\mathbb Z}HS^3}
\def\lz{L{\mathbb Z}HS^3}

\begin{document}
\title{On the Alexander polynomial of lens space knot}
\author{Motoo Tange\\
\vspace{10pt}\\
  {\small Institute of Mathematics, University of Tsukuba Ibaraki \textup{305-8571}, Japan}\\
    {\small e-mail\textup{: \texttt{tange@math.tsukuba.ac.jp}}}}
\date{}
\maketitle 
\abstract
Ozsv\'ath-Szab\'o proved the property that any coefficient of Alexander polynomial of lens space knot is either $\pm1$ or $0$ and the non-zero coefficients are alternating.
Combining the formulas of the Alexander polynomial of lens space knots due to Kadokami-Yamada and Ichihara-Saito-Teragaito, we refine Ozsv\'ath-Szab\'o's property as the existence of simple curves included in a region in ${\Bbb R}^2$.
The existence of curves, that has no end-points connected, is just 1-component in a region, can search distribution of non-zero coefficients of the Alexander polynomial of the lens space knot.
This curve is much useful to obtain constraints of Alexander polynomials of lens space knots.
For example, we can investigate the location of the second, third and fourth non-zero coefficients.
The curve extracts new invariant $\alpha$-index.
The invariant is an important factor to determine Alexander polynomial of lens space knot.
We classify lens space surgeries that the Alexander polynomial is the same as a $(2,r)$-torus knot
and lens space surgeries with small genus and so on.
As well as lens space knots in $S^3$, we also deal with lens space knots in homology spheres, which the surgery duals are simple (1,1)-knots.
\section{Introduction.}
\subsection{Lens space surgery.}
\label{lssurgery}
Let $Y_r(K)$ denote an $r$-surgery along a knot $K$ in a $\zhs$ $Y$.
$\zhs$ stands for integral homology sphere.
We call the rational number $r$ {\it slope} of the Dehn surgery.
A lens space $L(p,q)$ is defined to be $p/q$-surgery of the unknot in $S^3$.
A knot $K\subset Y$ is called a {\it (positive) lens space knot} if a positive integral Dehn surgery of $K$ 
is a lens space.
In the same way, a knot $K\subset Y$ is called a {\it (positive) L-space knot} if a positive integral Dehn surgery of $K$ 
is an L-space.
Here an L-space is a rational homology sphere whose Heegaard Floer homology for any spin$^c$ structure is isomorphic to that of $S^3$.
The first examples of L-spaces are any lens spaces.
If $Y_p(K)$ is a lens space, we call the number $p$ {\it lens surgery slope}. 

Let $p$ be a positive integer.
The dual knot $\tilde{K}$ in $Y_p(K)$ is defined to be the core of the solid torus for $K$.
Then $[\tilde{K}]\in H_1(Y_p(K))$ is called {\it dual class}.
If $Y_p(K)$ is a lens space surgery on $\zhs$, 
we can assign an integer $k$ in $({\mathbb Z}/p{\mathbb Z})^\times$ to the dual class of the surgery, as defined in detail in Section~\ref{lsk}.
The number $k$ may assume $0<k<p/2$ and $(p,k)=1$.
The $k_2$ denotes the integer with $kk_2\pm1\bmod p$ and $0<k_2<p/2$.
Here we may assume $k\le k_2$ exchanging $k$ and $k_2$ if necessary.
Then we call $k_2$ {\it the second dual class}.
We call such a triple $(p,k,k_2)$ (also $(p,k)$) {\it (lens surgery) parameter}.

Let $p,k$ be relatively prime positive integers.
In Section~\ref{simple1bridgeknot}, we will define a positive lens space knot $K_{p,k}$
in a $\zhs$ $Y_{p,k}$ (the surgery dual is a simple 1-bridge knot in a lens space).
The $p$-surgery of $K_{p,k}$ is $L(p,k^2)$ and its lens surgery parameter is $(p,k)$.

In this paper we mainly deal with the following two types of lens space knots:
\begin{itemize}
\item Lens space knot in an $\lz$.
\item $K_{p,k}$ in $Y_{p,k}$.
\end{itemize}
Here, $\lz$ stands for L-space $\zhs$.
Furthermore, for such lens surgery $Y_p(K)$ we always assume that 
$$2g(K)-1\neq p,$$
even if it is not mentioned in each statement.
We recall the following two facts.
\begin{fact}
\label{fact1}
Let $p,k$ be relatively prime positive integers.
\begin{enumerate}
\item $K_{p,k}$ is a double-primitive knot in $Y_{p,k}$.
\item Any double-primitive knot is a lens space knot.
\end{enumerate}
\end{fact}
Immediately, we understand that $K_{p,k}$ is a lens space knot in $Y_{p,k}$.
The definition of double-primitive knot will be done in Section~\ref{simple1bridgeknot}.
These facts are proven in \cite{B} in the case of lens space surgery on $S^3$.
$K_{p,k}$ is a double-primitive knot in a $\zhs$.
An integral surgery of any double-primitive knot in a $\zhs$ produces
two solid tori.
This means the integral surgery of a double-primitive knot gives a lens space.

Conversely, the following is known as {\it Berge conjecture} due to Gordon in \cite{Kirby}.
\begin{conj}[Berge conjecture \cite{Kirby}]
\label{Bconj}
Any lens space knot in $S^3$ is a double-primitive knot.
\end{conj}
\begin{conj}[\cite{tan1}]
\label{ponconj}
If $K$ is a lens space knot $K$ in an $\lz$ with $2g(K)-1<p$ and surgery parameter $(p,k)$,
then the $\lz$ is $S^3$ or $\Sigma(2,3,5)$ and $K$ is isotopic to $K_{p,k}$.
Furthermore, any lens space knot $K_{p,k}$ in $\Sigma(2,3,5)$ is 
in one of the 20 families in \cite{tan1}.
\end{conj}

Here we write the inclusion relationship as below:
\begin{eqnarray*}
\{\text{Double primitive knots in $S^3$}\}&\subset& \{\text{Lens space knots in }S^3\}\\
&\subset&\{\text{Lens space knots in $S^3$ or $\Sigma(2,3,5)$}\}\\
&\subset&\{\text{Lens space knots in L-space ${\mathbb Z}HS^3$'s}\}
\end{eqnarray*}

The assertion that the first $``\subset"$ would be actually $``="$ is Conjecture~\ref{Bconj}. 
The statement that the third $``\subset"$ would be actually $``="$ is 
the first statement in Conjecture~\ref{ponconj}.
We describe the inclusion relationship between Berge knots and $K_{p,k}$ in $\zhs$.
Berge knots are defined to be knots consisting of 10 families in \cite{B}.
\begin{eqnarray*}
\{\text{Berge knots}\}&=&\{\text{$K_{p,k}\subset Y_{p,k}=S^3$}\}\\
&\subset&\{\text{$K_{p,k}\subset Y_{p,k}$}|(p,k)=1\}
\end{eqnarray*}
The first equality is proven in \cite{G2}.

\subsection{Lens space knot and its Alexander polynomial}
Any $(r,s)$-torus knot $T(r,s)$ is a typical example of lens space knot.
The restrictions related to lens space knot $K$ affect the Alexander polynomial $\Delta_K(t)$ in many cases.
We define the following:
\begin{defn}
We call the Alexander polynomial of a lens space knot in $Y$ {\it lens surgery polynomial} in $Y$.
We call $\Delta_{T(r,s)}$ {\it a torus knot polynomial}.
\end{defn}
Hence, torus knot polynomial is lens surgery polynomial in $S^3$.
Throughout this paper, we use a symmetrized polynomial as $\Delta_K(t)$.
In \cite{[10]}, \cite{3}, and \cite{tan2}, there are many results for lens surgery polynomials.
Here we introduce the following theorem by Ozsv\'ath-Szab\'o (in the case of L-space knot in $S^3$) and by Ichihara-Saito-Teragaito (in the case of $K_{p,k}$).
\begin{theorem}[Ozsv\'ath-Szab\'o \cite{3}, Ichihara-Saito-Teragaito \cite{IST}]
\label{OS1}
Suppose that $K$ is an L-space knot in $S^3$ or $K_{p,k}$ in $Y_{p,k}$.
Then the Alexander polynomial of $K$ is of form 
\begin{equation}
\label{OS}
\Delta_K(t)=(-1)^r+\sum_{j=1}^{r}(-1)^{j-1}(t^{n_{j}}+t^{-n_{j}}),
\end{equation}
for some decreasing sequence of positive integers $d=n_1>n_2>\cdots>n_r>0$.
\end{theorem}
The Alexander polynomial of $T(r,s)$ is computed by $\Delta_{T(r,s)}=t^{-(r-1)(s-1)/2}(t^{rs}-1)(t-1)/(t^r-1)(t^s-1)$.
Hence, this polynomial satisfies (\ref{OS}).
The form of the Alexander polynomial conditions in Theorem~\ref{OS1} are rewritten as follows:
\begin{equation}
\label{OScond}
\begin{cases}\text{The absolute values of coefficients are zero or one.\ (Flat)}\\
\text{The non-zero coefficients alternate in sign.\ (Alternating).}
\end{cases}
\end{equation}
It is well-known in \cite{YN} and \cite{3} that any L-space knot in any $S^3$ or $K_{p,k}$ is fibered.
Thus, in these cases the Seifert genus $g(K)$ for such a knot coincides with the degree $d(K)$ of $\Delta_K(t)$.
Any L-space knots in $Y$ ($Y$:$\lz$) with $Y-K$ irreducible is fibered (read p.545 in \cite{tan6}).
If $K$ is a lens space knot in an $\lz$ $Y$, then $Y-K$ is clearly irreducible, thus in such a case $K$ is also fibered and $d(K)=g(K)$ holds.

\subsection{Main question and results.}
\label{resultssection}
Suppose that $K$ is a lens space knot in an $L{\mathbb Z}HS^3$ or $K_{p,k}$.
The following is our main question.
\begin{que}
\label{mainq}
How are the non-zero coefficient terms of $\Delta_K(t)$ of a lens space knot $K$ distributed?
\end{que}
This question is regarded as the question for a refinement of Theorem~\ref{OS1}.
In the present section we introduce a series of main results related to Question~\ref{mainq}.
\subsubsection{Non-zero curve}
In Section~\ref{nonzero} we define {\it non-zero curve} for a lens space knot.
This curve is a complete embedding of several ${\mathbb R}$ (this means there are no end points in ${\mathbb R}^2$).
The curve presents distribution of non-zero coefficients.
We prove the following fundamental properties.
The necessary terminologies will be defined  in Section~\ref{nonzero}.
\begin{theorem}
\label{non-zerocor}
Let $K$ be a lens space knot in an $L{\mathbb Z}HS^3$ or $K_{p,k}$ for relatively positive integers $p,k$.
There is one non-zero curve only in each non-zero region.
\end{theorem}
The existence of non-zero curve satisfying such properties naturally induces flat and alternating conditions of Alexander polynomial.
The distribution of non-zero coefficients of lens surgery polynomials are controlled by {\it non-zero curve} on ${\mathbb R}^2$.
It is proven in Lemma~\ref{contained} that each curve lies in a non-zero region,
which has a ${\mathbb Z}$-action on it.
By the action the curve is also invariant.

All the results which are introduced below follow form this theorem.
For example, we can immediately give the following corollary coming from the existence of non-zero curve.
\begin{cor}
\label{thenumber}
Let $K$ be a lens space knot in an $\lz$ or $K_{p,k}$ for relatively prime positive integers $p,k$.
If $2r+1$ is the number of the non-zero coefficients in $\Delta_K(t)$, then
we have
$$k_2\le 2r+1.$$
\end{cor}
\subsubsection{Non-zero sequence, $\alpha$-index, and adjacent sequence.}
Let $\{n_i\}$ be non-zero exponents (exponents of non-zero coefficients) 
in a flat and alternating Laurent symmetric polynomial as in (\ref{OS}).
We call the decreasing sequence $(d=n_1,\cdots,n_r,n_{r+1}=0)$ {\it half non-zero sequence} and
the sequence $(d=n_1,n_2,\cdots,n_{2r},n_{2r+1}=-d)$ {\it (full) non-zero sequence}, where $n_{2r+2-i}=-n_i$.
Let $K$ be a knot whose Alexander polynomial is flat and alternating.
$NS_h(K)$ and $NS(K)$ denote the half non-zero sequence and the full non-zero sequence of $\Delta_K(t)$ respectively.
The decreasing sequences construct a flat and alternating polynomial uniquely.

For example, $NS_h(3_1)=(1,0)$ and
\begin{equation}
\label{237pret}
NS(Pr(-2,3,7))=(5,4,2,1,0,-1,-2,-4,-5),
\end{equation}
where $Pr(k_1,k_2,k_3)$ is the $(k_1,k_2,k_3)$-pretzel knot.

Computing lens surgery polynomial, we know that the coefficients $1$ and $-1$ are adjacent in this order for some region from the top term.
To measure the region which adjacent coefficients $1,-1$ appear in the coefficients of a lens surgery 
polynomial we define {\it $\alpha$-index}.
\begin{defn}[$\alpha$-index]
We assume that $K$ is a knot which $\Delta_K(t)$ is flat, alternating, and $\Delta_K(t)\neq 1$.
We define the $\alpha$-index of $K$ with non-zero sequence $NS(K)=(n_1,\cdots, n_{2r+1})$ to be
$$\alpha(K)=\max\{n_1-n_{2j+1}|n_{2i-1}-n_{2i}=1,\ 1\le\forall i\le j\le r\}.$$
If $\Delta_K(t)=1$, then $\alpha(K)=0$.
\end{defn}
For example, $\alpha(Pr(-2,3,7))=7$ holds.
This $\alpha$-index extracts one of the most important information in lens surgery polynomial, as mentioned as below.
\begin{theorem}
\label{a22a2cor}
Let $\alpha_0$ be a positive integer.
If $K$ is a lens space knot in an $\lz$ or $K_{p,k}$ with $\alpha(K)=\alpha_0$ and with surgery slope $p$, then the following inequality holds
$$p\le \alpha_0^2+2\alpha_0+2.$$

As a result, as long as we consider lens space knots in an $\lz$ or $K_{p,k}$, the lens surgery polynomials that the $\alpha$-index can be less than or equal to $\alpha_0$  are finitely many.
\end{theorem}
Namely, the forms of lens surgery polynomial with $\alpha(K)\le \alpha_0$ have finite variations only.
Of course, the genus of a lens space knot has the same property.
However, $\alpha$-index can determine more detailed form of lens surgery polynomial than genus (it coincides with the top degree).
As a corollary, genus of lens surgery polynomials with a fixed $\alpha$-index has an upper bound.
The $\alpha$-index has the natural limitation $0\le \alpha(K)\le 2g(K)$.
By the definition the condition of $\alpha(K)=0$ is equivalent to what $K$ is the unknot.
Since $\alpha(K)=1$ is not satisfied naturally for any of such a lens space knot $K$, any non-trivial lens space knot satisfies with $2\le \alpha(K)\le 2g(K)$.
We classify the lens surgery polynomials with $\alpha(K)=2$ (Corollary~\ref{trefoilorollary}) and $\alpha(K)=2g(K)$ (Theorem~\ref{2ntorusknot}).
In the former case, immediately the exponent of the fourth non-trivial term of the lens surgery polynomials of the form $\Delta_K(t)=t^g-t^{g-1}+t^{g-2}+\cdots$ ($g=g(K)$) is determined (Corollary~\ref{trefoilorollary}).

The non-zero sequence with $\alpha(K)=\alpha_0$ satisfies with
$$n_{2i-1}=n_{2i}+1\ (1\le i\le s-1),\ n_{2s-1}=n_1-\alpha_0.$$
Namely, among the region $n_1\ge x\ge n_{2s-1}$ the non-zero coefficients $\pm1$ are
adjacent, and we call the region {\it adjacent region}
and $AS(K)$ denotes $(n_1,n_3,\cdots,n_{2s-1})$ and it is called {\it the adjacent sequence} and the integer $s$ is called {\it adjacent length}.
For example, $AS(Pr(-2,3,7))=(5,2,0,-2)$ and adjacent length is $4$.



\subsubsection{Lens space knots with torus knot polynomial}
\label{lensspaceknotresult}
We classify the lens space knots that the lens surgery polynomials are the same ones as the torus knot polynomials $\Delta_{T(2,2g+1)}$.
\begin{theorem}
\label{2ntorusknot}
Let $K$ be a lens space knot in an $L{\mathbb Z}HS^3$ $Y$ with surgery parameter $(p,k,k_2)$.
Let $g$ be the genus of $K$.
Then the following conditions are equivalent each other:
\begin{enumerate}
\item $\Delta_K(t)=\Delta_{T(2,2g+1)}(t)$.
\item The lens surgery parameter of $Y_p(K)=L(p,q)$ is $(p,2)$.
\item The lens surgery parameter $(p,k)$ can be realized by the surgery of $(2,2g+1)$-torus knot
\item $k_2=2g$ or $k_2=2g+1$.
\item $\alpha(K)=2g$
\end{enumerate}
\end{theorem}
From this theorem, if the knot $K$ satisfies one of the conditions holds,
the lens space surgery parameters are $(4g+1,2,2g)$ or $(4g+3,2,2g+1)$.
The parameters are realized by $T(2,2g+1)$.
On the definition of realization, read Definition~\ref{rlzd}.
\begin{theorem}
\label{DT22r+1}
Let $p,k$ be the coprime positive integers.
If $\Delta_{K_{p,k}}(t)=\Delta_{T(2,2g+1)}(t)$ holds for any integers $g$, then $Y_{p,k}=S^3$ and $K_{p,k}=T(2,2g+1)$.
\end{theorem}
Further, we can find the following examples.
\begin{theorem}
\label{rmk1}
For coprime integers $p,k$ the following statements hold.
\begin{enumerate}
\item There exists a knot $K_{p,k}$ in a non-L-space ${\mathbb Z}HS^3$ $Y_{p,k}$ such that $\Delta_{K_{p,k}}=\Delta_{T(r,s)}$ for relatively prime integers $(r,s)$ and 
the complement $Y_{p,k}-K_{p,k}$ is not homeomorphic to $S^3- T(r,s)$.
\item There exists a knot $K_{p,k}$ in a non-L-space $Y_{p,k}$ such that $\Delta_{K_{p,k}}$ is not the lens surgery polynomial in an $\lz$.
\end{enumerate}
\end{theorem}
For example, $Y_{10,3}=\Sigma(2,3,7)$ and $\Delta_{K_{10,3}}=\Delta_{T_{3,7}}(t)$.
However, $S^3- T(3,7)$ and $\Sigma(2,3,7)- K_{10,3}$ are not homeomorphic.
In fact, we cannot obtain $\Sigma(2,3,7)$ by any Dehn surgery of $T(3,7)$ by Moser's result \cite{Mo}.
On the statement 2. in Theorem~\ref{rmk1}, $K_{23,7}$ is a double-primitive knot in $\Sigma(2,3,11)$ and does not have lens surgery polynomial in any $\lz$.
\subsubsection{Dual class and $\alpha$-index.}
\label{dualclassbounds}  

The dual classes for lens space knots with parameter $(p,k,k_2)$ have an upper bound $p/2$ by the definition:
$$k\le k_2\le p/2.$$
We give an upper bound by using the $\alpha$-index:
\begin{theorem}
\label{them2}
Let $K$ be a non-trivial lens space knot in an $\lz$ or $K_{p,k}$ for relatively prime positive integers $p,k$.
Then we have
\begin{equation}
\label{alphainequation}
k\le k_2\le \alpha(K)+1.
\end{equation}

Let $(n_1,n_3,\cdots, n_{2s-1})$ be the adjacent sequence of $K$.
Then either of the following conditions holds:
\begin{enumerate}
\item[(1)] $k_2=\alpha(K)+1$
\item[(2)] There exists an integer $0<s_2\le s$ such that $k_2=n_1-n_{2s_2-1}$.
\end{enumerate}
Furthermore, if $k<k_2$, then there exists an integer $s_1$ such that $k=n_1-n_{2s_1-1}$ and $1<s_1< s_2$ hold.
\end{theorem}
If $K$ is a lens space knot in an $\lz$ with $k=k_2$, then we have $k=k_2=1$ or $3$.
These surgery parameters are realized by the unknot or $K_{8,3}$ in $\Sigma(2,3,5)$ due to p.288(A) in \cite{tan1}.
Here we give the further refined theorem for the case of (1).
\begin{theorem}
\label{3th}
Let $K$ be a lens space knot in an $\lz$ or $K_{p,k}$ in $Y_{p,k}$ with parameter $p,k$ and $NS(K)=(n_1,n_2,\cdots, n_{2r+1})$ and adjacent length $s$.
Suppose that $k_2=\alpha(K)+1$.
\begin{enumerate}
\item[(a)] If $n_{2s-1}-n_{2s}>3$, then $n_2-n_3=1$.
\item[(b)] If $n_{2s-1}-n_{2s}=3$, then $n_2-n_3\ge 2$.
\item[(c)] If $n_{2s-1}-n_{2s}=2$, then $0\le n_2-n_3-(n_{2s}-n_{2s+1})\le 1$ holds.
\end{enumerate}
\end{theorem}

\subsubsection{The non-zero exponents $n_1,n_2,n_3$, and $n_4$.}
\label{nthtermofAlex}
We apply the non-zero curve to non-zero coefficients of lens surgery polynomial.
\begin{cor}
\label{main}
Let $K$ be a non-trivial lens space knot in an $\lz$ or $K_{p,k}$ in $Y_{p,k}$ for relatively prime positive integers $p,k$.
Let $\{n_i\}$ be the non-zero exponents of $\Delta_K(t)$ defined in Theorem~\ref{OS1}.
Then we have
\begin{equation}
\label{second}
n_1-n_2=1.
\end{equation}
\end{cor}
In the case of lens space knot in $S^3$, this corollary was proven by the author in \cite{tan3}, although, here in more general cases we reprove this corollary by using the non-zero curve.
It is proven in \cite{HW} and \cite{K} that $n_2=g-1$ for any L-space knot in $S^3$,
where $g$ is the genus of the knot.


By using Corollary~\ref{thenumber}, we determine the third and fourth non-zero coefficients of any lens space knot.
\begin{theorem}
\label{thirdfourth}
The non-zero sequence $\{n_i\}$ of a lens space knot in an $\lz$ or $K_{p,k}$ with at least $4$ non-zero coefficients of the
Alexander polynomial satisfies with the following:
$$1\le n_3-n_4\le 3.$$
Furthermore, the non-zero sequence of a lens space knot $K$ in an $\lz$ does not satisfy $n_3-n_4=3$.
\end{theorem}

%

In the case of lens space surgery in an $\lz$ or $K_{p,k}$ we classify the realization of lens surgeries with $g(K)\le \frac{k_2+4}{2}$, $g(K)\le 5$ or
lens surgery polynomials with at most $7$ non-zero coefficients.\\
{\bf Theorem~\ref{dle5}.}
{\it 
Let $K$ be a lens space knot in an $\lz$ with $g(K)\le 5$.
Then the parameter can be realized by either of the following lens space knots:
$$T(2,3),\ T(2,5),\ T(2,7),\ T(3,4),\ T(2,9),\ T(3,5),\ T(2,11),\text{ or }Pr(-2,3,7)$$
}\\
{\bf Theorem~\ref{234class}.}
{\it 
If a lens space knot $K$ in an $\lz$ satisfies $2g(K)-4\le k_2\le 2g(K)-2$, then 
the lens surgery parameters are $(11,3)$, $(14,3)$, $(19,7)$ and are realized by 
$$T(3,4), T(3,5)\text{ or }Pr(-2,3,7),$$
respectively.
}\\
For the proofs of these theorems, see Section~\ref{K2n2g-42g-1} and Section~\ref{57nonzerocoeff}.
On lens space surgeries with $2g(K)-1\le k_2$ the proof will be in Section~\ref{2d-1k2class}.
\section*{Acknowledgements}
This research is a collection of results proven from 2006 to 2010.
I would like express the gratitude for my supervisor Masaaki Ue in that period. 
I would also thank for seminar organized by Teruhisa Kadokami (was taken place in 2006 and 2007) and useful communication with Noriko Maruyama and Tsuyoshi Sakai as well.
I was supported by the Grant-in-Aid for JSPS Fellows for Young Scientists(17-1604) and JSPS KAKENHI Grant Number 24840006, and 26800031.


\section{Preliminaries}
\label{Prel}
In this section we introduce several definitions and facts which are used in this paper.
\subsection{Lens surgery parameters}
\label{lsk}
If a $p$-Dehn surgery $Y_p(K)$ of $\zhs$ $Y$ is a lens space $L(p,q)$,
then the dual knot $\tilde{K}$ of this surgery represents an element in $H_1(L(p,q))$.
Let $c$ be a core of the genus one Heegaard splitting of $L(p,q)$.
If $[\tilde{K}]=k[c]$, then $p,k$ are relatively prime.
Then the pair $(p,k)$ is called {\it (lens space surgery) parameter}.
The second dual class is the integer that is the inverse of $\pm k\bmod p$ and satisfies $0< k\le k_2<p/2$.
This is the precise definition of the lens surgery parameter $(p,k,k_2)$ which is already 
introduced in Section~\ref{lssurgery}.


\subsection{Simple 1-bridge knot}
\label{simple1bridgeknot}
Berge in \cite{B} defined {\it double-primitive knot}, which is a class of lens space knots in $S^3$.
We will define a double-primitive knot in a $\zhs$.
\begin{defn}[Double-primitive knot in a $\zhs$]
Let $Y$ be a $\zhs$ with at most Heegaard-genus 2.
Let $Y=H_0\cup_{\Sigma_2} H_1$ be a genus 2 Heegaard decomposition of $Y$.
Then $K\subset Y$ is double-primitive knot if $K$ is isotopic to a knot $K'$ so that $K'$ can lie in $\Sigma_2$ and the knot $K'$ gives a primitive element in both $\pi_1(H_i)$ ($i=0,1$).
\end{defn}


We define 1-bridge knot in a lens space.
Let $U_\alpha\cup U_\beta$ be a genus one Heegaard splitting of the lens space.
If $K$ is isotopic to the union of two arcs $A_\alpha$ and $A_\beta$ satisfying the following conditions,
then we call it {\it a 1-bridge knot} in the lens space:
\begin{itemize}
\item $A_\gamma=K\cap U_\gamma$ for $\gamma=\alpha, \beta$.
\item $A_\gamma$ is boundary parallel in $U_\gamma$ for $\gamma=\alpha, \beta$.
\end{itemize}
We call $K'$ a {\it 1-bridge position} of $K$.

Suppose that $K'$ is a knot in a lens space with 1-bridge position.
For $\gamma=\alpha,\beta$, let $D_\gamma$ be the meridional disk of the Heegaard solid tori $U_\gamma$.
If the $A_\alpha,A_\beta$ (in the sense above) are both simple arcs in $D_\alpha$ and $D_\beta$, then the 1-bridge knot is called {\it simple}.

Let $K$ be a simple 1-bridge knot $A_\alpha\cup A_\beta$.
Let integers $0,\cdots, p-1$ denote minimal intersection points in order between $\alpha$ and the projection of $\beta$ on $\partial U_\alpha$.
Suppose the arc $A_\alpha$ is connecting the intersection points $0$ and $k$.
Then we denote the simple 1-bridge knot by $\tilde{K}_{p,q,k}$.

The knot $\tilde{K}_{p,q,k}$ satisfies $[\tilde{K}_{p,q,k}]=k[c]$ in $H_1(L(p,q))$, where $c$ is a core circle of $U_\beta$.
Here we choose some orientation of $\tilde{K}_{p,q,k}$.
Suppose that some integral Dehn surgery along $\tilde{K}_{p,q,k}$ produces a $\zhs$.
Then integral Dehn surgery of the dual knot goes back to $L(p,q)$ again.
If the slope is positive integral, then the $\zhs$ is uniquely determined and called $Y_{p,k}$.
The dual knot in $Y_{p,k}$ is denoted by $K_{p,k}$.
Then we have $q\equiv k^2\bmod p$.
For the proof read Corollary~2.3 in \cite{tan4}.
The knot $K_{p,k}$ in $Y_{p,k}$ gives a double-primitive knot (Fact~\ref{fact1}).
See \cite{Sa} for the detail.
Thus $K_{p,k}$ gives a lens space knot in $Y_{p,k}$ with lens surgery parameter $(p,k)$.


As an example of $K_{p,k}$, see Figure~\ref{torusheegaard}.
Consider the torus obtained by identifying both the slides in the rectangle as the Heegaard torus $T=\partial U_\alpha$.
The $\alpha$ and $\beta$ are the boundary circles of $D_\alpha$ and $D_\beta$ in the two solid tori 
$U_\alpha$ and $U_\beta$ respectively.
The slope of $\beta$ is $p/q$ for the lens space $L(p,q)$.
The picture is the case of $p=3$ and $q=1$.
The two circles stand for the two points $A_\alpha\cap T$.
The broken line in Figure~\ref{torusheegaard} is the simple arc $A_\beta$ projected on $T$ along the compressing disk.
Joining the arcs, we produce a knot $\tilde{K}_{3,1}$ in the lens space $L(3,1)$.
\begin{figure}[htbp]
\begin{center}
{\unitlength 0.1in%
\begin{picture}( 16.7000, 13.5500)(  2.3000,-16.5500)%
%
\special{pn 13}%
\special{pa 700 300}%
\special{pa 1900 300}%
\special{pa 1900 1500}%
\special{pa 700 1500}%
\special{pa 700 300}%
\special{pa 1900 300}%
\special{fp}%
%
\special{pn 13}%
\special{ar 1145 1400 55 55  0.0000000  6.2831853}%
%
\special{pn 13}%
\special{ar 1530 1400 55 55  0.0000000  6.2831853}%
%
\special{pn 8}%
\special{pa 700 1400}%
\special{pa 1900 1400}%
\special{fp}%
%
\special{pn 8}%
\special{pa 1700 1400}%
\special{pa 1800 1640}%
\special{fp}%
%
\special{pn 8}%
\special{pa 790 1205}%
\special{pa 420 1080}%
\special{fp}%
\put(4.3000,-10.2000){\makebox(0,0){$\beta$}}%
\put(18.2500,-17.2000){\makebox(0,0){$\alpha$}}%
%
\special{pn 8}%
\special{pa 700 1500}%
\special{pa 1100 300}%
\special{fp}%
%
\special{pn 8}%
\special{pa 1100 1500}%
\special{pa 1500 300}%
\special{fp}%
%
\special{pn 8}%
\special{pa 1500 1500}%
\special{pa 1900 300}%
\special{fp}%
%
\special{pn 8}%
\special{pa 1120 1345}%
\special{pa 1460 300}%
\special{dt 0.045}%
%
\special{pn 8}%
\special{pa 1460 1500}%
\special{pa 1490 1445}%
\special{dt 0.045}%
\end{picture}}%
\caption{The knot $\tilde{K}_{3,1}$ in the Heegaard torus $\partial U_\alpha$ of $L(3,1)$.}
\label{torusheegaard}
\end{center}
\end{figure}
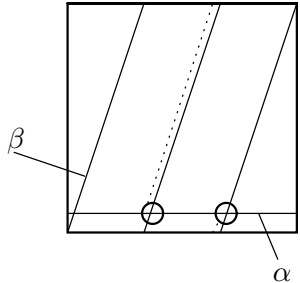


%
Here we define Dehn surgery realization.
\begin{defn}
\label{rlzd}
Let $(p,k)$ be a pair of relatively prime integers.
If there exists a lens space knot $K$ in a $\zhs$ $Y$ with the lens surgery parameter $(p,k)$, 
then we say that $(p,k)$ is realized by $K$ in $Y$.
\end{defn}

\subsection{Alexander polynomial formula of lens space knot}
\label{parameter}
We introduce the following theorem:
\begin{theorem}[\cite{[10]},\cite{tan2}]
\label{yamkatan}
Let $K$ be a lens space knot in a $\zhs$ $Y$ with surgery parameter $(p,k)$.
Then, for an integer $l$ with $kl=\pm1\bmod p$ and $\gcd(k,l)=1$, we have
$$\Delta_K(t)\equiv\Delta_{T(k,l)}(t)\ \ \bmod t^p-1.$$
\end{theorem}
Here we define the {\it smallest symmetric representative} $\bar{f}(t)$ of $f(t)$ in ${\Bbb Q}[t^{\pm1}]/(t^p-1)$
to be 
$$\bar{f}(t)=\begin{cases}
\sum_{|i|<\frac{p}{2}}\alpha_it^i&p\equiv 1\bmod 2\\
\sum_{|i|<\frac{p}{2}}\alpha_it^i+\frac{\alpha_{\frac{p}{2}}}2(t^{\frac{p}{2}}+t^{-\frac{p}{2}})&p\equiv 0\bmod 2,
\end{cases}$$
where $f(t)=\sum_{i\in {\mathbb Z}}\beta_i t^i$ and $\alpha_i=\sum_{j\equiv i\bmod p}\beta_j$.

If $2g(K)< p+1$, then Theorem~\ref{yamkatan} means 
$$\Delta_K(t)=\overline{\Delta_{T(k,l)}}(t).$$

Here we give the coefficient formula (Proposition~\ref{tangeprop}) of the Alexander polynomials of lens space knot.
This result has been proven in \cite{tan2}, and we will reprove it as a formula with a bit different form.

We put $kk_2\equiv e:=\pm1\bmod p$, $c=\frac{(k+1-p)(k-1)}{2}$, $m=\frac{kk_2-e}{p}$, and
$$I_\alpha=\begin{cases}\{1,2,\cdots,\alpha\}&\alpha>0\\\{\alpha+1,\alpha+2,\cdots,-1,0\}&\alpha<0\end{cases}$$
The bracket $[\cdot]_p$ stands for the least absolute remainder with respect to $p$.
Namely, the remainder satisfies $-\frac{p}{2}<[y]_p\le\frac{p}{2}$ for integer $y$.
In \cite{tan2}, the coefficient $a_i(K)$ of the symmetrized Alexander polynomial $\Delta_K(t)$ is computed as follows:
\begin{prop}[\cite{tan2}]
\label{tangeprop}
Let $K$ be a lens space knot in an $\lz$ with $Y_p(K)=L(p,k^2)$ and the parameter $(p,k,k_2)$ and $k_2^2=q_2\bmod p$.

If $2g(K)<p$, then we have
\begin{equation}\label{aiformula}
a_i(K)=\begin{cases}
-em+e\cdot\#\{j\in I_k|[q_2(j+ki+c)]_p\in I_{ek_2}\}&|i|\le g(K)\\0&|i|>g(K).\end{cases}\end{equation}

If $2g(K)=p$, then we have
$$a_i(K)=
\begin{cases}
-em+e\cdot\#\{j\in I_k|[q_2(j+ki+c)]_p\in I_{ek_2}\}&|i|<\frac{p}{2}\\
1&i=\pm\frac{p}{2}.\end{cases}$$
\end{prop}

Here we reprove this proposition again.
We define the $E$-function.
\begin{defn}
Let $\beta$ be a non-zero integer.
We define a function $E_\beta$ as follows:
$$E_{\beta}(\alpha)=\begin{cases}
e& [\alpha]_p\in I_\beta \\
0& \text{otherwise.}
\end{cases}$$
\end{defn}

Here we prove Proposition~\ref{tangeprop}.
\begin{proof}
Suppose that $2g(K)<p$.
We set the right hand side of (\ref{aiformula}) as $\alpha_i(K)$.
Let $l$ be a positive integer with $(k,l)=1$ and $l= k_2\bmod p$.
Then we have the following:
\begin{eqnarray}
&&t^{elc+1}(t^k-1)(t^{el}-1)\sum_{0\le i<p}\alpha_it^i\nonumber\\
&&\equiv\sum_{0\le i<p}(\alpha_{i-k-el-elc-1}-\alpha_{i-k-elc-1}-\alpha_{i-el-elc-1}+\alpha_{i-elc-1})t^i\bmod t^p-1.
\label{alex}
\end{eqnarray}
Then we have 
\begin{eqnarray*}
\alpha_{i-el}-\alpha_{i}&=&E_{ek_2}(q(ki+c))-E_{ek_2}(q(k(i+1)+c))\\
&=&E_{ek_2}(ek_2i+qc))-E_{ek_2}(ek_2(i+1)+qc))
\end{eqnarray*}
and 
\begin{eqnarray*}
&&\alpha_{i-k-el-elc-1}-\alpha_{i-k-elc-1}-\alpha_{i-el-elc-1}+\alpha_{i-elc-1}\\
&=&E_{ek_2}(ek_2(i-1)-1)-E_{ek_2}(ek_2(i-1))-E_{ek_2}(ek_2i-1)+E_{ek_2}(ek_2i)\\
&=&\begin{cases}
1&i=k+2,k\\-2&i=k+1\\0&\text{otherwise.}\end{cases}
\end{eqnarray*}
Thus, $(\ref{alex})$ is congruent to $t^k(t-1)^2$ and we have 
$$t^{elc+1}(t^k-1)(t^{el}-1)\Delta_K(t)\equiv t^k(t^{ekl}-1)(t-1)\bmod t^p-1.$$
Hence, we have 
$$t^{elc+1+\frac{1-e}{2}(kl-l)}(t^k-1)(t^{l}-1)\Delta_K(t)\equiv t^k(t^{kl}-1)(t-1)\bmod t^p-1$$
and 
\begin{equation}
\label{eq1}
\Delta_K(t)\equiv t^{-elc-1+k}\frac{(t^{kl}-1)(t-1)}{(t^k-1)(t^l-1)}
\end{equation}
in ${\Bbb Q}[t^{\pm1}]/\sum_{i=0}^{p-1}t^i$.
Here $elc+1-k+\frac{1-e}{2}(kl-l)=\frac12(k-1)(l-1)\bmod p$ holds.
In fact, $2elc+2-2k+(1-e)(kl-l)-(k-1)(l-1)=(k-1)(m+l)p\bmod 2p$.
If $(k-1)(m+l)\equiv 1\bmod 2$, then $k\equiv 0\bmod 2$ and $m+l=1\bmod 2$.
However, in this case, $p\equiv 1\bmod 2$, $l\equiv 1\bmod 2$ and $m\equiv 0\bmod 2$ hold.
These contradict about $kk_2=e+pm$.
As a result, $2elc+2-2k+\frac{1-e}{2}(kl-l)-(k-1)(l-1)\equiv 0\bmod 2p$.

Further, when $t=1$, the right hand side of (\ref{eq1}) is $1$.
$\Delta_K(1)=-emp+e\sum_{i=0}^{p-1}\#\{j\in I_k|[q(j+ki+c)]_p\in I_{ek_2}\}=-emp+ekk_2=1$.
Thus, (\ref{eq1}) lifts as the equality in ${\Bbb Z}[t^{\pm1}]/t^p-1$.

In the case of $2g(K)=p$, since $-em+e\cdot\#\{j\in I_k|[q_2(j+p/2+c)]_p\in I_{ek_2}\}=2$,
then we can use the same formula as (\ref{aiformula}).
\end{proof}

Let $(p,k)$ be relatively prime positive integers
and $a_i(p,k)$ the $i$-th coefficient of $\overline{\Delta_{T(k,l)}}(t)$.
Then we define the coefficient $\bar{a}_i(p,k)$ to be $\sum_{j\equiv i\bmod p}a_j(p,k)$.
The coefficients have the period $p$ namely, $\bar{a}_{i+p}(p,k)=\bar{a}_i(p,k)$.
We define {\it $A$-function} $A(x)$ and {\it $A$-matrix} $(A_{i,j})$ to be
$$A:{\Bbb Z}\to {\Bbb Z},\ \ A(x)=\bar{a}_{k_2(x-c)}(p,k)$$
and
$$A_{i,j}=\bar{a}_{k_2(i-c)+j}(p,k),\ \ (i,j)\in {\Bbb Z}^2$$
respectively.
Namely, $A(i+ekj)=A_{i,j}$ holds.
Further, we denote the difference $A(x)-A(x-1)$ by $dA(x)$, and $A_{i,j}-A_{i-1,j}=dA_{i,j}$.
We call these {\it $dA$-function} and {\it $dA$-matrix}.

We define {\it $A'$-function} and {\it $A'$-matrix} to be $A'(x)=\bar{a}_{k(x-c')}(p,k)$, where $c'=(k_2+e-p)(k_2-e)/2$
and $A_{i,j}'=\bar{a}_{ej+k(i-c')}(p,k)$.

\begin{lem}
\label{alternating}
Let $(p,k,k_2)$ be a surgery parameter and $q_2=k_2^2\bmod p$.
The difference $dA(x)$ is computed by 
\begin{equation}
\label{eq2}
dA(x)=E_{ek_2}(q_2x+ek_2)-E_{ek_2}(q_2x)=\begin{cases}1&[q_2x]_p\in I_{-k_2}\\-1&[q_2x]_p\in I_{k_2}\\0&\text{otherwise}\end{cases}
\end{equation}
and 
\begin{equation}
\label{alternatingAtype}
dA(x)=1\Leftrightarrow dA(x+ek)=-1.
\end{equation}
\end{lem}
\begin{proof}
By the definition of $A(x)$, we have
$A(x)=-em+e\#\{k\in I_k|[q_2(j+k(ek_2(x-c))+c)]_p\in I_{ek_2}\}=
-em+e\#\{j\in I_k|[q_2(j+x)]_p\in I_{ek_2}\}$ and 
\begin{eqnarray*}
dA(x)&=&A(x)-A(x-1)\\
&=&e\{j\in I_k|[q_2(j+x)]_p\in I_{ek_2}\}-e\{j\in I_k|[q_2(j+x-1)]_p\in I_{ek_2}\}\\
&=&E_{ek_2}(q_2x+ek_2)-E_{ek_2}(q_2x).
\end{eqnarray*}
Therefore we have (\ref{eq2}).
The (\ref{alternatingAtype}) follows from the equivalence relation: $dA(x)=1\Leftrightarrow [q_2x]_p\in I_{-k_2}\Leftrightarrow [q_2x+k_2]_p\in I_{k_2}\Leftrightarrow dA(x+ek)=-1$.
\end{proof}
In the same way as the case of $dA$, we have $dA'(x)=1\Leftrightarrow[qx]_p\in I_{-k}\Leftrightarrow [qx+k]_p\in I_{k}\Leftrightarrow dA'(x+ek_2)=-1$, hence,
$$dA'(x)=-1\Leftrightarrow dA'(x+ek_2)=1.$$
Here $q$ is an integer with $q\equiv k^2\bmod p$.
\label{surgeryslope}
%

\subsection{Alexander polynomial of $K_{p,k}$.}
Ichihara, Saito, and Teragaito gave the following formula of the Alexander polynomial of $K_{p,k}$ in $S^3$.
Here $k_2$ is the second dual class of $(p,k)$ and a symbol $[[\cdot]]_p$ presents the remainder dividing by $p$, where the values are reduced to an element in $\{1,\cdots, p\}$.
\begin{theorem}[\cite{IST}]
\label{ISTformula}
Let $q,q_2$ be integers with $q=k^2\bmod p$ and $q_2=k_2^2\bmod p$.
Then $\Delta_{K_{p,k}}(t)$ is computed by the following formula:
\begin{eqnarray}
\Delta_{K_{p,k}}(t)&\doteq&\frac{\sum_{i=0}^{k_2-1}t^{\Phi(i)\cdot p-[[qi]]_p\cdot k_2}}{\sum_{i=0}^{k_2-1}t^i}\label{ISTformulaitself}\\
&\doteq&\frac{\sum_{i=0}^{k-1}t^{\Phi'(i)\cdot p-[[q_2i]]_p\cdot k}}{\sum_{i=0}^{k-1}t^i},\label{ISTformulaitself2}
\end{eqnarray}
where $\Phi(i)=\#\{j\in I_{k_2-1}|[[qj]]_p<[[qi]]_p\}$ and $\Phi'(i)=\#\{j\in I_{k-1}|[[q_2j]]_p<[[q_2i]]_p\}$.
\end{theorem}
This formula works for the $K_{p,k}$ in arbitrary $Y_{p,k}$ as well as $S^3$.
In this paper we use this formula in many times to compute the Alexander polynomial and the genus of $K_{p,k}$.
We denote the $i$-th coefficient of the right hand side of (\ref{ISTformulaitself}) and (\ref{ISTformulaitself2}) by $b_i$ and $b_i'$ respectively.

We define {\it $B$-matrix} and {\it $B'$-matrix} to be 
$$B_{i,j}=b_{ik_2+j}, B'_{i,j}=b'_{ik+j}$$
for $(i,j)\in{\Bbb Z}^2$ respectively.
Similarly, we define the following differences to be $dB$-matrix and $dB'$-matrix.
$$dB_{i,j}:=B_{i,j}-B_{i-1,j}\text{ and }dB'_{i,j}:=B'_{i,j}-B'_{i-1,j}$$
respectively.
\begin{lem}
\label{dplem}
Let $q,q_2$ be integers with $q\equiv k^2\bmod p$ and $q_2\equiv k^2_2\bmod p$.
The $dB$-matrix and $dB'$-matrix are computed as follows:
$$dB_{i,j}=b_{ik_2+j}-b_{(i-1)k_2+j}=
\begin{cases}
1&\Phi(l-1)\cdot p=([[q(l-1)]]_p+i)\cdot k_2+j\text{ for some }l\in I_{k_2}\\
-1&\Phi(l-1)\cdot p=([[q(l-1)]]_p+i)\cdot k_2+j-1\text{ for some }l\in I_{k_2}\\
0&\text{otherwise.}
\end{cases}$$
$$dB'_{i,j}=b'_{ik+j}-b'_{(i-1)k+j}=
\begin{cases}
1&\Phi(l-1)\cdot p=([[q_2(l-1)]]_p+i)\cdot k+j\text{ for some }l\in I_{k}\\
-1&\Phi(l-1)\cdot p=([[q_2(l-1)]]_p+i)\cdot k+j-1\text{ for some }l\in I_{k}\\
0&\text{otherwise.}
\end{cases}$$
Thus $dB_{i,j}=-1\Leftrightarrow dB_{i,j-1}=1$ and $dB'_{i,j}=-1\Leftrightarrow dB'_{i,j-1}=1$ hold.
\end{lem}
\begin{proof}
Since we can do the same argument by exchanging $k$ and $k_2$, we prove the case of $dB$-matrix only.
Let $F(t)$ denote the right hand side of (\ref{ISTformulaitself}).
Then we have
\begin{eqnarray*}
(t^{k_2}-1)F(t)&=&\sum_{i}(b_{i-k_2}-b_i)t^i\\
&=&(t-1)\sum_{i=0}^{k_2-1}t^{\Phi(i)\cdot p-[[qi]]_p\cdot k_2}=\sum_{i=0}^{k_2-1}(t^{\Phi(i)\cdot p-[[qi]]_p\cdot k_2+1}-t^{\Phi(i)\cdot p-[[qi]]_p\cdot k_2})\\
&=&\sum_{i\in {\mathbb Z}}\left(\#\{l\in I_{k_2}|\Phi(l-1)\cdot p-[[q(l-1)]]_p\cdot k_2=i-1\}\right.\\
&&\left.-\#\{l\in I_{k_2}|\Phi(l-1)\cdot p-[[q(l-1)]]_p\cdot k_2=i\}\right)t^i.
\end{eqnarray*}
The number of the integers $l$ in $I_{k_2}$ satisfying $\Phi(l-1)\cdot p-[[q(l-1)]]_p\cdot k_2=i$ for a given $i$ is at most one, 
because taking modulo $p$ for the equation, we have $l\equiv 1-k_2i\bmod p$.
Hence the following holds.
\begin{eqnarray*}
b_{i}-b_{i-k_2}&=&\#\{l\in I_{k_2}|\Phi(l-1)\cdot p-[[q(l-1)]]_p\cdot k_2=i\}\\
&&-\#\{l\in I_{k_2}|\Phi(l-1)\cdot p-[[q(l-1)]]_p\cdot k_2=i-1\}\\
&=&\begin{cases}
1&\Phi(l-1)\cdot p-[[q(l-1)]]_p\cdot k_2=i\text{ for an integer }l\in I_{k_2-1}\\
-1&\Phi(l-1)\cdot p-[[q(l-1)]]_p\cdot k_2=i-1\text{ for an integer }l\in I_{k_2-1}\\
0&\text{otherwise.}\end{cases}
\end{eqnarray*}
From this formula, the last assertion follows easily.
\end{proof}
In the remaining part of this section we compute the non-zero sequence for small $p$.
For example, by the formula (\ref{ISTformulaitself}), we obtain the non-zero sequence of $K_{12,5}$ and $g(K_{12,5})=12$ as the table below.
On the other hand, by applying the pillowcase method in \cite{tan4} to the case of the parameter $(12,5)$, the homology sphere $Y_{12,5}$ is homeomorphic to $\Sigma(3,5,7)$.
Therefore we obtain the equality $\Sigma(3,5,7)_{12}(K_{12,5})=L(12,11)$.
Here we put the list of $K_{p,k}$ in a non-L-space homology sphere with $p\le23$.
\begin{table}[htbp]
\begin{center}
$\begin{array}{|c|c|c|c|c|}\hline
p&k&Y_{p,k}&g(K_{p,k})&NS_h\\\hline
10&3&\Sigma(2,3,7)&6&(6,5,3,2,0)\\\hline
12&5&\Sigma(3,5,7)&12&(12,11,7,6,5,4,2,1,0)\\\hline
13&5&\Sigma(3,5,8)&14&(14,13,9,8,6,5,4,3,1,0)\\\hline
15&4&\Sigma(3,4,11)&15&(15,14,11,10,7,6,4,2,0)\\\hline
16&7&\Sigma(4,7,9)&24&(24,23,17,16,15,14,10,9,8,7,6,5,3,2,1,0)\\\hline
17&3&\Sigma(2,3,11)&10&(10,9,7,6,4,2,1,0)\\\hline
17&4&\Sigma(3,4,13)&18&(18,17,14,13,10,9,6,4,2,0)\\\hline
17&5&\Sigma(2,5,7)&12&(12,11,7,6,5,4,2,1,0)\\\hline
19&3&\Sigma(2,3,13)&12&(12,11,9,8,6,5,3,2,0)\\\hline
20&9&\Sigma(5,9,11)&40&(40,39,31,30,29,28,22,21,20,19,18,17,\\
&&&&13,12,11,10,9,8,7,6,4,3,2,1,0)\\\hline
21&8&\Sigma(5,8,13)&42&(42,41,34,33,29,28,26,25,21,20,18,17,\\
&&&&16,15,13,12,10,9,8,7,5,4,3,1,0)\\\hline
23&5&\Sigma(2,5,9)&16&(16,15,11,10,7,5,2,0)\\\hline
23&7&\Sigma(2,3,11)&13&(13,12,10,9,6,5,3,2,0)\\\hline
\end{array}$
\caption{$K_{p,k}$ in non-L-space homology sphere $Y_{p,k}$ with $p\le 23$.}
\label{nonLspace23}
\end{center}
\end{table}

{\it Proof of Theorem~\ref{rmk1}.}
We make Table~\ref{nonLspace23} by computing the formula~(\ref{ISTformulaitself}).
These are half non-zero sequences of $K_{p,k}$ in non-L-space $\zhs$ with $p\le 23$.
The non-zero sequences can get the following equalities:
$$NS(K_{10,3})=NS(T(3,7)),\ NS(K_{12,5})=NS(K_{17,5})=NS(T(5,7))$$
$$NS(K_{13,5})=NS(T(5,8)),\ NS(K_{15,4})=NS(T(4,11))$$
$$NS(K_{16,7})=NS(T(7,9)),\ NS(K_{17,3})=NS(T(3,11))$$
$$NS(K_{19,3})=NS(T(3,13)),\ NS(K_{17,4})=NS(T(4,13))$$
$$NS(K_{20,9})=NS(T(9,11)),\ NS(K_{21,8})=NS(T(8,13))$$
$$NS(K_{23,5})=NS(T(5,9)).$$
These equalities mean the equalities of the corresponding Alexander polynomials.
Since any (even rational) Dehn surgeries of these torus knots do not produce any corresponding homology sphere $Y_{p,k}$,
those exteriors $Y_{p,k}- K_{p,k}$ and $S^3- T(r,s)$ are not homeomorphic each other.
$K_{23,7}$ lies in $\Sigma(2,3,11)$ and the polynomial $\Delta_{K_{23,7}}$ is not a cyclotomic polynomial and furthermore, 
it is not the Alexander polynomial of any lens space knot in an $\lz$.
In fact, suppose that the polynomial is a lens surgery polynomial in an $\lz$.
Since we have $\alpha(K_{23,7})=13$, hence we have $p\le 13^2+2\cdot 13+2=197$ by the estimate in Theorem~\ref{a22a2cor}.
We compute the formula (\ref{ISTformulaitself}) for $p\le 197$ and we can check that the polynomial is not any lens surgery polynomial in an $\lz$.
\qed
\section{Non-zero curve.}
\label{nonzero}
\subsection{Non-zero curve and non-zero region.}
In this section we define non-zero curves for a lens space knot.
The non-zero curves can visualize the places of all the non-zero coefficients in the $A$-matrix $(A_{i,j})$ and $B$-matrix $(B_{i,j})$.
Then the curves can explore the non-zero sequence.
The coefficients are computed by $dA_{i,j}$ or $dB_{i,j}$ in Lemma~\ref{alternating} and~\ref{dplem}.
\begin{lem}
\label{possible}
Let $K$ be a lens space knot in an $\lz$ or $K_{p,k}$ in $Y_{p,k}$.
We suppose that $X=A$ if $K$ is the former and $X=B$ if $K$ is the latter.
Let $(X_{i,j})$ be the $X$-matrix of the lens space surgery for a knot $K$ with $2g(K)\neq p$.
For the coordinate $(i,j)$ with $dA_{i,j}=-dA_{i,j+1}=1$ or $dB_{i,j}=-dB_{i,j+1}=1$.
The values of $X$-matrix around $(i,j)$ have one of the following local behaviors:
\begin{figure}[htbp]\begin{center}
{\unitlength 0.1in%
\begin{picture}( 40.3500,  5.7500)( 18.5500,-14.3500)%
\put(24.0000,-11.7000){\makebox(0,0){$-1$}}%
\put(20.0000,-9.8000){\makebox(0,0){$j+1$}}%
\put(24.0000,-15.0000){\makebox(0,0){$i-1$}}%
%
\special{pn 8}%
\special{pa 2230 1360}%
\special{pa 2230 860}%
\special{fp}%
\special{sh 1}%
\special{pa 2230 860}%
\special{pa 2210 927}%
\special{pa 2230 913}%
\special{pa 2250 927}%
\special{pa 2230 860}%
\special{fp}%
%
\special{pn 8}%
\special{pa 2190 1300}%
\special{pa 2990 1300}%
\special{fp}%
\special{sh 1}%
\special{pa 2990 1300}%
\special{pa 2923 1280}%
\special{pa 2937 1300}%
\special{pa 2923 1320}%
\special{pa 2990 1300}%
\special{fp}%
\put(28.0000,-11.7000){\makebox(0,0){$0$}}%
\put(24.0000,-9.7000){\makebox(0,0){$1$}}%
\put(28.0000,-9.7000){\makebox(0,0){$0$}}%
\put(36.0000,-11.6000){\makebox(0,0){$0$}}%
\put(40.0000,-9.5500){\makebox(0,0){$-1$}}%
\put(36.0000,-9.6000){\makebox(0,0){$0$}}%
\put(40.0000,-11.5500){\makebox(0,0){$1$}}%
%
\special{pn 8}%
\special{pa 3390 1300}%
\special{pa 4190 1300}%
\special{fp}%
\special{sh 1}%
\special{pa 4190 1300}%
\special{pa 4123 1280}%
\special{pa 4137 1300}%
\special{pa 4123 1320}%
\special{pa 4190 1300}%
\special{fp}%
%
\special{pn 8}%
\special{pa 3430 1360}%
\special{pa 3430 860}%
\special{fp}%
\special{sh 1}%
\special{pa 3430 860}%
\special{pa 3410 927}%
\special{pa 3430 913}%
\special{pa 3450 927}%
\special{pa 3430 860}%
\special{fp}%
\put(48.0000,-11.6500){\makebox(0,0){$0$}}%
\put(52.0000,-9.6500){\makebox(0,0){$0$}}%
\put(48.0000,-9.6500){\makebox(0,0){$1$}}%
\put(52.0000,-11.6500){\makebox(0,0){$1$}}%
%
\special{pn 8}%
\special{pa 4590 1300}%
\special{pa 5390 1300}%
\special{fp}%
\special{sh 1}%
\special{pa 5390 1300}%
\special{pa 5323 1280}%
\special{pa 5337 1300}%
\special{pa 5323 1320}%
\special{pa 5390 1300}%
\special{fp}%
%
\special{pn 8}%
\special{pa 4630 1360}%
\special{pa 4630 860}%
\special{fp}%
\special{sh 1}%
\special{pa 4630 860}%
\special{pa 4610 927}%
\special{pa 4630 913}%
\special{pa 4650 927}%
\special{pa 4630 860}%
\special{fp}%
\put(60.0000,-11.7000){\makebox(0,0){$-1$}}%
\put(64.0000,-11.7000){\makebox(0,0){$0$}}%
\put(60.0000,-9.7000){\makebox(0,0){$0$}}%
\put(64.0000,-9.7000){\makebox(0,0){$-1$}}%
%
\special{pn 8}%
\special{pa 5790 1300}%
\special{pa 6590 1300}%
\special{fp}%
\special{sh 1}%
\special{pa 6590 1300}%
\special{pa 6523 1280}%
\special{pa 6537 1300}%
\special{pa 6523 1320}%
\special{pa 6590 1300}%
\special{fp}%
%
\special{pn 8}%
\special{pa 5830 1360}%
\special{pa 5830 860}%
\special{fp}%
\special{sh 1}%
\special{pa 5830 860}%
\special{pa 5810 927}%
\special{pa 5830 913}%
\special{pa 5850 927}%
\special{pa 5830 860}%
\special{fp}%
%
\special{pn 8}%
\special{pa 2400 1330}%
\special{pa 2400 1270}%
\special{fp}%
\put(20.0000,-11.8000){\makebox(0,0){$j$}}%
\put(28.0000,-15.0000){\makebox(0,0){$i$}}%
%
\special{pn 8}%
\special{pa 2800 1330}%
\special{pa 2800 1270}%
\special{fp}%
%
\special{pn 8}%
\special{pa 2200 1200}%
\special{pa 2260 1200}%
\special{fp}%
%
\special{pn 8}%
\special{pa 2200 1000}%
\special{pa 2260 1000}%
\special{fp}%
\end{picture}}%
\\\end{center}\end{figure}
\end{lem}
\begin{proof}
The non-zero coefficients of $\Delta_{K}$ are $\pm1$ via Theorem~\ref{OS1}.
Thus the second statements in Lemma~\ref{alternating} and Lemma~\ref{dplem} tell us that
the possibilities of values $A_{i,j}$ or $B_{i,j}$ around $(i,j)\in {\mathbb Z}^2$ are four patterns above.
\end{proof}
In the same way, for the $A'$-matrix and $B'$-matrix the same conditions are satisfied.
Based on this lemma, we define the non-zero curve.
\begin{defn}[Non-zero curves]
\label{nzc}
Let $K$ be a lens space knot in an $\lz$ or $K_{p,k}$ and let $X$ be $A$ or $B$ respectively.
Then according to the process below, we describe oriented curves on ${\mathbb R}^2$.
Here we regard the $X$-matrix as a function on the lattice points ${\Bbb Z}^2$ in ${\Bbb R}^2$.
\begin{enumerate}
\item[(a)] Draw a horizontal oriented arrow on each lattice point $(i,j)$ with $X_{i,j}\neq 0$.
The direction is the right when $X_{i,j}=1$ and the left when $X_{i,j}=-1$ as below.
Draw nothing on any point $(i,j)$ with $X_{i,j}=0$.
We call the arrow with right direction or left direction {\it positive arrow} or {\it negative arrow} respectively.
\begin{center}\input{arrow.tex}\end{center}
\item[(b)] Connect the adjacent two arrows with the same direction.
Namely, if there exist two arrows on $(i,j)$ and $(i+1,j)$ with the same direction, then
we connect them as below:
\begin{center}
{\unitlength 0.1in%
\begin{picture}( 24.0000,  1.3200)( 10.0000, -8.2700)%
\put(12.0000,-7.7500){\makebox(0,0){$1$}}%
\put(28.0000,-7.6500){\makebox(0,0){$-1$}}%
\put(16.0000,-7.7500){\makebox(0,0){$1$}}%
%
\special{pn 8}%
\special{pa 1000 800}%
\special{pa 1800 800}%
\special{fp}%
\special{sh 1}%
\special{pa 1800 800}%
\special{pa 1733 780}%
\special{pa 1747 800}%
\special{pa 1733 820}%
\special{pa 1800 800}%
\special{fp}%
%
\special{pn 8}%
\special{pa 3400 800}%
\special{pa 2600 800}%
\special{fp}%
\special{sh 1}%
\special{pa 2600 800}%
\special{pa 2667 820}%
\special{pa 2653 800}%
\special{pa 2667 780}%
\special{pa 2600 800}%
\special{fp}%
\put(32.0000,-7.6000){\makebox(0,0){$-1$}}%
\end{picture}}%
\end{center}
\item[(c)] For any point $(i,j)$ satisfying $dA_{i,j}=-dA_{i,j+1}=1$ or $dB_{i,j}=-dB_{i,j+1}=1$, 
we connect the corresponding two non-empty arrows around $(i,j)\in {\mathbb Z}^2$ as in the figure below.
The four patterns are the four possibilities in Lemma~\ref{possible}:
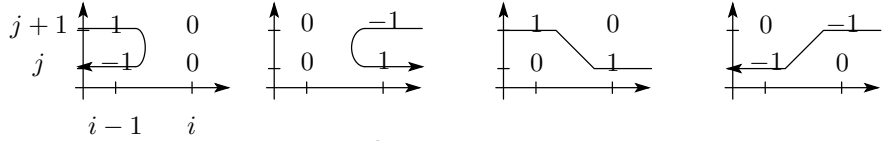
\begin{center}
{\unitlength 0.1in%
\begin{picture}( 45.5000,  5.7500)( 10.5000,-18.3500)%
\put(16.0000,-15.7000){\makebox(0,0){$-1$}}%
\put(26.0000,-15.6000){\makebox(0,0){$0$}}%
\put(20.0000,-15.7000){\makebox(0,0){$0$}}%
\put(30.0000,-15.5500){\makebox(0,0){$1$}}%
\put(16.0000,-13.7000){\makebox(0,0){$1$}}%
\put(26.0000,-13.6000){\makebox(0,0){$0$}}%
\put(20.0000,-13.7000){\makebox(0,0){$0$}}%
\put(30.0000,-13.5500){\makebox(0,0){$-1$}}%
%
\special{pn 8}%
\special{pa 1400 1390}%
\special{pa 1700 1390}%
\special{fp}%
%
\special{pn 8}%
\special{pa 1700 1590}%
\special{pa 1400 1590}%
\special{fp}%
\special{sh 1}%
\special{pa 1400 1590}%
\special{pa 1467 1610}%
\special{pa 1453 1590}%
\special{pa 1467 1570}%
\special{pa 1400 1590}%
\special{fp}%
%
\special{pn 8}%
\special{ar 1700 1490 50 100  4.7123890  1.5707963}%
%
\special{pn 8}%
\special{pa 3200 1390}%
\special{pa 2900 1390}%
\special{fp}%
%
\special{pn 8}%
\special{pa 2900 1590}%
\special{pa 3200 1590}%
\special{fp}%
\special{sh 1}%
\special{pa 3200 1590}%
\special{pa 3133 1570}%
\special{pa 3147 1590}%
\special{pa 3133 1610}%
\special{pa 3200 1590}%
\special{fp}%
%
\special{pn 8}%
\special{ar 2900 1490 70 100  1.5707963  4.7123890}%
\put(38.0000,-15.6500){\makebox(0,0){$0$}}%
\put(42.0000,-15.6500){\makebox(0,0){$1$}}%
\put(38.0000,-13.6500){\makebox(0,0){$1$}}%
\put(42.0000,-13.6500){\makebox(0,0){$0$}}%
\put(50.0000,-15.7000){\makebox(0,0){$-1$}}%
\put(54.0000,-15.7000){\makebox(0,0){$0$}}%
\put(50.0000,-13.7000){\makebox(0,0){$0$}}%
\put(54.0000,-13.7000){\makebox(0,0){$-1$}}%
%
\special{pn 8}%
\special{pa 3600 1400}%
\special{pa 3900 1400}%
\special{pa 4100 1600}%
\special{pa 4400 1600}%
\special{fp}%
%
\special{pn 8}%
\special{pa 4800 1600}%
\special{pa 5100 1600}%
\special{pa 5300 1400}%
\special{pa 5600 1400}%
\special{fp}%
%
\special{pn 8}%
\special{pa 4840 1600}%
\special{pa 4800 1600}%
\special{fp}%
\special{sh 1}%
\special{pa 4800 1600}%
\special{pa 4867 1620}%
\special{pa 4853 1600}%
\special{pa 4867 1580}%
\special{pa 4800 1600}%
\special{fp}%
\put(11.9500,-13.8000){\makebox(0,0){$j+1$}}%
\put(15.9500,-19.0000){\makebox(0,0){$i-1$}}%
%
\special{pn 8}%
\special{pa 1425 1760}%
\special{pa 1425 1260}%
\special{fp}%
\special{sh 1}%
\special{pa 1425 1260}%
\special{pa 1405 1327}%
\special{pa 1425 1313}%
\special{pa 1445 1327}%
\special{pa 1425 1260}%
\special{fp}%
%
\special{pn 8}%
\special{pa 1385 1700}%
\special{pa 2185 1700}%
\special{fp}%
\special{sh 1}%
\special{pa 2185 1700}%
\special{pa 2118 1680}%
\special{pa 2132 1700}%
\special{pa 2118 1720}%
\special{pa 2185 1700}%
\special{fp}%
%
\special{pn 8}%
\special{pa 1595 1730}%
\special{pa 1595 1670}%
\special{fp}%
\put(11.9500,-15.8000){\makebox(0,0){$j$}}%
\put(19.9500,-19.0000){\makebox(0,0){$i$}}%
%
\special{pn 8}%
\special{pa 1995 1730}%
\special{pa 1995 1670}%
\special{fp}%
%
\special{pn 8}%
\special{pa 1395 1600}%
\special{pa 1455 1600}%
\special{fp}%
%
\special{pn 8}%
\special{pa 1395 1400}%
\special{pa 1455 1400}%
\special{fp}%
%
\special{pn 8}%
\special{pa 2425 1760}%
\special{pa 2425 1260}%
\special{fp}%
\special{sh 1}%
\special{pa 2425 1260}%
\special{pa 2405 1327}%
\special{pa 2425 1313}%
\special{pa 2445 1327}%
\special{pa 2425 1260}%
\special{fp}%
%
\special{pn 8}%
\special{pa 2385 1700}%
\special{pa 3185 1700}%
\special{fp}%
\special{sh 1}%
\special{pa 3185 1700}%
\special{pa 3118 1680}%
\special{pa 3132 1700}%
\special{pa 3118 1720}%
\special{pa 3185 1700}%
\special{fp}%
%
\special{pn 8}%
\special{pa 2595 1730}%
\special{pa 2595 1670}%
\special{fp}%
%
\special{pn 8}%
\special{pa 2995 1730}%
\special{pa 2995 1670}%
\special{fp}%
%
\special{pn 8}%
\special{pa 2395 1600}%
\special{pa 2455 1600}%
\special{fp}%
%
\special{pn 8}%
\special{pa 2395 1400}%
\special{pa 2455 1400}%
\special{fp}%
%
\special{pn 8}%
\special{pa 3625 1760}%
\special{pa 3625 1260}%
\special{fp}%
\special{sh 1}%
\special{pa 3625 1260}%
\special{pa 3605 1327}%
\special{pa 3625 1313}%
\special{pa 3645 1327}%
\special{pa 3625 1260}%
\special{fp}%
%
\special{pn 8}%
\special{pa 3585 1700}%
\special{pa 4385 1700}%
\special{fp}%
\special{sh 1}%
\special{pa 4385 1700}%
\special{pa 4318 1680}%
\special{pa 4332 1700}%
\special{pa 4318 1720}%
\special{pa 4385 1700}%
\special{fp}%
%
\special{pn 8}%
\special{pa 3795 1730}%
\special{pa 3795 1670}%
\special{fp}%
%
\special{pn 8}%
\special{pa 4195 1730}%
\special{pa 4195 1670}%
\special{fp}%
%
\special{pn 8}%
\special{pa 3595 1600}%
\special{pa 3655 1600}%
\special{fp}%
%
\special{pn 8}%
\special{pa 3595 1400}%
\special{pa 3655 1400}%
\special{fp}%
%
\special{pn 8}%
\special{pa 4825 1760}%
\special{pa 4825 1260}%
\special{fp}%
\special{sh 1}%
\special{pa 4825 1260}%
\special{pa 4805 1327}%
\special{pa 4825 1313}%
\special{pa 4845 1327}%
\special{pa 4825 1260}%
\special{fp}%
%
\special{pn 8}%
\special{pa 4785 1700}%
\special{pa 5585 1700}%
\special{fp}%
\special{sh 1}%
\special{pa 5585 1700}%
\special{pa 5518 1680}%
\special{pa 5532 1700}%
\special{pa 5518 1720}%
\special{pa 5585 1700}%
\special{fp}%
%
\special{pn 8}%
\special{pa 4995 1730}%
\special{pa 4995 1670}%
\special{fp}%
%
\special{pn 8}%
\special{pa 5395 1730}%
\special{pa 5395 1670}%
\special{fp}%
%
\special{pn 8}%
\special{pa 4795 1600}%
\special{pa 4855 1600}%
\special{fp}%
%
\special{pn 8}%
\special{pa 4795 1400}%
\special{pa 4855 1400}%
\special{fp}%
\end{picture}}%
\end{center}
\end{enumerate}
Thus we can make oriented curves on ${\Bbb R}^2$.
We call the oriented curves non-zero curves.
\end{defn}
The curves are weakly-decreasing about $j$.
Notice that on no two points $(i,j)$ and $(i+1,j)$ opposite arrows are drawn, 
because values of $dA(x)$ are $0$ or $\pm1$.
\begin{lem}
\label{noendunbound}
Let $K$ be a lens space knot in an $\lz$ with $2g(K)<p$ or $K_{p,k}$ with the slope $p>1$.
Let $\Gamma$ be the non-zero curves of $K$.
Then, $\Gamma$ is submanifolds in ${\mathbb R}^2$,
namely, the manifolds are simple and no end points.
\end{lem}
\begin{proof}
We prove that any non-zero curve is simple and has no end points.
From the construction, the curve is simple since $|dA_{i,j}|\le 1$.
If the curve has an end at $(i,j)$, then $dA_{i,j}\neq 0$ or $dA_{i+1,j}\neq 0$ holds.
From Lemma~\ref{alternating}, in the former case, $dA_{i,j+1}=-dA_{i,j}\neq 0$ or $dA_{i,j-1}=-dA_{i,j}\neq 0$.
In the latter case, $dA_{i+1,j+1}=-dA_{i+1,j}\neq 0$ or $dA_{i+1,j}=-dA_{i+1,j}\neq 0$ holds.
Due to the definition of non-zero curve, these pairs of non-zero values of $dA$-function 
correspond to the $j$-level decreasing connection as in the process (c) (see Figure~\ref{co}).
This is contradiction.
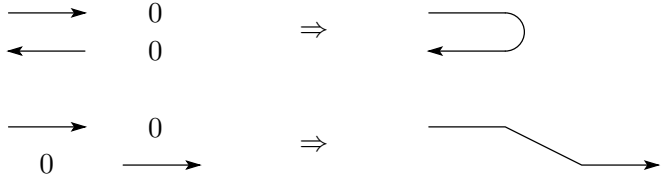
\begin{figure}[hbtp]
\begin{center}
{\unitlength 0.1in%
\begin{picture}(34.0000,8.6200)(8.0000,-13.9700)%
%
\special{pn 8}%
\special{pa 800 600}%
\special{pa 1200 600}%
\special{fp}%
\special{sh 1}%
\special{pa 1200 600}%
\special{pa 1133 580}%
\special{pa 1147 600}%
\special{pa 1133 620}%
\special{pa 1200 600}%
\special{fp}%
%
\special{pn 8}%
\special{pa 1200 800}%
\special{pa 800 800}%
\special{fp}%
\special{sh 1}%
\special{pa 800 800}%
\special{pa 867 820}%
\special{pa 853 800}%
\special{pa 867 780}%
\special{pa 800 800}%
\special{fp}%
\put(24.0000,-7.0000){\makebox(0,0){$\Rightarrow$}}%
%
\special{pn 8}%
\special{pa 800 1197}%
\special{pa 1200 1197}%
\special{fp}%
\special{sh 1}%
\special{pa 1200 1197}%
\special{pa 1133 1177}%
\special{pa 1147 1197}%
\special{pa 1133 1217}%
\special{pa 1200 1197}%
\special{fp}%
\special{pa 1400 1397}%
\special{pa 1800 1397}%
\special{fp}%
\special{sh 1}%
\special{pa 1800 1397}%
\special{pa 1733 1377}%
\special{pa 1747 1397}%
\special{pa 1733 1417}%
\special{pa 1800 1397}%
\special{fp}%
\put(24.0000,-12.9700){\makebox(0,0){$\Rightarrow$}}%
%
\special{pn 8}%
\special{pa 3000 1197}%
\special{pa 3400 1197}%
\special{pa 3800 1397}%
\special{pa 4200 1397}%
\special{fp}%
%
\special{pn 8}%
\special{pa 4160 1397}%
\special{pa 4200 1397}%
\special{fp}%
\special{sh 1}%
\special{pa 4200 1397}%
\special{pa 4133 1377}%
\special{pa 4147 1397}%
\special{pa 4133 1417}%
\special{pa 4200 1397}%
\special{fp}%
%
\special{pn 8}%
\special{pa 3000 600}%
\special{pa 3400 600}%
\special{fp}%
%
\special{pn 8}%
\special{pa 3400 800}%
\special{pa 3000 800}%
\special{fp}%
\special{sh 1}%
\special{pa 3000 800}%
\special{pa 3067 820}%
\special{pa 3053 800}%
\special{pa 3067 780}%
\special{pa 3000 800}%
\special{fp}%
%
\special{pn 8}%
\special{ar 3400 700 100 100 4.7123890 1.5707963}%
\put(15.7000,-6.0000){\makebox(0,0){0}}%
\put(15.7000,-8.0000){\makebox(0,0){0}}%
\put(15.7000,-12.0000){\makebox(0,0){0}}%
\put(10.0000,-13.9000){\makebox(0,0){0}}%
\end{picture}}%
\caption{Connection of the pair of non-zero values of $dA$-function.}
\label{co}
\end{center}
\end{figure}
\end{proof}

Consider an $A$ or $B$-matrix for a lens space knot in an $\lz$ or $K_{p,k}$ in $Y_{p,k}$ respectively.
The symmetry of Alexander polynomial tells us the following:
\begin{prop}
\label{symmetry}
A non-zero curve for $A$- or $B$-matrix has a $180^{\circ}$-rotation symmetry about a point in ${\Bbb R}^2$.
\end{prop}
\begin{proof}
Since we have 
\begin{eqnarray*}
A_{i,j+ck_2}&=&\bar{a}_{j+ck_2+k_2(i-c)}=\bar{a}_{-(j+ck_2)-k_2(i-c)}\\
&=&\bar{a}_{(-j+ck_2)+k_2(-i-c)}=A_{-i,-j+ck_2},
\end{eqnarray*}
$180^\circ$-rotation of $A$-matrix with the center $(0,ck_2)$ gives the same matrix.
$B$-matrix has also a similar rotation symmetry.
The proof is skipped.
\end{proof}
In the end, we can get some symmetric unbounded simple curves with orientation and no end points on ${\Bbb R}^2$.

\subsection{The $2g(K)=p$ case.}
Consider the case where $K$ is a lens space knot in an $\lz$ with $2g(K)=p$.
Even in this case, Lemma~\ref{possible}	holds.
Consider the non-constant local behavior of the $A$-function.
We need consider the points with $\bar{a}_{g}(K)=2$.
\begin{lem}
\label{possible2}
Let $A_{i,j}$ be the $A$-matrix of a lens space knot $K$ in an $\lz$ with parameter $(p,k)$ and $2g(K)=p$.
Suppose that $A_{i_0,j_0}=2$ for $(i_0,j_0)\in {\mathbb Z}^2$.
Then the values of $A$-matrix around $(i_0,j_0)$ have the following local behaviors:
\begin{figure}[htbp]\begin{center}
{\unitlength 0.1in%
\begin{picture}( 16.9500,  7.2700)( 11.7000,-16.2700)%
\put(22.0500,-13.6200){\makebox(0,0){$-1$}}%
\put(13.5500,-11.7200){\makebox(0,0){$j_0$}}%
\put(21.9500,-16.9200){\makebox(0,0){$i_0$}}%
%
\special{pn 8}%
\special{pa 1625 1552}%
\special{pa 1625 1052}%
\special{fp}%
\special{sh 1}%
\special{pa 1625 1052}%
\special{pa 1605 1119}%
\special{pa 1625 1105}%
\special{pa 1645 1119}%
\special{pa 1625 1052}%
\special{fp}%
%
\special{pn 8}%
\special{pa 1585 1492}%
\special{pa 2865 1492}%
\special{fp}%
\special{sh 1}%
\special{pa 2865 1492}%
\special{pa 2798 1472}%
\special{pa 2812 1492}%
\special{pa 2798 1512}%
\special{pa 2865 1492}%
\special{fp}%
\put(25.9500,-13.6200){\makebox(0,0){$0$}}%
\put(21.9500,-11.6200){\makebox(0,0){$2$}}%
\put(25.9500,-11.6200){\makebox(0,0){$1$}}%
\put(17.9700,-11.6500){\makebox(0,0){$1$}}%
\put(21.9700,-9.6500){\makebox(0,0){$-1$}}%
\put(17.9700,-9.6500){\makebox(0,0){$0$}}%
%
\special{pn 8}%
\special{pa 2195 1522}%
\special{pa 2195 1462}%
\special{fp}%
\put(13.5500,-13.7200){\makebox(0,0){$j_0-1$}}%
\put(25.9500,-16.9200){\makebox(0,0){$i_0+1$}}%
%
\special{pn 8}%
\special{pa 2595 1522}%
\special{pa 2595 1462}%
\special{fp}%
%
\special{pn 8}%
\special{pa 1595 1392}%
\special{pa 1655 1392}%
\special{fp}%
%
\special{pn 8}%
\special{pa 1595 1192}%
\special{pa 1655 1192}%
\special{fp}%
\put(17.9200,-16.9200){\makebox(0,0){$i_0-1$}}%
%
\special{pn 8}%
\special{pa 1792 1522}%
\special{pa 1792 1462}%
\special{fp}%
\end{picture}}%
\caption{The local behavior around $(i_0,j_0)$.}\label{4possible4}\end{center}\end{figure}
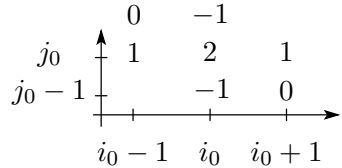
\end{lem}
\begin{proof}
Suppose that $A_{i_0,j_0}=2$.
Then $j_0+k_2(i_0-c)\equiv g\bmod p$.
If $A_{i_0+1,j_0},A_{i_0-1,j_0}\neq 2$, because if $A_{i_0+1,j_0}=2$, we have $j_0+k_2(i_0+1-c)=g\bmod p$ and  this means $p=1$.
Thus we have $A_{i_0+1,j_0}=1$ and $A_{i_0-1,j_0}=1$ due to $|dA_{i_0,j_0}|\le 1$ and symmetry of the Alexander polynomial. 
If $A_{i_0,j_0+1}=0$ then $A_{i_0-1,j_0+1}=1$.
The coefficients $A_{i_0-1,j_0}=A_{i_0-1,j_0+1}=1$ are inconsistent with the alternating condition in (\ref{OScond}).
See Figure~\ref{for}.
Thus $A_{i_0,j_0\pm1}=-1$ holds.
Here we use the symmetry of $A$-matrix.
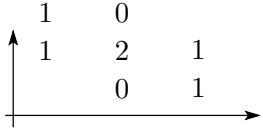
\begin{figure}[htbp]\begin{center}
\begin{center}
{\unitlength 0.1in%
\begin{picture}( 13.2500,  6.6200)( 16.9000,-18.5200)%
\put(23.0000,-16.5200){\makebox(0,0){$0$}}%
%
\special{pn 8}%
\special{pa 1730 1852}%
\special{pa 1730 1352}%
\special{fp}%
\special{sh 1}%
\special{pa 1730 1352}%
\special{pa 1710 1419}%
\special{pa 1730 1405}%
\special{pa 1750 1419}%
\special{pa 1730 1352}%
\special{fp}%
%
\special{pn 8}%
\special{pa 1690 1792}%
\special{pa 3015 1792}%
\special{fp}%
\special{sh 1}%
\special{pa 3015 1792}%
\special{pa 2948 1772}%
\special{pa 2962 1792}%
\special{pa 2948 1812}%
\special{pa 3015 1792}%
\special{fp}%
\put(27.0000,-16.4700){\makebox(0,0){$1$}}%
\put(23.0000,-14.5200){\makebox(0,0){$2$}}%
\put(27.0000,-14.4700){\makebox(0,0){$1$}}%
\put(19.0000,-14.5500){\makebox(0,0){$1$}}%
\put(19.0000,-12.5500){\makebox(0,0){$1$}}%
\put(23.0000,-12.5500){\makebox(0,0){$0$}}%
\end{picture}}%
\caption{The case of $A_{i_0,j_0\pm 1}=0$.}\label{for}\end{center}
\end{center}\end{figure}
\end{proof}
\begin{defn}[Non-zero curve of a lens space knot in an $\lz$ with $2g(K)=p$.]
Let $A_{i,j}$ be $A$-matrix of a lens space knot $K$ in an $\lz$ with parameter $(p,k)$ and with $2g(K)=p$.
Doing processes (a'), (c') in addition to (a), (b) and (c) in Definition~\ref{nzc}, we obtain non-zero curves.
Here, we state (a') and (c') as below.
\begin{enumerate}
\item[(a')] Draw a horizontal double arrow at the values $2$ as follows:
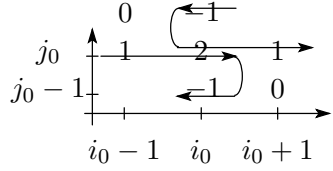
\begin{center}
{\unitlength 0.1in%
\begin{picture}(  3.2000,  1.4000)( 18.8000,-12.6400)%
\put(20.0000,-12.0000){\makebox(0,0){$2$}}%
%
\special{pn 8}%
\special{pa 1880 1244}%
\special{pa 2200 1244}%
\special{fp}%
\special{sh 1}%
\special{pa 2200 1244}%
\special{pa 2133 1224}%
\special{pa 2147 1244}%
\special{pa 2133 1264}%
\special{pa 2200 1244}%
\special{fp}%
%
\special{pn 8}%
\special{pa 1880 1164}%
\special{pa 2200 1164}%
\special{fp}%
\special{sh 1}%
\special{pa 2200 1164}%
\special{pa 2133 1144}%
\special{pa 2147 1164}%
\special{pa 2133 1184}%
\special{pa 2200 1164}%
\special{fp}%
\end{picture}}%
\end{center}
\item[(c')] Connect the arrows around $(i,j)$ with $A_{i_0,j_0}=2$ as below.
\begin{figure}[htbp]\begin{center}
{\unitlength 0.1in%
\begin{picture}( 16.5500,  7.2700)( 16.1000,-16.3500)%
\put(26.0500,-13.7000){\makebox(0,0){$-1$}}%
\put(17.9500,-11.8000){\makebox(0,0){$j_0$}}%
\put(25.9500,-17.0000){\makebox(0,0){$i_0$}}%
%
\special{pn 8}%
\special{pa 2025 1560}%
\special{pa 2025 1060}%
\special{fp}%
\special{sh 1}%
\special{pa 2025 1060}%
\special{pa 2005 1127}%
\special{pa 2025 1113}%
\special{pa 2045 1127}%
\special{pa 2025 1060}%
\special{fp}%
%
\special{pn 8}%
\special{pa 1985 1500}%
\special{pa 3265 1500}%
\special{fp}%
\special{sh 1}%
\special{pa 3265 1500}%
\special{pa 3198 1480}%
\special{pa 3212 1500}%
\special{pa 3198 1520}%
\special{pa 3265 1500}%
\special{fp}%
\put(29.9500,-13.7000){\makebox(0,0){$0$}}%
\put(25.9500,-11.7000){\makebox(0,0){$2$}}%
\put(29.9500,-11.7000){\makebox(0,0){$1$}}%
\put(21.9700,-11.7300){\makebox(0,0){$1$}}%
\put(25.9700,-9.7300){\makebox(0,0){$-1$}}%
\put(21.9700,-9.7300){\makebox(0,0){$0$}}%
%
\special{pn 8}%
\special{pa 2595 1530}%
\special{pa 2595 1470}%
\special{fp}%
\put(17.9500,-13.8000){\makebox(0,0){$j_0-1$}}%
\put(29.9500,-17.0000){\makebox(0,0){$i_0+1$}}%
%
\special{pn 8}%
\special{pa 2995 1530}%
\special{pa 2995 1470}%
\special{fp}%
%
\special{pn 8}%
\special{pa 1995 1400}%
\special{pa 2055 1400}%
\special{fp}%
%
\special{pn 8}%
\special{pa 1995 1200}%
\special{pa 2055 1200}%
\special{fp}%
%
\special{pn 8}%
\special{pa 2470 1155}%
\special{pa 3175 1155}%
\special{fp}%
\special{sh 1}%
\special{pa 3175 1155}%
\special{pa 3108 1135}%
\special{pa 3122 1155}%
\special{pa 3108 1175}%
\special{pa 3175 1155}%
\special{fp}%
%
\special{pn 8}%
\special{pa 2770 1410}%
\special{pa 2470 1410}%
\special{fp}%
\special{sh 1}%
\special{pa 2470 1410}%
\special{pa 2537 1430}%
\special{pa 2523 1410}%
\special{pa 2537 1390}%
\special{pa 2470 1410}%
\special{fp}%
%
\special{pn 8}%
\special{ar 2770 1305 40 100  4.7123890  1.5707963}%
%
\special{pn 8}%
\special{pa 2070 1200}%
\special{pa 2775 1200}%
\special{fp}%
\special{sh 1}%
\special{pa 2775 1200}%
\special{pa 2708 1180}%
\special{pa 2722 1200}%
\special{pa 2708 1220}%
\special{pa 2775 1200}%
\special{fp}%
%
\special{pn 8}%
\special{pa 2777 949}%
\special{pa 2477 949}%
\special{fp}%
\special{sh 1}%
\special{pa 2477 949}%
\special{pa 2544 969}%
\special{pa 2530 949}%
\special{pa 2544 929}%
\special{pa 2477 949}%
\special{fp}%
%
\special{pn 8}%
\special{ar 2475 1050 40 100  1.5707963  4.7123890}%
\put(21.9200,-17.0000){\makebox(0,0){$i_0-1$}}%
%
\special{pn 8}%
\special{pa 2192 1530}%
\special{pa 2192 1470}%
\special{fp}%
\end{picture}}%
\\\end{center}\end{figure}
\end{enumerate}
\end{defn}
Notice that the length of any double arrow is at most one.
\subsection{Regions containing non-zero curves.}
In the following, we investigate a domain on ${\mathbb R}^2$ in which the arrow lies.
Here we define {\it a closed $\epsilon$-neighborhood} of $(i',j')$ to be $\{(x,y)\in {\Bbb R}^2|\ \max\{|x-i'|,|y-j'|\}\le \epsilon\}$.
\begin{defn}[Non-zero regions: $\mathcal{N}^{A,m},\mathcal{N}^B$]
\label{Alexander}
Let $K$ be a lens space knot in an $\lz$ with the parameter $(p,k)$ the $d=\deg(\Delta_K(t))$.
For $j_0=-ck_2+mp$ we denote by ${\mathcal N}_0^{A,m}$ the union of closed $\frac{1}{2}$-neighborhood of points
\begin{equation}
\label{orireg}
\{(0,j_0-d),(0,j_0-d+1),\cdots,(0,j_0+d)\}.
\end{equation}

Let $K$ be $K_{p,k}$ in $Y_{p,k}$ with the parameter $(p,k)$ with $d_\text{top}$ and $d_{\text{bottom}}$ the top and bottom degree respectively of the polynomial of right hand side of (\ref{ISTformulaitself}).
We denote by ${\mathcal N}_0^B$ the union of closed $\frac{1}{2}$-neighborhood of points of 
$$\{(0,d_{\text{bottom}}),(0,d_{\text{bottom}}+1),\cdots,(0,d_{\text{top}})\}.$$

In the former case, let ${\mathcal N}_l^{A,m}$ denote the parallel transform $\{(x+l,y-k_2l)\in{\Bbb R}^2|(x,y)\in {\mathcal N}_0^{A,m}\}$ by $(l,-k_2l)$ for $l\in {\mathbb Z}$.
In the latter case, let ${\mathcal N}_l^B$ denote $\{(x+l,y-k_2l)\in{\Bbb R}^2|(x,y)\in {\mathcal N}_0^B\}$.
We call the unions
$${\mathcal N}^{A,m}=\displaystyle{\cup_{l\in {\Bbb Z}}{\mathcal N}_l^{A,m}},\  {\mathcal N}^B=\displaystyle{\cup_{l\in {\Bbb Z}}{\mathcal N}_l^B}$$
a non-zero region.
\end{defn}
Note that $\mathcal{N}^{A,m+1}$ is the parallel transform of $\mathcal{N}^{A,m}$ by $(0,p)$.
In the same way we can define $\mathcal{N}^{A',m}$ and $\mathcal{N}^{B'}$.
\begin{lem}
\label{contained}
Let $\gamma$ be one component non-zero curve of a lens space knot in an $\lz$ or $K_{p,k}$ in $Y_{p,k}$.
Then $\gamma$ lies in a non-zero region.
\end{lem}
\begin{proof}
By the definition of non-zero region any non-zero coefficients of $K_{p,k}$ is in a non-zero region.

Let $\mathcal{N}^{A,m}$ be a non-zero region of a lens space knot $K$ in an $\lz$ with parameter $(p,k)$.\\
({\bf Case 1.}) 
Suppose that a non-zero curve is passing from $\mathcal{N}^{A,m-1}$ to $\mathcal{N}^{A,m}$ as in Figure~\ref{transverse}.
Let $\gamma$ be a segment in the curve which is connecting two 1s in the both side.
Let $l$ be the $i$-coordinate of the left $1$ on $\gamma$.
Let $s$ be the union of the right boundary segment of $\mathcal{N}^{A,m-1}_l$ and left boundary segment of $\mathcal{N}^{A,m}_{l+1}$.

Consider the nearest and upper $-1$ (with respect to $j$-coordinate) to the right $1$ on $\gamma$.
Suppose that the $j$-coordinate of the $-1$ is higher than the bottom coefficient $a_g=1$ in $\mathcal{N}^{A,m}_{l}$.
Then the curve is passing from $\mathcal{N}^{A,m}$ to $\mathcal{N}^{A,m-1}$ (see the first picture in Figure~\ref{transverse}).
Because the $-1$ is the next non-zero coefficient to the $\gamma$.
The non-zero curve inlcuding the bottom coefficient in $\mathcal{N}^{A,m}_{l}$ has to have an end point.
Because the non-zero curve is monotone-decreasing (in the wider sense), the previous point of the bottom coefficient
does not exist.
This contradicts to Lemma~\ref{noendunbound}.
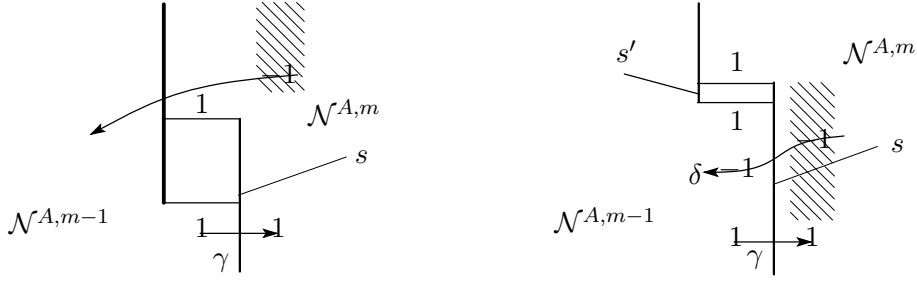
\begin{figure}[thpb]
\begin{center}
{\unitlength 0.1in%
\begin{picture}( 45.6500, 14.2500)( 17.5000,-19.5500)%
\put(36.9500,-11.7000){\makebox(0,0)[rb]{$\mathcal{N}^{A,m}$}}%
%
\special{pn 13}%
\special{pa 2958 1940}%
\special{pa 2958 1140}%
\special{fp}%
%
\special{pn 20}%
\special{pa 2560 1580}%
\special{pa 2560 540}%
\special{fp}%
%
\special{pn 8}%
\special{pa 2958 1140}%
\special{pa 2558 1140}%
\special{fp}%
%
\special{pn 8}%
\special{pa 2750 1740}%
\special{pa 3150 1740}%
\special{fp}%
\special{sh 1}%
\special{pa 3150 1740}%
\special{pa 3083 1720}%
\special{pa 3097 1740}%
\special{pa 3083 1760}%
\special{pa 3150 1740}%
\special{fp}%
\put(27.6000,-17.1000){\makebox(0,0){$1$}}%
\put(31.6000,-17.1000){\makebox(0,0){$1$}}%
\put(28.6000,-18.9000){\makebox(0,0){$\gamma$}}%
%
\special{pn 4}%
\special{pa 2950 1540}%
\special{pa 3515 1340}%
\special{fp}%
\put(35.9500,-13.4000){\makebox(0,0){$s$}}%
\put(31.6000,-9.1000){\makebox(0,0){$-1$}}%
%
\special{pn 8}%
\special{pa 2958 1580}%
\special{pa 2558 1580}%
\special{fp}%
%
\special{pn 8}%
\special{pa 2180 1220}%
\special{pa 2208 1205}%
\special{pa 2237 1189}%
\special{pa 2265 1174}%
\special{pa 2294 1159}%
\special{pa 2322 1144}%
\special{pa 2351 1129}%
\special{pa 2379 1114}%
\special{pa 2408 1100}%
\special{pa 2437 1087}%
\special{pa 2466 1073}%
\special{pa 2495 1060}%
\special{pa 2525 1048}%
\special{pa 2554 1037}%
\special{pa 2614 1015}%
\special{pa 2644 1006}%
\special{pa 2675 997}%
\special{pa 2705 988}%
\special{pa 2798 967}%
\special{pa 2829 961}%
\special{pa 2861 955}%
\special{pa 2892 950}%
\special{pa 2924 945}%
\special{pa 3052 929}%
\special{pa 3116 923}%
\special{pa 3149 920}%
\special{pa 3213 914}%
\special{pa 3246 911}%
\special{pa 3260 910}%
\special{fp}%
%
\special{pn 8}%
\special{pa 2208 1205}%
\special{pa 2180 1220}%
\special{fp}%
\special{sh 1}%
\special{pa 2180 1220}%
\special{pa 2248 1206}%
\special{pa 2227 1195}%
\special{pa 2229 1171}%
\special{pa 2180 1220}%
\special{fp}%
\put(27.6000,-10.6000){\makebox(0,0){$1$}}%
\put(51.3500,-17.3500){\makebox(0,0)[rb]{${\mathcal N}^{A,m-1}$}}%
\put(65.0000,-8.5000){\makebox(0,0)[rb]{${\mathcal N}^{A,m}$}}%
%
\special{pn 13}%
\special{pa 5753 1955}%
\special{pa 5753 950}%
\special{fp}%
%
\special{pn 13}%
\special{pa 5360 1050}%
\special{pa 5360 535}%
\special{fp}%
%
\special{pn 8}%
\special{pa 5753 955}%
\special{pa 5353 955}%
\special{fp}%
%
\special{pn 8}%
\special{pa 5545 1785}%
\special{pa 5945 1785}%
\special{fp}%
\special{sh 1}%
\special{pa 5945 1785}%
\special{pa 5878 1765}%
\special{pa 5892 1785}%
\special{pa 5878 1805}%
\special{pa 5945 1785}%
\special{fp}%
\put(55.5500,-17.5500){\makebox(0,0){$1$}}%
\put(59.5500,-17.5500){\makebox(0,0){$1$}}%
\put(56.5500,-18.7500){\makebox(0,0){$\gamma$}}%
%
\special{pn 4}%
\special{pa 5760 1480}%
\special{pa 6295 1285}%
\special{fp}%
\put(64.0000,-12.8000){\makebox(0,0){$s$}}%
%
\special{pn 8}%
\special{pa 5753 1055}%
\special{pa 5353 1055}%
\special{fp}%
%
\special{pn 8}%
\special{pa 5360 1010}%
\special{pa 4970 905}%
\special{fp}%
\put(49.8000,-7.8500){\makebox(0,0){$s'$}}%
\put(55.6000,-11.5000){\makebox(0,0){$1$}}%
\put(59.6000,-12.5000){\makebox(0,0){$-1$}}%
\put(55.6000,-14.0000){\makebox(0,0){$-1$}}%
\put(17.5000,-17.4000){\makebox(0,0)[lb]{${\mathcal N}^{A,m-1}$}}%
\put(55.6000,-8.4000){\makebox(0,0){$1$}}%
%
\special{pn 8}%
\special{pa 5390 1420}%
\special{pa 5457 1420}%
\special{pa 5490 1419}%
\special{pa 5523 1417}%
\special{pa 5555 1415}%
\special{pa 5587 1411}%
\special{pa 5618 1405}%
\special{pa 5649 1398}%
\special{pa 5679 1388}%
\special{pa 5708 1376}%
\special{pa 5736 1361}%
\special{pa 5763 1345}%
\special{pa 5817 1309}%
\special{pa 5844 1293}%
\special{pa 5873 1279}%
\special{pa 5902 1268}%
\special{pa 5933 1258}%
\special{pa 5964 1251}%
\special{pa 5995 1245}%
\special{pa 6027 1240}%
\special{pa 6060 1236}%
\special{pa 6092 1233}%
\special{pa 6120 1230}%
\special{fp}%
%
\special{pn 8}%
\special{pa 5423 1420}%
\special{pa 5390 1420}%
\special{fp}%
\special{sh 1}%
\special{pa 5390 1420}%
\special{pa 5457 1440}%
\special{pa 5443 1420}%
\special{pa 5457 1400}%
\special{pa 5390 1420}%
\special{fp}%
%
\special{pn 4}%
\special{pa 3045 560}%
\special{pa 3295 810}%
\special{fp}%
\special{pa 3045 620}%
\special{pa 3295 870}%
\special{fp}%
\special{pa 3045 680}%
\special{pa 3285 920}%
\special{fp}%
\special{pa 3045 740}%
\special{pa 3225 920}%
\special{fp}%
\special{pa 3045 800}%
\special{pa 3165 920}%
\special{fp}%
\special{pa 3045 860}%
\special{pa 3105 920}%
\special{fp}%
\special{pa 3075 530}%
\special{pa 3295 750}%
\special{fp}%
\special{pa 3135 530}%
\special{pa 3295 690}%
\special{fp}%
\special{pa 3195 530}%
\special{pa 3295 630}%
\special{fp}%
\special{pa 3255 530}%
\special{pa 3295 570}%
\special{fp}%
%
\special{pn 4}%
\special{pa 3165 920}%
\special{pa 3245 1000}%
\special{fp}%
\special{pa 3245 940}%
\special{pa 3295 990}%
\special{fp}%
\special{pa 3105 920}%
\special{pa 3185 1000}%
\special{fp}%
\special{pa 3055 930}%
\special{pa 3125 1000}%
\special{fp}%
%
\special{pn 4}%
\special{pa 5840 970}%
\special{pa 6060 1190}%
\special{fp}%
\special{pa 5840 1030}%
\special{pa 6050 1240}%
\special{fp}%
\special{pa 5840 1090}%
\special{pa 5990 1240}%
\special{fp}%
\special{pa 5840 1150}%
\special{pa 5940 1250}%
\special{fp}%
\special{pa 5840 1210}%
\special{pa 5900 1270}%
\special{fp}%
\special{pa 5880 950}%
\special{pa 6070 1140}%
\special{fp}%
\special{pa 5940 950}%
\special{pa 6070 1080}%
\special{fp}%
\special{pa 6000 950}%
\special{pa 6070 1020}%
\special{fp}%
%
\special{pn 4}%
\special{pa 5860 1290}%
\special{pa 5970 1400}%
\special{fp}%
\special{pa 5840 1330}%
\special{pa 5930 1420}%
\special{fp}%
\special{pa 5840 1390}%
\special{pa 5880 1430}%
\special{fp}%
\special{pa 5900 1270}%
\special{pa 6010 1380}%
\special{fp}%
\special{pa 5950 1260}%
\special{pa 6060 1370}%
\special{fp}%
\special{pa 6000 1250}%
\special{pa 6070 1320}%
\special{fp}%
%
\special{pn 4}%
\special{pa 5840 1510}%
\special{pa 6000 1670}%
\special{fp}%
\special{pa 5850 1460}%
\special{pa 6060 1670}%
\special{fp}%
\special{pa 5890 1440}%
\special{pa 6070 1620}%
\special{fp}%
\special{pa 5930 1420}%
\special{pa 6070 1560}%
\special{fp}%
\special{pa 5970 1400}%
\special{pa 6070 1500}%
\special{fp}%
\special{pa 6020 1390}%
\special{pa 6070 1440}%
\special{fp}%
\special{pa 5840 1570}%
\special{pa 5940 1670}%
\special{fp}%
\special{pa 5840 1630}%
\special{pa 5880 1670}%
\special{fp}%
\put(53.8000,-14.8000){\makebox(0,0)[rb]{$\delta$}}%
\end{picture}}%
\end{center}
\caption{A non-zero curves passing two non-zero regions.}
\label{transverse}
\end{figure}

Thus, the $j$-coodinate of the next $-1$ of the right $1$ on $\gamma$ is lower than the bottom coefficient of $\mathcal{N}^{A,m}_{l}$.
We define the curve including the next $-1$ to be $\delta$ (see the second picture in Figure~\ref{transverse}).
We assume that the segment $\gamma$ meets at the highest point on $s$ 
among segments positively-passing from $\mathcal{N}^{A,m-1}$ to $\mathcal{N}^{A,m}$.
Such a segment $\delta$ passing from $\mathcal{N}^{A,m}$ to $\mathcal{N}^{A,m-1}$ whose $j$-coordinate is lower than the top coefficient 
of $\mathcal{N}^{A,m}_{l}$, is unique.
Removing the segment $\delta$ from the plane, we have the odd non-zero coefficinents in $\mathcal{N}^{A,m-1}_{l+1}$ with the $j$-coordinate upper than $\gamma$.
The curves on the remained coefficients are not passing between $\mathcal{N}^{A,m-1}$ and $\mathcal{N}^{A,m}$ from the previous condition.
On the other hand, such a curve must turn in the left side of $s$.
To turn curves we need the even coefficients because the curves are no end points.
This is a contradiction.
Therefore, this case does not occur.
As a result, there exists no such $\gamma$ which is passing from $\mathcal{N}^{A,m-1}$ to $\mathcal{N}^{A,m}$.\\
%
%
({\bf Case 2.}) Suppose that a non-zero negative arrow $\gamma$ is passing from $\mathcal{N}^{A,m}$ to $\mathcal{N}^{A,m-1}$ in the negative direction.
Let $s$ be the same vertical segment as above, which $\gamma$ is passing $s$ as Figure~\ref{transverse2}.
We set the $i$-coordinate of the left of $\gamma$ to $s$ to be $l$.
The non-zero coefficients upper than $\gamma$ in $\mathcal{N}^{A,m}_l$ has odd terms.
Then there exists a curve $\gamma'$ passing from $\mathcal{N}^{A,m-1}$ to $\mathcal{N}^{A,m}$.
See Figure~\ref{transverse2}.
This reduces to Case~1.

Therefore, any non-zero curve is included in a non-zero region.
%
%
\begin{figure}[thpb]
\begin{center}
{\unitlength 0.1in%
\begin{picture}( 15.2500, 14.5000)( 17.9000,-24.3000)%
\put(17.9000,-17.7000){\makebox(0,0)[lb]{$\mathcal{N}^{A,m-1}$}}%
\put(34.3500,-12.9500){\makebox(0,0)[rb]{$\mathcal{N}^{A,m}$}}%
%
\special{pn 13}%
\special{pa 2810 2430}%
\special{pa 2810 1405}%
\special{fp}%
%
\special{pn 13}%
\special{pa 2410 1495}%
\special{pa 2410 980}%
\special{fp}%
%
\special{pn 8}%
\special{pa 2803 1395}%
\special{pa 2403 1395}%
\special{fp}%
\put(26.0500,-22.1500){\makebox(0,0){$-1$}}%
\put(30.0500,-22.1500){\makebox(0,0){$-1$}}%
\put(27.1000,-23.4000){\makebox(0,0){$\gamma$}}%
%
\special{pn 4}%
\special{pa 2795 2045}%
\special{pa 3295 2170}%
\special{fp}%
\put(34.0000,-21.9000){\makebox(0,0){$s$}}%
%
\special{pn 8}%
\special{pa 2803 1495}%
\special{pa 2403 1495}%
\special{fp}%
%
\special{pn 8}%
\special{pa 2590 1730}%
\special{pa 2970 1730}%
\special{fp}%
\special{sh 1}%
\special{pa 2970 1730}%
\special{pa 2903 1710}%
\special{pa 2917 1730}%
\special{pa 2903 1750}%
\special{pa 2970 1730}%
\special{fp}%
\put(26.0500,-17.4500){\makebox(0,0){$1$}}%
\put(30.0500,-17.4500){\makebox(0,0){$1$}}%
%
\special{pn 8}%
\special{pa 3060 2230}%
\special{pa 2590 2230}%
\special{fp}%
\special{sh 1}%
\special{pa 2590 2230}%
\special{pa 2657 2250}%
\special{pa 2643 2230}%
\special{pa 2657 2210}%
\special{pa 2590 2230}%
\special{fp}%
\put(27.1000,-18.7000){\makebox(0,0){$\gamma'$}}%
\end{picture}}%
\end{center}
\caption{Case 2: There exists a non-zero curve passing $s$ with negative direction.
Then we can find a non-zero curve $\gamma'$ passing $s$ on the upper point of $s$ in the positive direction.}
\label{transverse2}
\end{figure}
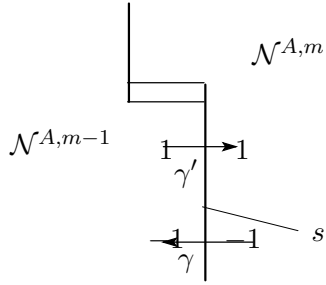
\end{proof}
\begin{lem}
\label{bounded}
Let $K$ be a lens space knot in an $\lz$ or $K_{p,k}$ in $Y_{p,k}$.
If $\gamma$ is one component of non-zero curve, then $\gamma$ is unbounded both $i$- and $j$-coordinate. 
\end{lem}
The unbounded-ness about $i$- or $j$-coordinate means that the projection of 
$\gamma$ to $i$- or $j$-component respectively is surjective.
\begin{proof}
Let $\gamma$ be a non-zero curve in a non-zero region $\mathcal{N}^{A,m}$ or $\mathcal{N}^B$.
Since $\gamma$ is monotone about $j$-coordinate, it is unbounded in $j$-coordinate.
If $\gamma$ is bounded above by $i=i_0$, then $\gamma$ is bounded above by $j$-coordinate of 
the top coefficient of $\mathcal{N}^{A,m}_{i_0}$.
Therefore $\gamma$ is unbounded about $i$-coordinate.

The proof for $K=K_{p,k}$ can be similarly proven by replacing $A$-matrix with $B$-matrix.
\end{proof}
As a corollary, $\mathcal{N}^{A,m}$ and $\mathcal{N}^B$ are connected, because if it is disconnected, then 
each component is bounded.
This means that any non-zero curve is bounded.\\
{\bf Theorem~\ref{non-zerocor}.}
{\it 
Let $K$ be a lens space knot in an $L{\mathbb Z}HS^3$ or $K_{p,k}$ for relatively positive integers $p,k$.
There is one non-zero curve only in each non-zero region.
}
\medskip
\begin{proof}
Suppose that two connected components $\gamma$ and $\delta$ of non-zero curves are contained in a non-zero region.
Since $\gamma$ and $\delta$ are disjoint each other, we may assume that one of three components in ${\Bbb R}^2-\gamma-\delta$ 
does not have any non-zero curve.
If $\gamma$ is upper than $\delta$, then we can find a not-alternating pair $1$ and $1$ in $\Delta_K(t)$ as seen in Figure~\ref{not-alternating}
\begin{figure}[htbp]
\begin{center}
{\unitlength 0.1in%
\begin{picture}( 28.0000, 20.2700)( 28.0000,-23.2700)%
%
\special{pn 8}%
\special{pa 5200 400}%
\special{pa 5200 800}%
\special{fp}%
\special{pa 5400 800}%
\special{pa 5400 1600}%
\special{fp}%
\special{pa 5600 1600}%
\special{pa 5600 2200}%
\special{fp}%
%
\special{pn 8}%
\special{pa 5400 800}%
\special{pa 5200 800}%
\special{fp}%
\special{pa 5400 1600}%
\special{pa 5600 1600}%
\special{fp}%
%
\special{pn 8}%
\special{pa 2800 400}%
\special{pa 2800 800}%
\special{fp}%
\special{pa 3000 800}%
\special{pa 3000 1600}%
\special{fp}%
\special{pa 3200 1600}%
\special{pa 3200 2200}%
\special{fp}%
%
\special{pn 8}%
\special{pa 3000 800}%
\special{pa 2800 800}%
\special{fp}%
\special{pa 3000 1600}%
\special{pa 3200 1600}%
\special{fp}%
%
\put(32.0000,-4.0000){\makebox(0,0)[rb]{}}%
%
\special{pn 8}%
\special{pa 5100 500}%
\special{pa 5100 500}%
\special{fp}%
%
\special{pn 8}%
\special{pa 4900 2100}%
\special{pa 4900 2300}%
\special{fp}%
%
\special{pn 8}%
\special{pa 5500 2100}%
\special{pa 4900 2100}%
\special{fp}%
%
\special{pn 8}%
\special{pa 5500 1900}%
\special{pa 5500 2100}%
\special{fp}%
%
\special{pn 8}%
\special{pa 5300 1900}%
\special{pa 5500 1900}%
\special{fp}%
%
\special{pn 8}%
\special{pa 5100 1700}%
\special{pa 5300 1900}%
\special{fp}%
%
\special{pn 8}%
\special{pa 4900 1700}%
\special{pa 4700 1500}%
\special{fp}%
%
\special{pn 8}%
\special{pa 4300 1300}%
\special{pa 4300 1500}%
\special{fp}%
%
\special{pn 8}%
\special{pa 5300 1300}%
\special{pa 4300 1300}%
\special{fp}%
%
\special{pn 8}%
\special{pa 5300 1100}%
\special{pa 5300 1300}%
\special{fp}%
%
\special{pn 8}%
\special{pa 5100 900}%
\special{pa 5300 1100}%
\special{fp}%
%
\special{pn 8}%
\special{pa 4900 700}%
\special{pa 5100 900}%
\special{fp}%
%
\special{pn 8}%
\special{pa 4500 700}%
\special{pa 4900 700}%
\special{fp}%
%
\special{pn 8}%
\special{pa 4500 500}%
\special{pa 4500 700}%
\special{fp}%
%
\special{pn 8}%
\special{pa 5100 500}%
\special{pa 4500 500}%
\special{fp}%
\special{sh 1}%
\special{pa 4500 500}%
\special{pa 4567 520}%
\special{pa 4553 500}%
\special{pa 4567 480}%
\special{pa 4500 500}%
\special{fp}%
%
\special{pn 8}%
\special{pa 4200 2300}%
\special{pa 4600 2300}%
\special{fp}%
%
\special{pn 8}%
\special{pa 3400 700}%
\special{pa 3200 500}%
\special{fp}%
%
\special{pn 8}%
\special{pa 3400 1100}%
\special{pa 3400 700}%
\special{fp}%
%
\special{pn 8}%
\special{pa 4000 1100}%
\special{pa 3400 1100}%
\special{fp}%
%
\special{pn 8}%
\special{pa 4000 1300}%
\special{pa 4000 1100}%
\special{fp}%
%
\special{pn 8}%
\special{pa 3800 1300}%
\special{pa 4000 1300}%
\special{fp}%
%
\special{pn 8}%
\special{pa 3800 1500}%
\special{pa 3800 1300}%
\special{fp}%
%
\special{pn 8}%
\special{pa 3000 500}%
\special{pa 3200 500}%
\special{fp}%
\special{sh 1}%
\special{pa 3200 500}%
\special{pa 3133 480}%
\special{pa 3147 500}%
\special{pa 3133 520}%
\special{pa 3200 500}%
\special{fp}%
%
\special{pn 8}%
\special{pa 4300 1500}%
\special{pa 4700 1500}%
\special{fp}%
\special{sh 1}%
\special{pa 4700 1500}%
\special{pa 4633 1480}%
\special{pa 4647 1500}%
\special{pa 4633 1520}%
\special{pa 4700 1500}%
\special{fp}%
%
\special{pn 8}%
\special{pa 4900 1700}%
\special{pa 5100 1700}%
\special{fp}%
\special{sh 1}%
\special{pa 5100 1700}%
\special{pa 5033 1680}%
\special{pa 5047 1700}%
\special{pa 5033 1720}%
\special{pa 5100 1700}%
\special{fp}%
%
\special{pn 8}%
\special{pa 4900 2300}%
\special{pa 5300 2300}%
\special{fp}%
\special{sh 1}%
\special{pa 5300 2300}%
\special{pa 5233 2280}%
\special{pa 5247 2300}%
\special{pa 5233 2320}%
\special{pa 5300 2300}%
\special{fp}%
%
\special{pn 8}%
\special{pa 3800 1500}%
\special{pa 4400 2100}%
\special{fp}%
\special{sh 1}%
\special{pa 4400 2100}%
\special{pa 4367 2039}%
\special{pa 4362 2062}%
\special{pa 4339 2067}%
\special{pa 4400 2100}%
\special{fp}%
%
\special{pn 8}%
\special{pa 4400 2100}%
\special{pa 4600 2100}%
\special{fp}%
\special{sh 1}%
\special{pa 4600 2100}%
\special{pa 4533 2080}%
\special{pa 4547 2100}%
\special{pa 4533 2120}%
\special{pa 4600 2100}%
\special{fp}%
%
\special{pn 8}%
\special{pa 4600 2290}%
\special{pa 4600 2090}%
\special{fp}%
%
\special{pn 8}%
\special{pa 5600 2300}%
\special{pa 5600 2200}%
\special{dt 0.045}%
%
\special{pn 8}%
\special{pa 3200 2300}%
\special{pa 3200 2200}%
\special{dt 0.045}%
%
\special{pn 8}%
\special{pa 2800 400}%
\special{pa 2800 300}%
\special{dt 0.045}%
%
\special{pn 8}%
\special{pa 5200 400}%
\special{pa 5200 300}%
\special{dt 0.045}%
%
\special{pn 8}%
\special{pa 4560 1460}%
\special{pa 4460 1460}%
\special{pa 4460 2200}%
\special{pa 4560 2200}%
\special{pa 4560 1460}%
\special{dt 0.045}%
\put(44.2000,-4.5000){\makebox(0,0)[rb]{$\gamma$}}%
\put(38.0000,-10.4000){\makebox(0,0)[rb]{$\delta$}}%
\end{picture}}%
\caption{Two non-zero coefficients which are not allowed in the broken box.}
\label{not-alternating}
\end{center}
\end{figure}
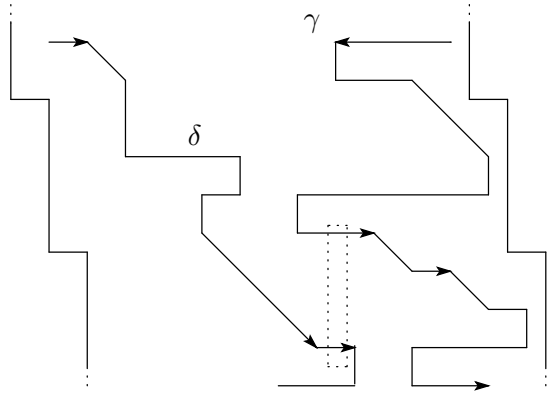
If one cannot find not-alternating coefficients, then $\gamma$ and $\delta$ must be separated by a vertical line $i=x_0$
for a real number $x_0$.
However, from Lemma~\ref{bounded} any non-zero curve is unbounded about the $i$-coordinate.
\end{proof}
We prove Corollary~\ref{thenumber}.
\begin{proof}
If on the non-zero curve one starts with a fixed non-zero lattice point ${\bf p}_0$ and goes to
the point ${\bf p}_0+(1,-k_2)$, then the curve tracks all the non-zero coefficients of the Alexander polynomial per one time.
The number is at least $k_2$ by considering the shifting length of non-zero region about the $j$-coordinate.
\end{proof}
\section{Applications by non-zero curve.}
\subsection{The lens surgeries with $2g(K)-1\le k_2$.}
\label{2d-1k2class}
By applying non-zero region and non-zero curve to lens surgery polynomial, 
we classify lens space knots in an $\lz$ with the parameter $(p,k,k_2)$ and $2g(K)-1\le k_2$.
\begin{prop}
\label{upk2bound}
Let $K$ be a lens space knot in an $\lz$ with parameter $(p,k,k_2)$ or $K_{p,k}$ in $Y_{p,k}$ and the genus $g$.
Then an inequality $k_2\le 2g+1$ holds.
\end{prop}
\begin{proof}
Let $\mathcal{N}^{X}=\cup_{l\in {\Bbb Z}}{\mathcal N}_l^{X}$ be a non-zero region, where $X=(A,m)$ or $B$.
If $k_2\ge 2g+2$, then $\mathcal{N}^X$ is disconnected,
since the parallel transform is $(1,-k_2)$.
Thus, this is contradiction.
\end{proof}

\begin{theorem}
\label{k22g1}
Let $K$ be a lens space knot in an $\lz$ with parameter $(p,k,k_2)$ and $p>2$ and genus $g$.
If $2g\le k_2\le 2g+1$, then $e=-1$, $k=2$ and $\Delta_K(t)=\Delta_{T(2,2g+1)}$.
The parameters are $(p,k,k_2)=(4g+3,2,2g+1)$, $(4g+1,2,2g)$ respectively.
There exist no lens surgeries with $k_2=2g-1$.
\end{theorem}
\begin{proof}
Let $K$ be a lens space knot in an $\lz$ with $g=g(K)$ and $p>2$.
If $k_2=2g+1$, then by the parallel translation by $(1,-k_2)=(1,-(2g+1))$ two adjacent
$\mathcal{N}_l^{A,m}$ and $\mathcal{N}_{l+1}^{A,m}$ meet one corner point (Figure~\ref{meetonepoint}).
Hence, all the lattice points in the non-zero region ${\mathcal N}^{A,m}$ for any integer $m$ give non-zero coefficients.
Thus the Alexander polynomial is as follows:
$$\Delta_K(t)=t^g-t^{g-1}+t^{g-2}-\cdots+t^{-g}=\Delta_{T(2,2g+1)}(t).$$

By the definition of $k_2$, we have $2k_2\le p-1$.
If $2k_2\le p-2$, then there exists some integer $x$ such that $dA(x)=dA(x+ek)=0$.
From (\ref{eq2}) in Lemma~\ref{alternating}, the remainder of $q_2=k_2^2$ in $\{0,1,\cdots, p-1\}$ is 
smaller than $p-k_2$.
The $dA$-matrix $dA_{l,n}$ is described as follows by using (\ref{eq2}) in Lemma~\ref{alternating}
$$\begin{cases}1&[q_2l+enk_2]_{p}\in I_{-k_2}\\-1&[q_2l+enk_2]_{p}\in I_{k_2}\\0&\text{otherwise}\end{cases}$$
Thus non-zero values of the sequence $\{dA_{l,n}\}_{n\in {\mathbb Z}}$ are adjacent mutually as follows:
$$\cdots, 0,1,-1,1,-1,1,-1,\cdots, 1,-1,0,\cdots$$
In particular, $p<3k_2$ holds.
Thus, the zero values in the sequence $\{dA_{l,n}\}_{n\in {\mathbb Z}}$ are isolated, i.e., if $dA_{i,j}=0$ then $dA_{i,j\pm1}\neq 0$ holds.
Therefore, the $x$ with $dA(x)=0$ in the period $p$ is unique.
Thus $2(2g+1)+1=p$ holds.
Thus $k=2$ and $e=-1$ hold.
\begin{figure}[htbp]
\begin{center}
{\unitlength 0.1in%
\begin{picture}( 27.2000, 23.8500)( 15.0000,-24.5500)%
%
\special{pn 8}%
\special{pa 4020 800}%
\special{pa 3820 800}%
\special{pa 3820 1400}%
\special{pa 4020 1400}%
\special{pa 4020 800}%
\special{pa 3820 800}%
\special{fp}%
%
\special{pn 8}%
\special{pa 4220 1400}%
\special{pa 4020 1400}%
\special{pa 4020 2000}%
\special{pa 4220 2000}%
\special{pa 4220 1400}%
\special{pa 4020 1400}%
\special{fp}%
%
\special{pn 8}%
\special{pa 3820 200}%
\special{pa 3620 200}%
\special{pa 3620 800}%
\special{pa 3820 800}%
\special{pa 3820 200}%
\special{pa 3620 200}%
\special{fp}%
\put(37.2000,-14.8000){\makebox(0,0){0}}%
\put(39.2000,-14.8000){\makebox(0,0){-1}}%
\put(39.0000,-25.2000){\makebox(0,0){$dA_{i,j}$}}%
%
\special{pn 8}%
\special{pa 2000 800}%
\special{pa 1800 800}%
\special{pa 1800 1400}%
\special{pa 2000 1400}%
\special{pa 2000 800}%
\special{pa 1800 800}%
\special{fp}%
%
\special{pn 8}%
\special{pa 2200 1400}%
\special{pa 2000 1400}%
\special{pa 2000 2000}%
\special{pa 2200 2000}%
\special{pa 2200 1400}%
\special{pa 2000 1400}%
\special{fp}%
%
\special{pn 8}%
\special{pa 1800 200}%
\special{pa 1600 200}%
\special{pa 1600 800}%
\special{pa 1800 800}%
\special{pa 1800 200}%
\special{pa 1600 200}%
\special{fp}%
%
\special{pn 8}%
\special{pa 1720 320}%
\special{pa 1500 114}%
\special{fp}%
%
\special{pn 8}%
\special{pa 2100 1510}%
\special{pa 1900 1310}%
\special{fp}%
%
\special{pn 8}%
\special{ar 2100 1800 60 110  1.5707963  4.7123890}%
%
\special{pn 8}%
\special{ar 2090 1600 70 92  4.7123890  1.5707963}%
%
\special{pn 8}%
\special{ar 1900 1200 60 110  1.5707963  4.7123890}%
%
\special{pn 8}%
\special{ar 1700 600 60 110  1.5707963  4.7123890}%
%
\special{pn 8}%
\special{ar 1890 1000 70 92  4.7123890  1.5707963}%
%
\special{pn 8}%
\special{ar 1690 400 70 92  4.7123890  1.5707963}%
%
\special{pn 8}%
\special{pa 2100 1910}%
\special{pa 2300 2110}%
\special{fp}%
\special{sh 1}%
\special{pa 2300 2110}%
\special{pa 2267 2049}%
\special{pa 2262 2072}%
\special{pa 2239 2077}%
\special{pa 2300 2110}%
\special{fp}%
%
\special{pn 8}%
\special{pa 2200 800}%
\special{pa 2000 1000}%
\special{fp}%
\put(24.0000,-13.3000){\makebox(0,0)[lb]{$\mathcal{N}_{l+1}^{A,m}$}}%
%
\special{pn 8}%
\special{pa 2400 1330}%
\special{pa 2200 1530}%
\special{fp}%
\put(20.0000,-2.0000){\makebox(0,0)[lb]{$\mathcal{N}_{l-1}^{A,m}$}}%
%
\special{pn 8}%
\special{pa 2000 200}%
\special{pa 1800 400}%
\special{fp}%
\put(22.0000,-8.0000){\makebox(0,0)[lb]{$\mathcal{N}_l^{A,m}$}}%
%
\special{pn 8}%
\special{pa 1900 910}%
\special{pa 1700 710}%
\special{fp}%
\put(37.2000,-2.7000){\makebox(0,0){1}}%
\put(41.2000,-14.7000){\makebox(0,0){1}}%
\put(35.2000,-2.7000){\makebox(0,0){-1}}%
%
\special{pn 8}%
\special{pa 3820 1600}%
\special{pa 3620 1600}%
\special{pa 3620 2200}%
\special{pa 3820 2200}%
\special{pa 3820 1600}%
\special{pa 3620 1600}%
\special{fp}%
\put(37.2000,-21.0000){\makebox(0,0){1}}%
\put(37.2000,-4.7000){\makebox(0,0){-1}}%
\put(35.2000,-6.7000){\makebox(0,0){-1}}%
\put(37.2000,-6.7000){\makebox(0,0){1}}%
\put(37.2000,-8.7000){\makebox(0,0){-1}}%
\put(37.2000,-12.7000){\makebox(0,0){-1}}%
\put(37.2000,-10.7000){\makebox(0,0){1}}%
\put(39.2000,-8.7000){\makebox(0,0){1}}%
\put(39.2000,-12.7000){\makebox(0,0){1}}%
\put(39.2000,-10.7000){\makebox(0,0){-1}}%
\put(35.2000,-4.7000){\makebox(0,0){1}}%
\put(39.2000,-16.7000){\makebox(0,0){1}}%
\put(41.2000,-18.7000){\makebox(0,0){1}}%
\put(37.2000,-16.7000){\makebox(0,0){1}}%
\put(37.2000,-18.8000){\makebox(0,0){-1}}%
\put(41.2000,-16.8000){\makebox(0,0){-1}}%
\put(39.2000,-18.8000){\makebox(0,0){-1}}%
\put(35.2000,-16.8000){\makebox(0,0){-1}}%
\put(35.2000,-20.8000){\makebox(0,0){-1}}%
\put(35.2000,-18.7000){\makebox(0,0){1}}%
%
\special{pn 8}%
\special{pa 3620 1000}%
\special{pa 3420 1000}%
\special{pa 3420 1600}%
\special{pa 3620 1600}%
\special{pa 3620 1000}%
\special{pa 3420 1000}%
\special{fp}%
\put(35.2000,-14.7000){\makebox(0,0){1}}%
\put(35.2000,-10.7000){\makebox(0,0){1}}%
\put(35.2000,-12.7000){\makebox(0,0){-1}}%
\end{picture}}%
\caption{The case of $k_2=2g+1$.}
\label{meetonepoint}
\end{center}
\end{figure}

If $k_2=2g$, then by the translation $(1,-2g)$, the adjacent non-zero regions $\mathcal{N}_l^{A,m}$ and $\mathcal{N}_{l+1}^{A,m}$ are
attached at length one segment as in Figure~\ref{torus2}.
Then we have:
$$\Delta_K(t)=t^g-t^{g-1}+t^{g-2}-\cdots+t^{-g}=\Delta_{T(2,2g+1)}(t).$$
In the same reason as the case of $k_2=2g+1$, $p<3k_2$ and zero values in the sequence $\{dA_{l,n}\}_{n\in {\mathbb Z}}$ are isolated.
Thus $2\cdot 2g+1=p$ holds.
This means $k=2$ and $e=-1$.
\begin{figure}[htbp]
\begin{center}
{\unitlength 0.1in%
\begin{picture}( 30.0000, 21.3500)( 12.0000,-23.3500)%
\put(26.0000,-12.0000){\makebox(0,0){$\Rightarrow$}}%
%
\special{pn 8}%
\special{pa 3800 200}%
\special{pa 3800 800}%
\special{fp}%
%
\special{pn 8}%
\special{pa 3600 800}%
\special{pa 3600 200}%
\special{fp}%
%
\special{pn 8}%
\special{pa 3800 800}%
\special{pa 3600 800}%
\special{fp}%
%
\special{pn 8}%
\special{pa 4000 600}%
\special{pa 3800 600}%
\special{pa 3800 1600}%
\special{pa 4000 1600}%
\special{pa 4000 600}%
\special{pa 3800 600}%
\special{fp}%
%
\special{pn 8}%
\special{pa 4200 1400}%
\special{pa 4200 2000}%
\special{fp}%
%
\special{pn 8}%
\special{pa 4000 1400}%
\special{pa 4200 1400}%
\special{fp}%
%
\special{pn 8}%
\special{pa 4000 2000}%
\special{pa 4000 1400}%
\special{fp}%
%
\special{pn 8}%
\special{pa 3800 1600}%
\special{pa 3800 2200}%
\special{fp}%
%
\special{pn 8}%
\special{pa 3600 1600}%
\special{pa 3800 1600}%
\special{fp}%
%
\special{pn 8}%
\special{pa 3600 2200}%
\special{pa 3600 1600}%
\special{fp}%
\put(37.0000,-15.2000){\makebox(0,0){1}}%
\put(41.0000,-7.2000){\makebox(0,0){-1}}%
%
\special{pn 8}%
\special{pa 1400 200}%
\special{pa 1400 800}%
\special{fp}%
%
\special{pn 8}%
\special{pa 1200 800}%
\special{pa 1200 200}%
\special{fp}%
%
\special{pn 8}%
\special{pa 1400 800}%
\special{pa 1200 800}%
\special{fp}%
%
\special{pn 8}%
\special{pa 1600 600}%
\special{pa 1400 600}%
\special{pa 1400 1600}%
\special{pa 1600 1600}%
\special{pa 1600 600}%
\special{pa 1400 600}%
\special{fp}%
%
\special{pn 8}%
\special{pa 1800 1400}%
\special{pa 1800 2000}%
\special{fp}%
%
\special{pn 8}%
\special{pa 1600 1400}%
\special{pa 1800 1400}%
\special{fp}%
%
\special{pn 8}%
\special{pa 1600 2000}%
\special{pa 1600 1400}%
\special{fp}%
%
\special{pn 8}%
\special{pa 1260 720}%
\special{pa 1540 720}%
\special{fp}%
\special{sh 1}%
\special{pa 1540 720}%
\special{pa 1473 700}%
\special{pa 1487 720}%
\special{pa 1473 740}%
\special{pa 1540 720}%
\special{fp}%
%
\special{pn 8}%
\special{pa 1460 1520}%
\special{pa 1735 1520}%
\special{fp}%
\special{sh 1}%
\special{pa 1735 1520}%
\special{pa 1668 1500}%
\special{pa 1682 1520}%
\special{pa 1668 1540}%
\special{pa 1735 1520}%
\special{fp}%
%
\special{pn 8}%
\special{ar 1540 1220 50 100  4.7123890  1.5707963}%
%
\special{pn 8}%
\special{ar 1540 820 50 100  4.7123890  1.5707963}%
%
\special{pn 8}%
\special{pa 1560 920}%
\special{pa 1440 920}%
\special{fp}%
%
\special{pn 8}%
\special{pa 1560 1120}%
\special{pa 1440 1120}%
\special{fp}%
%
\special{pn 8}%
\special{pa 1560 1320}%
\special{pa 1440 1320}%
\special{fp}%
%
\special{pn 8}%
\special{ar 1460 1020 50 100  1.5707963  4.7123890}%
%
\special{pn 8}%
\special{ar 1460 1410 50 100  1.5707963  4.7123890}%
%
\special{pn 8}%
\special{pn 8}%
\special{pa 1260 720}%
\special{pa 1253 719}%
\special{fp}%
\special{pa 1230 701}%
\special{pa 1227 695}%
\special{fp}%
\special{pa 1215 664}%
\special{pa 1214 658}%
\special{fp}%
\special{pa 1210 622}%
\special{pa 1210 614}%
\special{fp}%
\special{pa 1214 581}%
\special{pa 1216 574}%
\special{fp}%
\special{pa 1227 544}%
\special{pa 1231 539}%
\special{fp}%
\special{pa 1253 521}%
\special{pa 1260 520}%
\special{fp}%
%
\special{pn 8}%
\special{pn 8}%
\special{pa 1740 1520}%
\special{pa 1747 1521}%
\special{fp}%
\special{pa 1768 1538}%
\special{pa 1772 1542}%
\special{fp}%
\special{pa 1784 1574}%
\special{pa 1786 1581}%
\special{fp}%
\special{pa 1790 1614}%
\special{pa 1790 1622}%
\special{fp}%
\special{pa 1786 1657}%
\special{pa 1785 1664}%
\special{fp}%
\special{pa 1772 1696}%
\special{pa 1769 1702}%
\special{fp}%
\special{pa 1747 1719}%
\special{pa 1740 1720}%
\special{fp}%
\put(39.0000,-15.2000){\makebox(0,0){1}}%
\put(39.0000,-13.2000){\makebox(0,0){-1}}%
\put(39.0000,-11.2000){\makebox(0,0){1}}%
\put(39.0000,-9.2000){\makebox(0,0){-1}}%
\put(39.0000,-7.2000){\makebox(0,0){0}}%
\put(37.0000,-5.2000){\makebox(0,0){-1}}%
\put(41.0000,-15.2000){\makebox(0,0){1}}%
\put(41.0000,-17.2000){\makebox(0,0){-1}}%
\put(39.0000,-17.2000){\makebox(0,0){1}}%
%
\special{pn 8}%
\special{pa 3600 1400}%
\special{pa 3600 2000}%
\special{fp}%
\put(35.0000,-15.2000){\makebox(0,0){-1}}%
\put(37.0000,-17.2000){\makebox(0,0){0}}%
\put(37.0000,-19.2000){\makebox(0,0){-1}}%
\put(35.0000,-17.2000){\makebox(0,0){1}}%
\put(41.0000,-9.2000){\makebox(0,0){1}}%
\put(41.0000,-11.2000){\makebox(0,0){-1}}%
\put(41.0000,-13.2000){\makebox(0,0){1}}%
%
\special{pn 8}%
\special{pa 3600 800}%
\special{pa 3400 800}%
\special{pa 3400 1800}%
\special{pa 3600 1800}%
\special{pa 3600 800}%
\special{pa 3400 800}%
\special{fp}%
\put(37.0000,-13.2000){\makebox(0,0){-1}}%
\put(37.0000,-11.2000){\makebox(0,0){1}}%
\put(37.0000,-9.2000){\makebox(0,0){-1}}%
\put(37.0000,-7.2000){\makebox(0,0){1}}%
\put(35.0000,-9.2000){\makebox(0,0){0}}%
\put(35.0000,-11.2000){\makebox(0,0){-1}}%
\put(35.0000,-13.2000){\makebox(0,0){1}}%
\put(38.0000,-24.0000){\makebox(0,0){$dA_{i,j}$}}%
\put(15.0000,-24.0000){\makebox(0,0){$A_{i,j}$}}%
\end{picture}}%
\caption{The case of $k=2g$.}
\label{torus2}
\end{center}
\end{figure}

Suppose that $k_2=2g-1$.
Let $A_{i,j}$ be the entry with $k_2(i-c)+j=-d$.
If $A_{i+1,j}=1$ then since $A_{i+1,j+1}=1$ holds, this does not satisfy alternating condition in Theorem~\ref{OS1}.
If $A_{i+1,j}=0$ then $(i,j)$ is an end point of the non-zero curve.
If $A_{i+1,j}=-1$, then $dA_{i+1,j}=-2$ holds.
This contradicts to (\ref{eq2}) in Lemma~\ref{alternating}.
\end{proof}

Thus we have only to consider the case of $k_2\le 2g-2$ to investigate the lens space knot in an $\lz$
to classify the lens space realization.

\subsection{Lens surgeries with $(2,r)$-torus knot polynomial.}
We prove Theorem~\ref{2ntorusknot} (the classification of lens space knots in an $\lz$ with $(2,r)$-torus knot polynomial).
\begin{proof}
In Theorem~\ref{k22g1}, we proved $4\Rightarrow 1$ and $4\Rightarrow 2$.
Here we prove $1\Rightarrow 4$.
Suppose that $\Delta_K(t)=\Delta_{T(2,2g+1)}$, $k_2\le 2g-2$ and $e=-1$.
If $\mathcal{N}^{A,m}_l\cap \mathcal{N}^{A,m-1}_{l+1}$ is some interval, then 
for some integer $x$, $dA(x)=\pm2$ holds or some non-zero curve is passing between different non-zero regions.
Both the cases give contradiction.

Therefore $\mathcal{N}^{A,m}\cap \mathcal{N}^{A,m-1}$ is a set of discrete (corner) points or
empty.
See Figure~\ref{2torus}.
\begin{figure}[bthp]
\begin{center}
\input{2torus.tex}
\caption{$\Delta_K(t)=\Delta_{T(2,2g+1)}$, $k_2\le 2g-2$.}
\label{2torus}
\end{center}
\end{figure}
If $k_2\ge 4$, then there exists some integer $x$ such that $dA(x)=1$ and $dA(x+ek)=-1$.
Thus $p<3k_2$ holds.
Now, from the condition $k_2\le 2g-2$, there exists some integer $y$ such that
$dA(y)=dA(y+k)=dA(y+2k)=0$ holds.
This is contradiction.
Therefore $2g-1\le k_2$ holds.
If $k_2\le 3$ then the parameter is $(p,2,3)$ or $(p,2,2)$ only.
The parameters correspond to $(7,2,3)$ or $(5,2,2)$ respectively.
These cases are also $\Delta_K(t)=\Delta_{T(2,3)}(t)$.

We prove $2\Rightarrow 1$.
If $(p,2)$ is lens surgery parameter, then by using Theorem~\ref{yamkatan},
$\Delta_K(t)=\Delta_{T(2,2g+1)}$ holds.

The equivalence $3\Leftrightarrow 2$ is due to the definition of realization of lens surgery parameter.
The equivalence $1\Leftrightarrow 5$ is due to the definition of $\alpha$-index.
\end{proof}

Next, we classify the knots $K_{p,k}$ with torus knot polynomial $T(2,r)$.

{\it Proof of Theorem~\ref{DT22r+1}.}
Form Proposition~\ref{upk2bound}, we may assume that $k_2\le 2g+1$.
Suppose that $k_2\ge 4$ as in the first picture in Figure~\ref{Bn2d+1po}.
\begin{figure}[bthp]
\begin{center}
{\unitlength 0.1in%
\begin{picture}( 31.8000, 30.1000)( 13.6000,-32.0000)%
%
\special{pn 8}%
\special{pa 2120 1800}%
\special{pa 1920 1800}%
\special{pa 1920 2800}%
\special{pa 2120 2800}%
\special{pa 2120 1800}%
\special{pa 1920 1800}%
\special{fp}%
%
\special{pn 8}%
\special{pa 1920 1000}%
\special{pa 1720 1000}%
\special{pa 1720 2000}%
\special{pa 1920 2000}%
\special{pa 1920 1000}%
\special{pa 1720 1000}%
\special{fp}%
%
\special{pn 8}%
\special{pa 1720 200}%
\special{pa 1520 200}%
\special{pa 1520 1200}%
\special{pa 1720 1200}%
\special{pa 1720 200}%
\special{pa 1520 200}%
\special{fp}%
\put(18.0000,-32.0000){\makebox(0,0){$dB$-function}}%
\put(18.2000,-11.0000){\makebox(0,0){0}}%
\put(18.2000,-13.0000){\makebox(0,0){-1}}%
\put(20.2000,-21.0000){\makebox(0,0){-1}}%
\put(20.2000,-19.0000){\makebox(0,0){0}}%
\put(16.2000,-5.0000){\makebox(0,0){-1}}%
\put(16.2000,-11.0000){\makebox(0,0){1}}%
\put(18.2000,-19.0000){\makebox(0,0){1}}%
\put(20.2000,-23.0000){\makebox(0,0){0}}%
%
\special{pn 8}%
\special{pa 2340 2400}%
\special{pa 2540 3200}%
\special{dt 0.045}%
%
\special{pn 8}%
\special{pa 1360 190}%
\special{pa 1450 550}%
\special{dt 0.045}%
\put(20.2000,-25.0000){\makebox(0,0){0}}%
\put(20.2000,-27.0000){\makebox(0,0){0}}%
\put(16.2000,-7.0000){\makebox(0,0){1}}%
\put(16.2000,-9.0000){\makebox(0,0){-1}}%
%
\special{pn 8}%
\special{pa 4320 1400}%
\special{pa 4120 1400}%
\special{pa 4120 2400}%
\special{pa 4320 2400}%
\special{pa 4320 1400}%
\special{pa 4120 1400}%
\special{fp}%
%
\special{pn 8}%
\special{pa 4120 1000}%
\special{pa 3920 1000}%
\special{pa 3920 2000}%
\special{pa 4120 2000}%
\special{pa 4120 1000}%
\special{pa 3920 1000}%
\special{fp}%
%
\special{pn 8}%
\special{pa 3920 600}%
\special{pa 3720 600}%
\special{pa 3720 1600}%
\special{pa 3920 1600}%
\special{pa 3920 600}%
\special{pa 3720 600}%
\special{fp}%
\put(40.2000,-11.0000){\makebox(0,0){0}}%
\put(40.2000,-13.0000){\makebox(0,0){0}}%
\put(42.2000,-17.0000){\makebox(0,0){0}}%
\put(42.2000,-15.0000){\makebox(0,0){0}}%
\put(38.2000,-9.0000){\makebox(0,0){0}}%
\put(38.2000,-7.0000){\makebox(0,0){0}}%
\put(40.2000,-15.0000){\makebox(0,0){0}}%
\put(38.2000,-15.0000){\makebox(0,0){1}}%
\put(40.2000,-17.0000){\makebox(0,0){-1}}%
\put(40.2000,-19.0000){\makebox(0,0){1}}%
\put(42.2000,-19.0000){\makebox(0,0){0}}%
%
\special{pn 8}%
\special{pa 4340 2000}%
\special{pa 4540 2800}%
\special{dt 0.045}%
%
\special{pn 8}%
\special{pa 3560 590}%
\special{pa 3650 950}%
\special{dt 0.045}%
\put(42.2000,-21.0000){\makebox(0,0){-1}}%
\put(42.2000,-23.0000){\makebox(0,0){1}}%
\put(38.2000,-11.0000){\makebox(0,0){0}}%
\put(38.2000,-13.0000){\makebox(0,0){-1}}%
\put(40.0000,-28.0000){\makebox(0,0){$dB$-function}}%
\put(18.2000,-15.0000){\makebox(0,0){1}}%
\put(18.2000,-17.0000){\makebox(0,0){-1}}%
\put(22.2000,-19.0000){\makebox(0,0){-1}}%
\put(20.2000,-15.0000){\makebox(0,0){-1}}%
\put(20.2000,-17.0000){\makebox(0,0){1}}%
\put(20.2000,-13.0000){\makebox(0,0){1}}%
\put(22.2000,-21.0000){\makebox(0,0){1}}%
\put(22.2000,-25.0000){\makebox(0,0){1}}%
\put(22.2000,-23.0000){\makebox(0,0){-1}}%
\put(20.2000,-11.0000){\makebox(0,0){-1}}%
\put(18.2000,-9.0000){\makebox(0,0){1}}%
\put(18.2000,-5.0000){\makebox(0,0){1}}%
\put(18.2000,-7.0000){\makebox(0,0){-1}}%
\put(42.2000,-13.0000){\makebox(0,0){1}}%
\put(42.2000,-11.0000){\makebox(0,0){-1}}%
\put(44.2000,-17.0000){\makebox(0,0){1}}%
\put(44.2000,-15.0000){\makebox(0,0){-1}}%
\put(40.2000,-7.0000){\makebox(0,0){-1}}%
\put(40.2000,-9.0000){\makebox(0,0){1}}%
\put(16.2000,-3.0000){\makebox(0,0){0}}%
\end{picture}}%
\caption{$\Delta_{K_{p,k}}(t)=\Delta_{T(2,2g+1)}(t)$.}
\label{Bn2d+1po}
\end{center}
\end{figure}
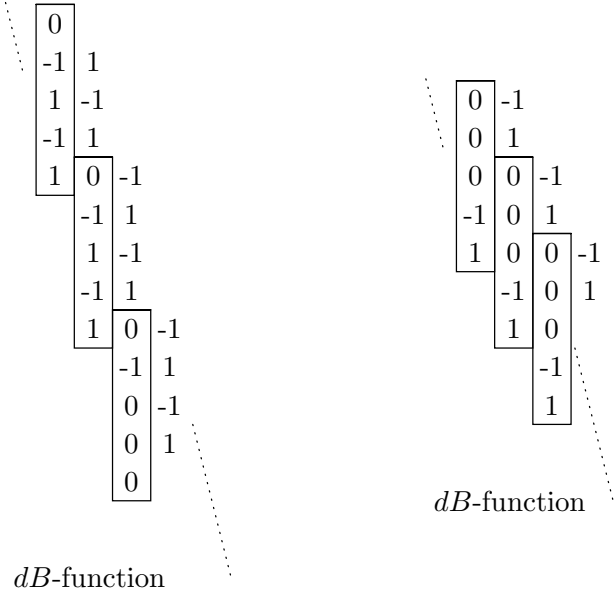
Then for any $i$ there exists an integer $j$ such that
$$\Phi(l-1)p-([[q(l-1)]]_p+i)k_2=j-1\ \ \ \text{for some }l\in I_{k_2}$$
$$\Phi(l'-1)p-([[q(l'-1)]]_p+i)k_2=j+1\ \ \ \text{for some }l'\in I_{k_2},$$
Thus $l,l'$ satisfy
$$l\equiv 1-qi+k_2(j+1),\ \ l'\equiv 1-qi+k_2(j-1)\bmod p,$$
where $1\le l,l'\le k_2$ holds.
Thus, in particular, $p<3k_2$ holds.

If $2k_2+3<p<3k_2$, then there exists an integer $j_0$ such that $[1-qi+k_2(j_0+1)]_p\not\in I_{k_2}$
and $[1-qi+k_2(j_0+2)]_p\not\in I_{k_2}$.
This is contradiction to $3k_2<p$.
Thus we obtain $p\le 2k_2+2$.
Since $(p,k_2)=1$, then we have $p=2k_2+1$.
Then we have $k=2$ and $e=-1$.
This knot $K_{2k_2+1,2}$ is in $S^3$.

Suppose that $k_2\le 3$ (the second picture in Figure~\ref{Bn2d+1po}).
From the inequality $k\le k_2$, in this case, we have $(p,k)=(5,2), (8,3)$, or $(10,3)$ only.
Here the non-zero sequence of $K_{8,3}$ is
$$NS_h(K_{8,3})=(4,3,1,0)$$
and for $K_{10,3}$ can be seen in Table~\ref{nonLspace23}.
These are not $(2,r)$-torus knot polynomials.
Thus this case also holds.
\qed
\subsection{An $\alpha$-index inequality}
In this section we prove the theorems in Section~\ref{dualclassbounds}, and \ref{nthtermofAlex}.
We prove Theorem~\ref{them2}.
\medskip\\
{\it Proof of Theorem~\ref{them2}.}
Suppose that $X=(A,m)$ or $B$.
Let $\mathcal{N}^X$ be the non-zero region.
Let $(i_0,j_0)$ be the top lattice point in $\mathcal{N}_{i_0}^X$.

Suppose that there exists an integer $0<j<k_2$ with $X_{i_0,j_0-j}=1$.
Since any non-zero curve has no end points, there is the next point $(i_0,j_0-j-1)$ of $(i_0,j_0-j)$.
Thus $X_{i_0,j_0-j-1}=-1$ holds.
This means $n_1-k_2+1\le i\le n_1$ is included in the adjacent region.
The leftmost picture in Figure~\ref{alternative} presents this argument.

If $X_{i_0,j_0-k_2+1}=1$, then $(i_0+1,j_0-k_2)$ (the lattice point corresponding to $a_g(K)$) is the next point of $(i_0,j_0-k_2+1)$.
Thus $X_{i_0,j_0-k_2}=0$ holds.
Then $\alpha(K)=k_2-1$ (the second picture in Figure~\ref{alternative}).

Let $s$ be the adjacent length.
If $X_{i_0,j_0-k_2+1}\neq 1$, then $X_{i_0,j_0-k_2}=1$ and $n_1-k_2$ is in the adjacent region.
Thus, we have $n_1-k_2=n_{2s_2-1}$ for some integer $1<s_2\le s$.
This case is one of the third, fourth or fifth pictures in Figure~\ref{alternative}.
The latter case is $\alpha(K)=k_2$ holds.
If $k<k_2$ holds, considering the $A'$-matrix, we also obtain $n_1-k=n_{2s_1-1}$ for some integer $1<s_1<s_2$.
\begin{figure}[bthp]
\begin{center}
\input{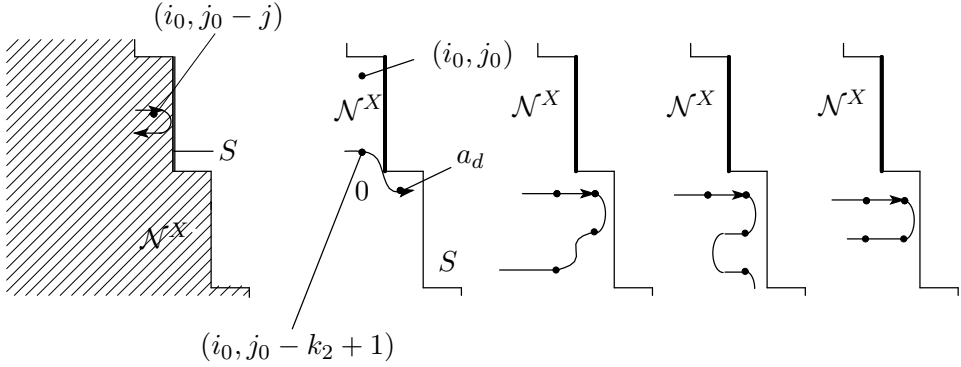}
\caption{The cases of $0\le j< k_2-1$ or $j=k_2-1$.}
\label{alternative}
\end{center}
\end{figure}
\qed

We prove Corollary~\ref{main}.
\begin{proof}
Suppose that $n_2<n_1-1$.
Then $\alpha$ is equal to $0$.
By using the right previous inequality $k_2\le \alpha+1$, we have $k=k_2=1$.
This is the surgery parameter of trivial knot surgery.
\end{proof}

Here we give a proof of Theorem~\ref{a22a2cor}.
\begin{proof}
By using $mp=kk_2-e$ and Theorem~\ref{them2}, we have
$$p=\frac{kk_2-e}{m}\le k_2^2+1\le\ (\alpha(K)+1)^2+1=\alpha(K)^2+2\alpha(K)+2.$$
The last statement of this theorem is due to this inequality.
\end{proof}

Here we give a classification of lens space knots with $\alpha(K)=2$.
\begin{cor}
\label{trefoilorollary}
Let $K$ be a lens space knot in an $\lz$ or $K_{p,k}$ in $Y_{p,k}$ with genus $g$.
$K$ is a lens space knot in an $\lz$ or $K_{p,k}$ with $\alpha(K)=2$ if and only if $\Delta_K(t)=t-1+t^{-1}$.

In other words, if $\Delta_K(t)$ with more than three non-zero coefficients
and is expanded as follows
$$\Delta_K(t)=t^g-t^{g-1}+t^{g-2}-t^{g-g'}+\cdots,$$
then $g'=3$ holds.
\end{cor}
\begin{proof}
If $\alpha(K)=2$ then $p\le 10$ and $k_2\le 3$ by using Theorem~\ref{a22a2cor}.
The surgery parameters with $\alpha(K)=2$ are $(5,2),(7,2)$, $(8,3)$ or $(10,3)$.
The half non-zero sequences of the parameters $(8,3)$ and $(10,3)$ are $(4,3,1,0)$ and $(6,5,3,2,0)$ respectively.
These sequences do not satisfy $\alpha(K)=2$.
The parameters $(5,2)$ and $(7,2)$ are realized by lens surgeries by the trefoil.
Conversely if $\Delta_K(t)=t-1+t^{-1}$, then $\alpha(K)=2$ by the definition of $\alpha$-index.

If the $\Delta_K(t)$ has more than three non-zero coefficients, then $\alpha(K)>2$
holds.
Thus if $g,g-1,g-2$ are exponents of non-zero coefficients, then $g-3$ is 
in the adjacent region.
Thus $g'=3$ holds.
\end{proof}

In the following, we prove Theorem~\ref{3th}.
\begin{proof}
Let $K$ be a lens space knot in an $\lz$ or $K_{p,k}$ with genus $g$.
Suppose that $\alpha(K)=k_2-1$ and $s$ is the adjacent length i.e., $\alpha(K)=n_1-n_{2s-1}$ holds.
Let $(i_0,j_0)$ be the lattice point for the top coefficient $a_{g}(K)$ in $\mathcal{N}^X_{i_0}$.
From the assumption, we have $X_{i_0-1,j_0}=0$ and $X_{i_0-1,j_0+1}=1$.

In the case of $n_{2s}<n_{2s-1}-3$, the non-zero curve is as (a) in Figure~\ref{alternatfig}.
Since the non-zero curve is connected, $n_3=n_2-1$ holds.

In the case of $n_{2s}=n_{2s-1}-3$, the non-zero curve is as (b) in Figure~\ref{alternatfig}.
Since the non-zero curve is connected, $n_3\le n_2-2$ holds.

In the case of $n_{2s}=n_{2s-1}-2$, since the non-zero curve is connected, 
$n_2-n_3=n_{2s}-n_{2s+1}$ or $n_2-n_3=n_{2s}-n_{2s+1}+1$ holds.
See (c) in Figure~\ref{alternatfig}.
Furthermore, if $n_3=n_2-1$, then the former case holds.
Thus, $n_{2s+1}=n_{2s}-1$.
\begin{figure}[htbp]
\begin{center}
{\unitlength 0.1in%
\begin{picture}( 34.4500, 16.7000)(  1.3500,-22.7000)%
%
\special{pn 8}%
\special{pa 1120 1600}%
\special{pa 920 1600}%
\special{fp}%
%
\special{pn 8}%
\special{pa 400 600}%
\special{pa 400 600}%
\special{fp}%
%
\special{pn 8}%
\special{pa 950 1700}%
\special{pa 1065 1700}%
\special{fp}%
\special{sh 1}%
\special{pa 1065 1700}%
\special{pa 998 1680}%
\special{pa 1012 1700}%
\special{pa 998 1720}%
\special{pa 1065 1700}%
\special{fp}%
%
\special{pn 8}%
\special{pa 955 1700}%
\special{pa 855 1530}%
\special{fp}%
%
\special{pn 8}%
\special{pa 1070 1834}%
\special{pa 950 1834}%
\special{fp}%
\special{sh 1}%
\special{pa 950 1834}%
\special{pa 1017 1854}%
\special{pa 1003 1834}%
\special{pa 1017 1814}%
\special{pa 950 1834}%
\special{fp}%
%
\special{pn 8}%
\special{ar 1060 1769 40 69  4.7123890  1.5707963}%
%
\special{pn 8}%
\special{pa 735 1530}%
\special{pa 855 1530}%
\special{fp}%
\special{sh 1}%
\special{pa 855 1530}%
\special{pa 788 1510}%
\special{pa 802 1530}%
\special{pa 788 1550}%
\special{pa 855 1530}%
\special{fp}%
%
\special{pn 8}%
\special{pa 950 1990}%
\special{pa 1070 1990}%
\special{fp}%
\special{sh 1}%
\special{pa 1070 1990}%
\special{pa 1003 1970}%
\special{pa 1017 1990}%
\special{pa 1003 2010}%
\special{pa 1070 1990}%
\special{fp}%
%
\special{pn 8}%
\special{ar 945 1915 25 80  1.5707963  4.7123890}%
%
\special{pn 8}%
\special{pa 1025 1700}%
\special{pa 1160 1345}%
\special{fp}%
\put(11.6500,-12.3000){\makebox(0,0){$(i_0,j_0)$}}%
%
\special{pn 8}%
\special{pa 2270 1600}%
\special{pa 2070 1600}%
\special{fp}%
%
\special{pn 8}%
\special{pa 1600 600}%
\special{pa 1600 600}%
\special{fp}%
%
\special{pn 8}%
\special{pa 2105 1700}%
\special{pa 2220 1700}%
\special{fp}%
\special{sh 1}%
\special{pa 2220 1700}%
\special{pa 2153 1680}%
\special{pa 2167 1700}%
\special{pa 2153 1720}%
\special{pa 2220 1700}%
\special{fp}%
%
\special{pn 8}%
\special{pa 2105 1700}%
\special{pa 2005 1530}%
\special{fp}%
%
\special{pn 8}%
\special{pa 2210 1844}%
\special{pa 2090 1844}%
\special{fp}%
\special{sh 1}%
\special{pa 2090 1844}%
\special{pa 2157 1864}%
\special{pa 2143 1844}%
\special{pa 2157 1824}%
\special{pa 2090 1844}%
\special{fp}%
%
\special{pn 8}%
\special{ar 2220 1770 30 70  4.7123890  1.5707963}%
%
\special{pn 8}%
\special{pa 1890 1530}%
\special{pa 2010 1530}%
\special{fp}%
\special{sh 1}%
\special{pa 2010 1530}%
\special{pa 1943 1510}%
\special{pa 1957 1530}%
\special{pa 1943 1550}%
\special{pa 2010 1530}%
\special{fp}%
%
\special{pn 8}%
\special{pa 1990 1990}%
\special{pa 2095 1845}%
\special{fp}%
%
\special{pn 8}%
\special{pa 1990 1990}%
\special{pa 1870 1990}%
\special{fp}%
\special{sh 1}%
\special{pa 1870 1990}%
\special{pa 1937 2010}%
\special{pa 1923 1990}%
\special{pa 1937 1970}%
\special{pa 1870 1990}%
\special{fp}%
%
\special{pn 8}%
\special{pa 3410 1700}%
\special{pa 3525 1700}%
\special{fp}%
\special{sh 1}%
\special{pa 3525 1700}%
\special{pa 3458 1680}%
\special{pa 3472 1700}%
\special{pa 3458 1720}%
\special{pa 3525 1700}%
\special{fp}%
%
\special{pn 8}%
\special{pa 3285 1820}%
\special{pa 3165 1820}%
\special{fp}%
\special{sh 1}%
\special{pa 3165 1820}%
\special{pa 3232 1840}%
\special{pa 3218 1820}%
\special{pa 3232 1800}%
\special{pa 3165 1820}%
\special{fp}%
%
\special{pn 8}%
\special{pa 3415 1700}%
\special{pa 3315 1530}%
\special{fp}%
\put(32.6000,-17.0000){\makebox(0,0){$0$}}%
%
\special{pn 8}%
\special{pa 3400 1820}%
\special{pa 3280 1820}%
\special{fp}%
\special{sh 1}%
\special{pa 3280 1820}%
\special{pa 3347 1840}%
\special{pa 3333 1820}%
\special{pa 3347 1800}%
\special{pa 3280 1820}%
\special{fp}%
%
\special{pn 8}%
\special{ar 3530 1760 30 60  4.7123890  1.5707963}%
%
\special{pn 8}%
\special{pa 3195 1530}%
\special{pa 3315 1530}%
\special{fp}%
\special{sh 1}%
\special{pa 3315 1530}%
\special{pa 3248 1510}%
\special{pa 3262 1530}%
\special{pa 3248 1550}%
\special{pa 3315 1530}%
\special{fp}%
%
\special{pn 8}%
\special{pa 920 1600}%
\special{pa 920 1000}%
\special{fp}%
%
\special{pn 8}%
\special{pa 2070 1600}%
\special{pa 2070 1000}%
\special{fp}%
%
\special{pn 8}%
\special{pa 3580 1600}%
\special{pa 3380 1600}%
\special{fp}%
%
\special{pn 8}%
\special{pa 3380 1600}%
\special{pa 3380 1000}%
\special{fp}%
%
\special{pn 8}%
\special{pa 3580 2100}%
\special{pa 3580 1600}%
\special{fp}%
%
\special{pn 8}%
\special{pa 2270 2100}%
\special{pa 2270 1600}%
\special{fp}%
%
\special{pn 8}%
\special{pa 1120 2100}%
\special{pa 1120 1600}%
\special{fp}%
\put(9.2000,-9.1000){\makebox(0,0){$(a)$}}%
\put(21.2000,-9.1000){\makebox(0,0){$(b)$}}%
\put(33.8000,-9.1000){\makebox(0,0){$(c)$}}%
\put(19.3000,-17.0500){\makebox(0,0){$0$}}%
%
\special{pn 8}%
\special{pa 3530 1820}%
\special{pa 3330 1820}%
\special{fp}%
\put(8.0000,-17.0000){\makebox(0,0){$0$}}%
%
\special{pn 8}%
\special{pa 810 1530}%
\special{pa 560 1685}%
\special{fp}%
\put(4.4000,-17.3000){\makebox(0,0){$n_{2s-1}$}}%
\put(16.4000,-22.0000){\makebox(0,0){$n_{2s}$}}%
%
\special{pn 8}%
\special{pa 1920 1990}%
\special{pa 1635 2090}%
\special{fp}%
%
\special{pn 8}%
\special{pa 1960 1532}%
\special{pa 2260 1320}%
\special{fp}%
\put(24.4000,-12.4000){\makebox(0,0){$n_{2s-1}$}}%
%
\special{pn 8}%
\special{pa 3390 2121}%
\special{pa 3510 2121}%
\special{fp}%
\special{sh 1}%
\special{pa 3510 2121}%
\special{pa 3443 2101}%
\special{pa 3457 2121}%
\special{pa 3443 2141}%
\special{pa 3510 2121}%
\special{fp}%
%
\special{pn 8}%
\special{pa 3190 2121}%
\special{pa 3390 2121}%
\special{fp}%
\put(8.0000,-18.4000){\makebox(0,0){$0$}}%
\put(8.0000,-19.9000){\makebox(0,0){$0$}}%
\put(19.3000,-18.4500){\makebox(0,0){$0$}}%
\put(21.3000,-19.9500){\makebox(0,0){$0$}}%
%
\special{pn 8}%
\special{pa 3290 2124}%
\special{pa 3326 2119}%
\special{pa 3354 2124}%
\special{pa 3368 2146}%
\special{pa 3373 2181}%
\special{pa 3379 2216}%
\special{pa 3392 2243}%
\special{pa 3417 2258}%
\special{pa 3450 2264}%
\special{pa 3500 2264}%
\special{fp}%
%
\special{pn 8}%
\special{pa 3488 2264}%
\special{pa 3500 2264}%
\special{fp}%
\special{sh 1}%
\special{pa 3500 2264}%
\special{pa 3433 2244}%
\special{pa 3447 2264}%
\special{pa 3433 2284}%
\special{pa 3500 2264}%
\special{fp}%
%
\special{pn 4}%
\special{sh 1}%
\special{ar 3255 1820 16 16 0  6.28318530717959E+0000}%
\special{sh 1}%
\special{ar 3255 2120 16 16 0  6.28318530717959E+0000}%
\put(29.5500,-23.2000){\makebox(0,0){$n_{2s+1}$}}%
%
\special{pn 8}%
\special{pa 3005 2270}%
\special{pa 3255 2120}%
\special{fp}%
\special{pa 3255 1820}%
\special{pa 3055 1920}%
\special{fp}%
\put(28.5500,-20.2000){\makebox(0,0){$n_{2s}$}}%
\end{picture}}%
\caption{The case of $k_2=\alpha(K)+1$.}
\label{alternatfig}
\end{center}
\end{figure}
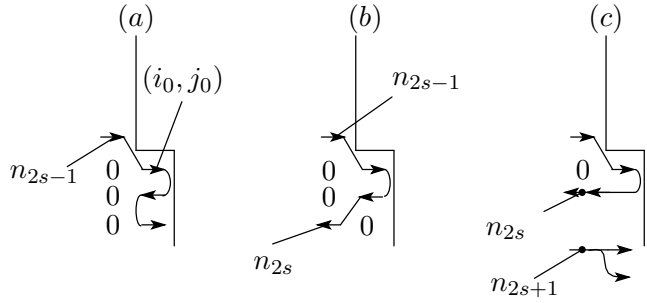
\end{proof}

\begin{cor}
\label{kk-1}
Suppose that $K$ is a lens space knot in an $\lz$ or $K_{p,k}$ in $Y_{p,k}$ with parameter $(p,k,k+1)$.
Then, $\alpha(K)=k$ holds.
\end{cor}
\begin{proof}
We use the $X$-function and $X'$-function, which is obtained by exchanging $k$ and $k_2$.

From the Theorem~\ref{them2}, there exist integers $s_1,s_2$ with $1<s_1\le s_2\le s$ such that
$k=n_1-n_{2s_1-1}$ and $k_2=n_1-n_{2s_2-1}$ or $=n_1-n_{2s_2-1}+1$.

If $s_1<s_2$, then 
$k=n_1-n_{2s_1-1}\le n_1-n_{2s_2-1}-2\le k_2-2=k-1$.
This is contradiction.
Thus $s_1=s_2$ holds.
As a result we have $k_2=n_1-n_{2s_2-1}+1$.
This is the case of (1) in Theorem~\ref{them2}, thus $s_1=s$, i.e., 
$k=\alpha(K)$ holds.

\end{proof}
Here we prove the following lemma.
\begin{lem}
\label{0number}
Let $(p,k)$ be a surgery parameter for lens surgery in an $\lz$.
Then for any column of $dA$-matrix, there exists an integer $m$ such that the number of $0$s between the two pairs of the adjacent $\{-1,1\}$ is $m$ or $m+1$.
\end{lem}
\begin{proof}
From the computation of $dA$-matrix we have the following:
$$dA_{i_0,j}=dA(i_0+ejk)=\begin{cases}-1&[i_0q_2+jk_2]_p\in I_{k_2}\\1&[i_0q_2+jk_2]_p\in I_{-k_2}\\0&\text{otherwise.}\end{cases}$$
Thus, the number of sequent $0$s in $dA$-matrix in a vertical line is determined by the sequence $[nk_2]_p\not\in I_{k_2}\cup I_{-k_2}$.
\end{proof}
Here we prove Theorem~\ref{thirdfourth}.
\begin{proof}
Let $K$ be a lens space knot in $\lz$ or $K_{p,k}$ in $Y_{p,k}$.
Suppose that $n_i$ is a non-zero sequence with at least four trerms.
If $(n_1,n_2,n_3,n_4)=(g,g-1,g',g'-r)$ with $r\ge 4$ and $g'<g-1$, then
by using Theorem~\ref{them2} and Theorem~\ref{3th}, $n_3+1=n_2$ holds.
From Corollary~\ref{trefoilorollary}, this case does not exists.

We suppose $K$ is a lens space knot in an $\lz$ and $r=3$.
If $p-k_2\ge 2g+1$, then since the $m$ in Lemma~\ref{0number} is $0$, and $k_2=3$ or $4$ holds.
The left two pictures in Figure~\ref{34} are non-zero curves of $k_2=3$ or $4$.
However, each of the pictures does not describe any non-zero curve of lens surgery.
Because, in the case of $k_2=3$, we cannot connect the curve as a connected curve
and in the case of $k_2=4$, the curve does not have a symmetry about a point (Proposition~\ref{symmetry}).

If $p-k_2<2g+1$, then the non-zero curve is the right picture in Figure~\ref{34}.
Since the number $m$ in Lemma~\ref{0number} is $0$, then $(p,k_2)=(2g+3,7),(2g+3,6),(2g+2,6)$, or $(2g+2,5)$ holds.
These cases are realized by any lens space surgery in an $\lz$.
\begin{figure}[htbp]
\begin{center}
\input{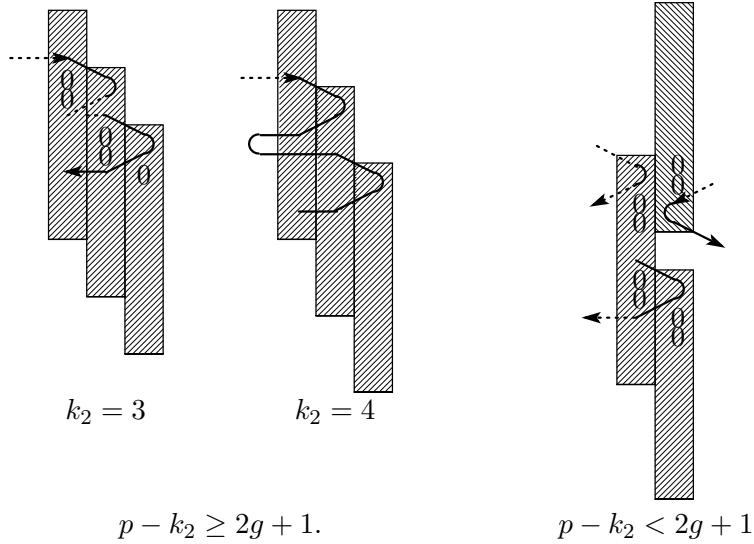}
\caption{The case of $(n_1,n_2,n_3,n_4)=(g,g-1,g',g'-3)$ and $k_2=3$ or $4$.}
\label{34}
\end{center}
\end{figure}
\end{proof}
\begin{example}
$K=T(u,u+1)$ for $u\ge 2$ satisfies this condition.
In this case we have $k=u=\alpha(T(u,u+1))$ and $k_2=u+1$.
The expansion of $\Delta_{T(u,u+1)}(t)$ is as follows:
$$\Delta_{T(u,u+1)}(t)=t^d-t^{d-1}+t^{d-u}-t^{d-u-2}+\cdots,$$
where $g=\frac{u(u-1)}{2}$.
The adjacent sequence is $(g,g-u)$.
\end{example}
\begin{example}
The parameter $(19,7,8)$ is a lens surgery parameter of a lens space knot $Pr(-2,3,7)$.
The adjacent sequence is $AS=(5,2,0,-2)$.
The Alexander polynomial is expanded as follows:
$$\Delta_{Pr(-2,3,7)}(t)=t^5-t^4+\cdots+t^{5-7}-t^{5-9}+\cdots.$$
Here $\alpha(Pr(-2,3,7))=7$ holds and $s=s_1=s_2=4$.
Thus, we know that the inequality $k_2\le \alpha+1$ is best possible.

The Alexander polynomial of $K_{61,13}$ in $S^3$ is expanded as follows:
$$\Delta_{K_{61,13}}(t)=t^{22}-t^{21}+\cdots+t^{22-13}-t^{22-15}+\cdots.$$
Thus $\alpha(K_{61,13})=13$ holds.
\end{example}
\subsection{Lens surgeries with $g(K)\le 5$ or with at most $7$ non-zero coefficients.}
\label{57nonzerocoeff}
We classify lens space knots in an $\lz$ or $K_{p,k}$ in $Y_{p,k}$ with $g(K)\le 5$ or at most $7$ non-zero coefficients.
Before proving the theorems below we introduce some results.
Rasmussen proved the following:
\begin{theorem}[Rasmussen \cite{ras}]
\label{rasmussen}
Let $K$ be a knot in an L-space.
Suppose that some integral surgery
on $K$ yields a homology sphere $Y$. If $2g(K)<p+1$, then $Y$ is an L-space, while
if $2g(K)>p+1$, then $Y$ is not an L-space.
\end{theorem}
This means that if $p$ is a lens surgery slope of the lens space knot $K$ in a non-L-space homology sphere, then 
$$2g(K)\ge p+1$$
holds.

Here we prove the following:
\begin{prop}
\label{103}
Let $K_{p,k}$ be a lens space knot in a non-L-space $\zhs$.
Then $g(K_{p,k})\ge 6$ holds.
If $g(K_{p,k})=6$, then $(p,k)=(10,3)$ and $K_{10,3}$ lies in $\Sigma(2,3,7)$.
\end{prop}

\begin{proof}
When the slope is $p\le9$, any $K_{p,k}$ lies in $S^3$ or $\Sigma(2,3,5)$.
See the list in \cite{B} and \cite{tan1}.
Hence, any $K_{p,k}$ in a non-L-space $\zhs$ satisfies $p\ge 10$.
From Theorem~\ref{rasmussen}, those knots satisfy $g(K_{p,k})\ge \frac{p+1}{2}>5$.
By using the formula (\ref{ISTformulaitself}), we have $g(K_{10,3})=6$ and $K_{10,3}\subset \Sigma(2,3,7)$.
This example is all the $K_{p,k}$ in a non-L-space homology sphere $10\le p\le 11$ by checking such $Y_{p,k}$.
Any other $K_{p,k}$ with $p\ge12$ in a non-L-space homology sphere has $g(K_{p,k})\ge \frac{p+1}{2}>6$.
\end{proof}
The knot $K_{10,3}$ lies in $\Sigma(2,3,7)$ and satisfies $\Sigma(2,3,7)_{10}(K_{10,3})=L(10,1)$.
Check the case of $\ell=-1$ in $\text{A}_1$ or $\text{A}_2$ type in {\sc Table} 3 in \cite{tan4}.

Let $K_{p,k}$ be a knot in a non-L-space homology sphere.
Theorem~\ref{rasmussen} and Proposition~\ref{103} imply $g(K_{p,k})\ge 6$.
Here we list surgery parameters and realizations
of a lens space knot $K_{p,k}$ in an $\lz$ with $g(K)\le 5$.
\begin{theorem}
\label{dle5}
Let $K$ be a lens space knot in an $\lz$ with $g(K)\le 5$.
Then the parameter can be realized by either of the following lens space knots:
$$T(2,3),\ T(2,5),\ T(2,7),\ T(3,4),\ T(2,9),\ T(3,5),\ T(2,11),\text{ or }Pr(-2,3,7)$$
\end{theorem}
We classify the lens space knots in an $\lz$ or $K_{p,k}$ with at most $7$ non-zero coefficients of the 
Alexander polynomial.
\begin{theorem}
\label{atmost7}
Let $K$ be a lens space knot in an $\lz$ or $K_{p,k}$ in $Y_{p,k}$ with at most $7$ non-zero coefficients.
Then the lens space parameters can be realized by either of the following lens space knots:
$$T(2,3),\ T(2,5),\ T(4,3),\ T(2,7),\ T(3,5),\text{ or }T(4,5).$$
\end{theorem}
We can also find concrete lens surgery parameters from the realization.
\begin{prop}
\label{torusknot}
Let $K$ be a lens space knot in an $\lz$ with $(p,k)$.
Let $\gamma$ be a non-zero curve and $\gamma'$ the non-zero curve obtained by the $(k,-e)$ parallel translation.
Then, the parameter $(p,k)$ can be realized by a surgery parameter of a torus knot surgery in $S^3$,
if and only if $\gamma$ and $\gamma'$ are included in $\mathcal{N}^{A,m}$ and $\mathcal{N}^{A,m+e}$ respectively.
\end{prop}
\begin{proof}
The parallel transformation by a vector ${\bf v}_1=(1,-k_2)$ acts on $\mathcal{N}^{A,m}$
and the non-zero curve in $\mathcal{N}^{A,m}$.
The vector ${\bf v}_2=(0,ep)$ gives a congruence map to next non-zero region.
Lens space surgeries of $(r,s)$-torus knot in $S^3$ ($r<s$) are $p=rs\pm1$ in \cite{Mo}.
If $r>2$, then the parameter is $(rs-e,r,s)$.
If $r=2$, then the parameter is $(2s+1,2,s)$ $(e=-1)$ or $(2s-1,2,s-1)$ $(e=1)$.
Thus in both cases we have $p=kk_2-e$.

This equality is equivalent to $(k,-e)=(k,-kk_2+p)=k{\bf v}_1+e{\bf v}_2$.
Thus $(k,-e)$ moves the next non-zero regions in $\mathcal{N}^{A}=\cup_{m\in {\mathbb Z}}\mathcal{N}^{A,m}$ as $\mathcal{N}^{A,m}\mapsto \mathcal{N}^{A,m+1}$.
\end{proof}
In the case of $g\le 3$, the lens surgery polynomials are all torus knot polynomials.
Because, if the polynomial is not $T(2,n)$-torus knot, then the non-zero terms are at most five and $\alpha\le 3$ holds.
From Theorem~\ref{a22a2cor} and Corollary~\ref{thenumber}, $p\le 17$ and $k_2\le 5$.
Thus lens space surgery on this restriction is realized by some torus knot surgery.

\begin{lem}
\label{36lemma}
If $K$ is a lens space knot in an $\lz$ or $K_{p,k}$ with $3\le k\le k_2\le 6$, then the parameter $(p,k)$ is either of 
the following:
$$(31,5),(29,5),(26,5),(24,5),(13,5),(12,5),(25,4),(23,4),(21,4),(19,4),$$
$$(19,3),(17,4),(17,3),(15,4),(13,3),(11,3),(14,3),(16,3),(8,3),(10,3).$$
\end{lem}
\begin{proof}
If $k_2=5$ holds, then $k\le k_2=5$.
If $k=5$, then $p$ is the divisor of $26$ or $24$.
Other cases are all listed similarly.
\end{proof}
Suppose that $K$ is a lens space knot in an $\lz$ with the parameters $(26,5)$, $(24,5)$, $(13,5)$, $(12,5)$, $(17,4)$, $(15,4)$, $(19,3)$, $(17,3)$ and $(10,3)$.
Then some $a_i(K)$ (in Theorem~\ref{tangeprop}) is not absolutely less than or equal to $1$.
If one of others is not realized by a torus knot, it is realized by either of the following lens space knots:
\begin{itemize}
\item $K_{25,4}=(T(2,3))_{13,2}$ (the $(13,2)$-cable knot of $T(2,3)$)
\item $K_{23,4}=(T(2,3))_{11,2}$ (the $(11,2)$-cable knot of $T(2,3)$)
\item $K_{8,3}$ in $Y_{8,3}=\Sigma(2,3,5)$.
\end{itemize}

Here we classify lens space surgery with $g(K)=4$.
\begin{prop}
\label{g=4}
Let $K$ be a lens space knot in an $\lz$ with $g(K)=4$, then the parameters are $(17,2)$, $(19,2)$, or $(8,3)$.
The parameters are realized by $T(2,9)$ or
$K_{8,3}$ in $\Sigma(2,3,5)$.
\end{prop}
\begin{proof}
Let $(p,k)$ be a lens surgery parameter of a lens space knot $K$ with genus 4.
Suppose that $\Delta_{K}(t)$ does not equal to a $(2,r)$-torus knot polynomial.
Then the half non-zero sequence is $(4,3,1,0)$ or $(4,3,0)$ from Corollary~\ref{trefoilorollary}.
The $\alpha$-index of the polynomial is $5$ or $4$.
Thus $k_2\le 6$ holds from Theorem~\ref{them2}.
The parameters with $k_2\le 6$ are already classified above.
Then the parameters with $g(K)=4$ are realized by torus knots or $K_{8,3}$.
$NS_h(K_{8,3})=(4,3,1,0)$.
Since $(4,3,0)$ is not a non-zero sequence of a torus knot polynomial, this case does not occur.
\end{proof}
Note that the equality of non-zero sequence $NS_h(K_{8,3})=NS_h(K_{14,3})$ holds.
This proposition also proves that the sequence $(4,3,0)$ is not non-zero sequence of lens space knot in an $\lz$.

The author in \cite{tan2} proved the following:
\begin{theorem}[Theorem 16 in \cite{tan2}]
\label{15tan2}
If a knot $K$ satisfying $\Delta_K(t)=t^n-1+t^{-n}$ admits lens surgery, then
$n=1$ and moreover $K$ is the trefoil knot.
\end{theorem}
This theorem was in \cite{tan2} proved by a longer argument of coefficients, however, 
it is an immediate application of Corollary~\ref{main}.
In the present paper, we can continue to discuss the existence of lens surgery in terms of the number of non-zero coefficients of the Alexander polynomial.

For example, the polynomial $t^n-t^{n-1}+1-t^{-n+1}+t^{-n}$ satisfies the condition $n_2=d-1$ in Corollary~\ref{main},
however, there exists an upper bound of $n$ such that it is a lens surgery polynomial in an $\lz$ or $\Delta_{K_{p,k}}(t)$.
\begin{cor}
\label{5coeff}
Let $K$ be a lens space knot in an $\lz$ or $K_{p,k}$ with $5$ non-zero coefficients of
the Alexander polynomial.
Then $\Delta_K(t)=\Delta_{T(2,5)}(t)$, or $\Delta_{T(4,3)}(t)$.

In other words, if the lens surgery polynomial is of form $t^n-t^{n-1}+1-t^{-n+1}+t^{-n}$, then
$n=2$ or $3$ holds.

These polynomials do not coincide with $\Delta_{K_{p,k}}(t)$ in any non-L-space $\zhs$.
\end{cor}
\begin{proof}
Suppose that $\Delta_K(t)\neq \Delta_{T(2,r)}(t)$.
Then $\alpha(K)=n\ge 3$ holds.
From Corollary~\ref{thenumber}, the inequality $k_2\le 5$ holds.
Since p.288 (A) in \cite{tan1} $k<k_2$ holds.
Then $k_2=k+1=n+1\le 5$ holds from Theorem~\ref{them2}.
If $g(K)=4$, then the non-zero sequence is $(4,3,1,0)$ only.
Thus we have $g(K)\le 3$, that is, $g(K)=3$ from the condition above.
This case is realized by $T(3,4)$.
\end{proof}
\begin{cor}
\label{g=5}
Let $K$ be a lens space knot in an $\lz$ with $g(K)=5$ with parameter $(p,k)$.
Then $(p,k)$ is realized by a lens space surgery of $T(2,11)$ and $Pr(-2,3,7)$.
\end{cor}
\begin{proof}
We assume that $\Delta_K(t)\neq \Delta_{T(2,r)}(t)$.
From Theorem~\ref{2ntorusknot} and \ref{k22g1}, since $3\le k< k_2\le 8$ or $k=k_2=3$ holds.
In the case $k=k_2=3$, we have $p=8$ and $g(K)\neq 5$.
From Corollary~\ref{trefoilorollary} and \ref{5coeff}, the non-zero sequence is one of the following:
$$NS_h(K)=(5,4,1,0),(5,4,2,1,0),(5,4,2,0),(5,4,3,2,0).$$

If $NS_h(K)=(5,4,3,2,0)$ or $(5,4,2,0)$, then the $\alpha(K)$ is $5$ or $3$ respectively.
These cases are $k_2\le 6$ and are already classified in Lemma~\ref{36lemma}.
However any of these cases is not $g(K)=5$.
If $NS_h(K)=(5,4,1,0)$ then $\alpha(K)=6$ and $k_2\le 7$ due to Corollary~\ref{thenumber}.
If $k_2\le 6$, then by the same arguments as above $g(K)\neq 5$.
Thus $k_2=7$ holds.
From p.288 (A) in \cite{tan1} $k<k_2$ holds.
Then $k=4$ holds due to Theorem~\ref{them2}.
Thus $p$ is a divisor of $27$ or $29$.
Since $2k_2=14<p$ holds, $p=27$, $29$.
$(p,k)=(27,4)$ and $(29,4)$ are realized by $T(4,7)$.
However $g(K)\neq 5$ holds.

Suppose that $NS_h(K)=(5,4,2,1,0)$ and $\alpha(K)=7$.
If $k_2\le 6$, then the genus is not $5$ as described above.
Then $k_2=7,8$ holds.
From the above description $k<k_2$ holds.

Suppose that $k_2=7$.
If $k=3$, then there exists $dA=2$ point as in the leftmost picture in Figure~\ref{-237}.
If $k=5$, then $p$ is a divisor of $34$, $36$ and satisfies $p>2k_2=14$.
Thus $p=17,18,34$, or $36$ holds.
The case of $(p,k)=(17,5)$ does not satisfy the flat coefficients for the formula of $a_i(K)$ in Proposition~\ref{tangeprop}.
Other cases are $T(5,7)$ or $Pr(-2,3,7)$.
The former case does not satisfy $g(K)=5$.

Suppose that $k_2=8$.
The $k=3$ does not occur by the same reason as above.
If $k=5$, then $p$ is a divisor of $39$ or $41$ and satisfies $p>2k_2=16$.
Then we have $p=39,41$.
This case is realized by the $(5,8)$-torus knot surgery.
However $g(T(5,8))\neq 5$.
If $k=7$, then $p$ is a divisor of $55$ or $57$ and satisfies $p>2k_2=16$.
Then $p=19,55,57$.
The case of $(p,k)=(19,7)$ is realized by $Pr(-2,3,7)$ with $g(K)=5$.
The case of $(p,k)=(55,7)$ or $(57,7)$ is realized by $T(7,8)$ with $g(K)=21$.

Thus lens space knot in an $\lz$ with $g(K)=5$ is realized by $T(2,11)$ or $Pr(-2,3,7)$.
\begin{figure}[htbp]
\begin{center}
\input{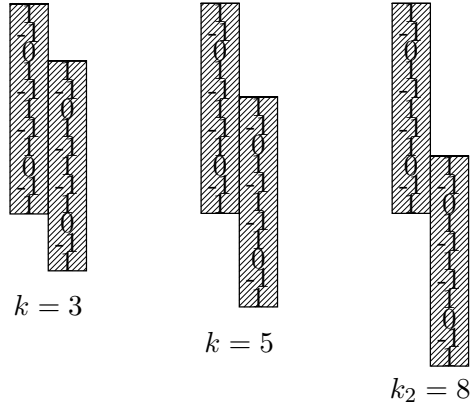}
\caption{The case of $NS_h(K)=(5,4,2,1,0)$.}
\label{-237}
\end{center}
\end{figure}

\end{proof}
Next, we classify lens space knots with $7$ non-zero coefficients.
\begin{cor}
\label{7coeff}
If the Alexander polynomial $\Delta_K(t)$ of lens space knot in an $\lz$ or $K_{p,k}$ with $7$ non-zero coefficients, 
then $\Delta_K(t)=\Delta_{T(2,7)}(t)$, $\Delta_{T(3,5)}(t)$, or $\Delta_{T(4,5)}(t)$.

If $K$ is a lens space knot in an $\lz$ or $K_{p,k}$ with
\begin{equation}
\label{7first}
\Delta_K(t)=t^g-t^{g-1}+t-1+t^{-1}-t^{-g+1}+t^{-g}
\end{equation}
\begin{equation}
\label{7second}
\Delta_K(t)=t^g-t^{g-1}+t^2-1+t^{-2}-t^{-g+1}+t^{-g},
\end{equation}
then when (\ref{7first}), $g=3,4$ and when (\ref{7second}), $g=6$.
\end{cor}
\begin{proof}
Suppose that $K$ is a lens space knot in an $\lz$ or $K_{p,k}$ and $\Delta_K(t)\neq \Delta_{T(2,2g+1)}(t)$.

If the lens surgery polynomial is of the form (\ref{7first}), then $g\ge 4$ and $\alpha(K)=g+1$.
If $k=k_2$, then $k_2\le 3$ and $k=k_2=1$ or $3$.
These cases do not satisfy the Alexander polynomial condition.
Thus we have $k<k_2$.
If $g\ge 5$ and $(k,k_2)=(g+1,g+2)$, $(g-1,g+2)$ $(g-1,g+1)$, then the non-zero curve is not disconnected.
Thus $g\le 4$ holds.
Then $(k,k_2)=(3,6),(3,5),(5,6)$.
The candidate $p$ is a divisor of $17,19,14,16,29,31$.
Among these candidates, genus $4$ cases are $(p,k,k_2)=(14,3,5)$ or $(16,3,5)$ only.
This is the $\Delta_{T(3,5)}(t)$.

If the lens surgery polynomial is of form (\ref{7second}), then $\alpha(K)=g-2$.
If $g\ge8$ then the non-zero curve is disconnected in the same reason as above.
Thus, $g\le 7$ holds.
Searching connected curves, we can find the possibilities of non-zero curves as in the pictures in Figure~\ref{7nonzero}.
The next is the table of the 4 non-zero curves.
$$\begin{array}{|c|c|c|c|}\hline
k\text{ or }k_2&g&\alpha\\\hline
4&6&4\\\hline
5&7&5\\\hline
5&6&4\\\hline
4&5&3\\\hline
\end{array}
$$
By using the inequality (\ref{alphainequation}) in Theorem~\ref{them2}, considering all the cases $(p,k,k_2)$,
we get $(p,k,g)=(19,4,6),(21,4,6)$ and $\Delta_K(t)=\Delta_{T_{4,5}}(t)$.
These cases are torus knot surgeries.
\begin{figure}[htbp]
\begin{center}
\input{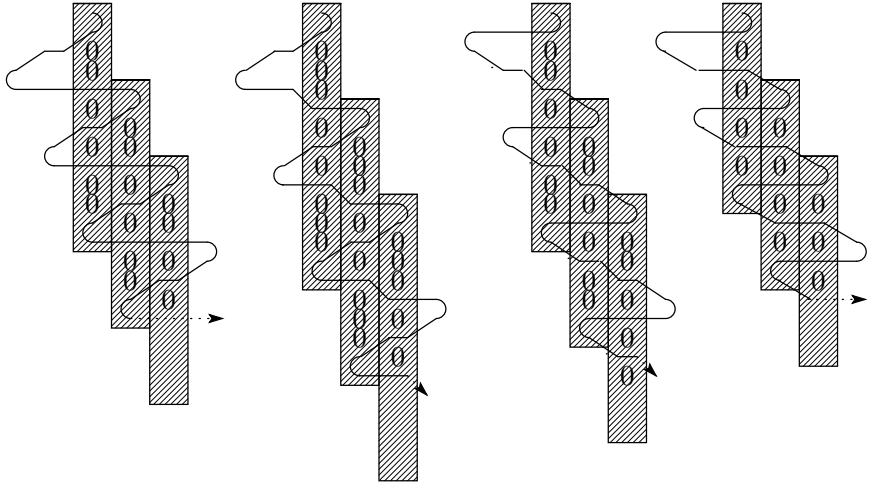}
\caption{The patterns with $g=6,7,6,5$ respectively.}
\label{7nonzero}
\end{center}
\end{figure}

\end{proof}
Here we prove Theorem~\ref{dle5} and \ref{atmost7}.
\begin{proof}
Suppose that $\Delta_K(t)\neq \Delta_{T(2,r)}(t)$.
If $g\le 3$, then then 3 and 5 non-zero coefficients classification, such parameters are $(11,3)$ or $(13,3)$ and are realized by $T(4,3)$.
If $g=4$ or $5$, then the parameters are $(8,3)$, $(18,5)$ or $(19,7)$ and are realized by $K_{8,3}$, $Pr(-2,3,7)$
due to Proposition \ref{g=4} and Corollary \ref{g=5}.
\end{proof}
\begin{proof}
From Theorem~\ref{15tan2} and Corollary~\ref{5coeff} and \ref{7coeff}, this theorem follows.
\end{proof}

The next classification of lens surgeries should be done for the surgeries with $9$ non-zero coefficients.
This is left for readers.
\begin{prob}
Classify the Alexander polynomial with $9$ non-zero coefficients.
The polynomials are of the form:
$$t^n-t^{n-1}+t^{m}-t^{m-1}+1+t^{-m}-t^{-m+1}-t^{-n}+t^n,$$

$$t^n-t^{n-1}+t^{m}-t^{m-2}+1+t^{-m}-t^{-m+2}-t^{-n}+t^n$$
or
$$t^n-t^{n-1}+t^{m}-t^{m-3}+1+t^{-m}-t^{-m+3}-t^{-n}+t^n.$$
\end{prob}
\subsection{Lens surgeries with $2g(K)-4\le k_2\le 2g(K)-2$.}
\label{K2n2g-42g-1}
In Theorem~\ref{k22g1}, we classify lens space knots in an $\lz$ or $K_{p,k}$ in $Y_{p,k}$ in the cases of $2g(K)-1\le k_2$.
Here we classify lens space knots in an $\lz$ with $2g(K)-4\le k_2\le 2g(K)-2$.
\begin{theorem}
\label{234class}
If a lens space knot $K$ in an $\lz$ satisfies $2g(K)-4\le k_2\le 2g(K)-2$, then 
the lens surgery parameters are $(11,3)$, $(14,3)$, $(19,7)$ and are realized by 
$$T(3,4), T(3,5)\text{ or }Pr(-2,3,7),$$
respectively.
\end{theorem}
These lens surgeries in the theorem just correspond to the ones with the half non-zero sequence
$$(g,g-1,g-3,g-4,\cdots,2,1,0).$$

The following table is the classification of $NS$ with $2g(K)-k_2\le 4$.
$$\begin{array}{|c|c|c|}\hline
2g(K)-k_2&(p,k)&NS\\\hline
-1&(4d+3,2)&NS(T(2,2d+1))\\\hline
0&(4d+1,2)&NS(T(2,2d+1))\\\hline
1&\text{no}&\text{no}\\\hline
2&(19,7)&NS(Pr(-2,3,7))\\\hline
2&(11,3)&NS(T(3,4))\\\hline
3&(14,3)&NS(T(3,5))\\\hline
4&\text{ no }&\text{ no }\\\hline
\end{array}$$

Theorem~\ref{234class} is proven by decomposing it into the following three propositions.
\begin{prop}
\label{2g2cl}
Let $K$ be a lens space knot in an $\lz$ with genus $g$.
If $K$ satisfies $k_2=2g-2$, then 
the lens surgery parameters are $(11,3)$, $(19,7)$ and are realized by $T(3,4)$, or $Pr(-2,3,7)$ respectively.
\end{prop}
\begin{proof}
In the case of $2g-2=k_2$, the half non-zero sequence $NS_h(K)$ is 
$$(g,g-1,g-2,\cdots,0)\text{ or }(g,g-1,g-3,g-4,\cdots,2,1,0).$$
The former case corresponds to the surgery parameter realized by $(2,r)$-torus knot due to Theorem~\ref{2ntorusknot}.

We consider the latter case.
In the case of $g\ge 5$, there exists a sequent values $\cdots,-1,1,-1,1,\cdots$ for $dA$ on a column.
Hence, from Lemma~\ref{0number}, the $m$ in the statement is $0$.

By Lemma~\ref{0number}, we have $p-k_2=2g+1$ and $p-k_2=2g+2$.
See Figure~\ref{e-1case}.
Since the latter case does not satisfy $(p,k_2)=1$, we have $(p,k_2)=(4g-1,2g-2)$.

Consider the 3 lattice points on a vertical line with $dA=0$ among a period.
The distances of these 3 points are $5$, $2g-3$ and $2g-3$, see Figure~\ref{e-1case}.
Let $(i_0,j_0)$ be a lattice point with $i_0q+j_0k_2\equiv k_2+1$.
Then $(i_0,j_0+ek)$ and $(i_0,j_0+2ek)$ are the points satisfying $dA=0$.
Since $j_0,j_0+ek,j_0+2ek$ are $dA=0$ points in the period $p$.
Thus, $k=2g-3$ holds if $g>4$.
Therefore, we have $kk_2=(2g-3)(2g-2)=4g^2-10g+6\equiv3g+3\equiv \pm1\bmod 4g-1$.
Solving this equality, we have $(p,k)=(19,7)$ and $g=5$.
This case is realized by the $(-2,3,7)$-pretzel knot.

In the case of $g\le 4$, we have $g=3$ and $(p,k)=(11,3)$ from Theorem~\ref{dle5}.
\begin{figure}[htbp]
\begin{center}
\input{torus3.tex}
\caption{The $A$-function and $dA$-function in the case of $(p,k_2)=(4g-1,2g-2)$.}
\label{e-1case}
\end{center}
\end{figure}
\end{proof}
Next is the classification of the cases of $2g(K)-3=k_2$.
\begin{prop}
\label{2gk-3k2}
Let $K$ be a lens space knot in $\lz$ with genus $g$.
If $K$ satisfies $k_2=2g-3$, then 
the lens surgery parameter is $(14,3)$ and is realized by $T(3,5)$.
\end{prop}

\begin{proof}
From the condition $k_2=2g-3$, the half non-zero sequence of the Alexander polynomial is
$$(g,g-1,g-3,g-4,\cdots,2,1,0)\text{ or }(g,g-1,g-4,g-5,\cdots,2,1,0).$$
$N_1$ and $N_2$ denote the half non-zero sequences respectively.
The non-zero regions for these cases are the first picture in Figure~\ref{2g-3}.

If $p-k_2<2g+1$ holds, then the case of $NS_h=N_2$ holds and the $A$-function is the picture (a) in Figure~\ref{2g-3}.
This picture $\cup_{m\in {\mathbb Z}}\mathcal{N}^{A,m}$ has the symmetry of the parallel transformation $(2,0)$.
Thus $p=2$ holds.
This case does not occur.
Hence $p-k_2\ge 2g+1$ holds.

Suppose that $p-k_2\ge 2g+1$.
If $NS_h=N_2$, then the $dA$-function is the picture (b).
This picture is inconsistent with Lemma~\ref{0number}.
In fact in the sequence $\{dA_{i_0,j}\}_{j\in {\mathbb Z}}$ there exist parts of the following
$$\cdots, 1,-1,1,-1,\cdots,0,0,\cdots.$$
This violates Lemma~\ref{0number}.

If $NS_h=N_1$ and $g>5$, then the number $m$ in Lemma~\ref{0number} is $0$.
Thus $p-k_2=2g+1$ or $2g+2$, therefore, $(p,k_2)=(4g-2,2g-3),(4g-1,2g-3)$ respectively.
If $(p,k)=(4g-2,2g-3)$ then $p-2k_2=4$ is satisfied, i.e., the points with $dA=0$ have four points among a period (sequent $p$ lattice points).
If $(p,k)=(4g-1,2g-3)$ then $p-2k_2=5$ is satisfied, i.e., the points with $dA=0$ have five points among a period.
The distances of these 4 points (or 5 points) are $3,2g-5,5,2g-5$ (or $3,2g-5,3,3,2g-5$ respectively).
Here, the existence of the distance $3$ implies $3k_2-p=2g+8\le 4$.
However this is contradiction of the condition $g\ge 5$.

Suppose that $NS_h=N_1$ and $g\le 5$.
From Theorem~\ref{dle5} we can find $(p,k,k_2)=(14,3,5)$.

\end{proof}
\begin{figure}[htbp]
\begin{center}
\input{2g-3.tex}
\caption{The $A$-function and $dA$-function in the case of $2g(K)-3=k_2$ and $N_2$.}
\label{2g-3}
\input{2d+12d+2.tex}
\caption{The $A$-function and $dA$-function in the case of $g\ge 5$.}
\end{center}
\end{figure}

\begin{figure}[htbp]
\begin{center}
\input{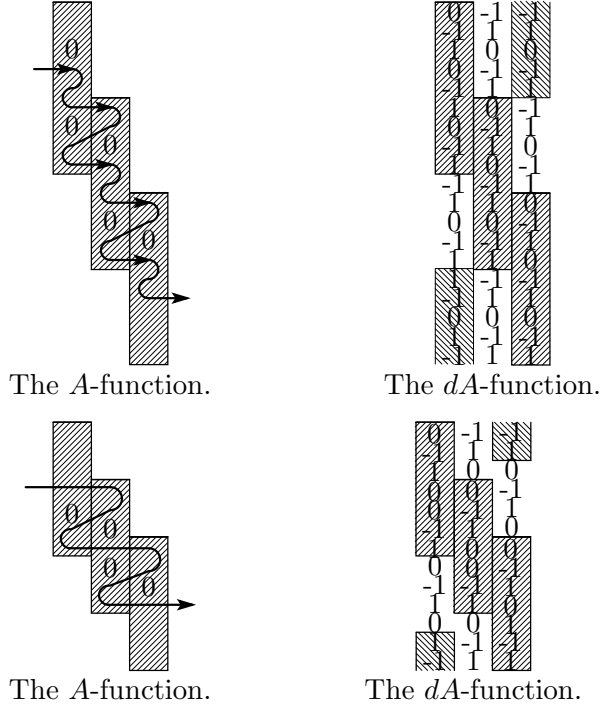}
\caption{The $A$-function or $dA$-function in the case of $NS_h=N_1$ and $g=3$.}
\label{2g-32}
\end{center}
\end{figure}

\begin{prop}
\label{2gk-4k2}
There exists no lens space knot $K$ in an $\lz$ with $2g(K)-4=k_2$ and parameter $(p,k)$.
\end{prop}
\begin{proof}
Let $g$ be a genus of such a knot.
From Corollary~\ref{main} $n_1=g$ and $n_2=g-1$ hold.
If $n_3=g-3$, then we can find a point with $dA=-2$ from the leftmost picture in Figure~\ref{422}.

Thus $n_3=g-2$ and then the parameter is $(p,2,3)$ or $n_4=g-3$ from Corollary~\ref{trefoilorollary}.
The former case does not satisfy $2g-4=k_2$.
We assume $n_4=g-3$.
If $n_5=g-4$ and $g>5$, then from the connectivity of non-zero sequence, $n_6=g-5$ and $\Delta_K(t)=\Delta_{T(2,2g+1)}(t)$.
See the second picture in Figure~\ref{422}.
However, this is contradiction to Theorem~\ref{2ntorusknot}.
\begin{figure}[htbp]
\begin{center}
\input{422.tex}
\caption{}
\label{422}
\end{center}
\end{figure}
If $n_5=g-4$ and $g\le 5$, then from the classification of Theorem~\ref{dle5} there are no such lens space knots.

Thus, consider $n_5=g-5$.
Then from the third picture in Figure~\ref{422} the non-zero sequence is the following sequence
$$(g,g-1,g-2,g-3,g-5,g-6,\cdots).$$
If $g\le 5$, then from Theorem~\ref{dle5}, lens space knots satisfying this do not exist.
Thus we assume $g>5$.
Then we can find the parts of the sequence $\{dA_{i_0,j}\}_{j\in {\mathbb Z}}$
$$-1,1,0,0,0,-1,1,\cdots, -1,1,-1,1,\cdots.$$
See the third and fourth picture in Figure~\ref{422}.
Such a sequence does not satisfy Lemma~\ref{0number}.
\end{proof}
\begin{prop}
Let $K$ be a lens space knot in an $\lz$ with genus $g$.
If the half non-zero sequence is $(g,g-1,g-3,g-4,\cdots,3,2,1,0)$
are $g=3,5$ only.
These cases can be realized by the $(3,4)$-torus knot and $(-2,3,7)$-pretzel knot respectively.
\end{prop}
\begin{proof}
We suppose the case of $k_2<2g-4$.
If $k_2<2g-5$, then we can find the following part in the sequence $\{dA_{i_0,j}\}_{j\in {\mathbb Z}}$
$$0,0,0,\cdots  1,-1,1,-1,\cdots,$$
as in the left two pictures in Figure~\ref{dd1d3}.
\begin{figure}[htbp]
\begin{center}
\input{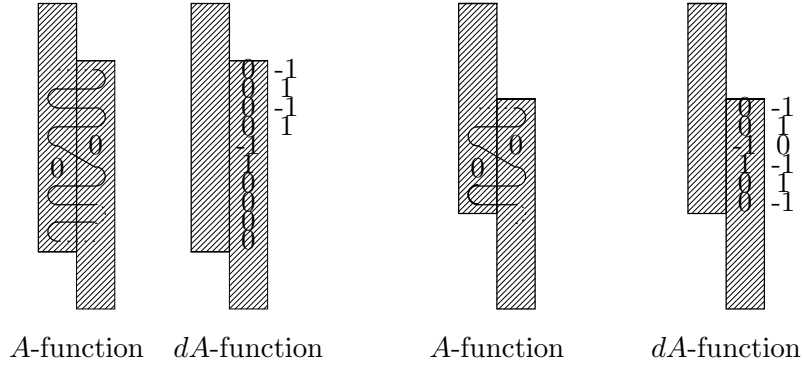}
\caption{The cases of $k_2<2g-5$ and $k_2=2g-5$.}
\label{dd1d3}
\end{center}
\end{figure}
Thus we obtain $k_2=2g-5$.
Suppose that $g>5$.
Then there exists the following part of the sequence $\{dA_{i_0,j}\}_{j\in {\mathbb Z}}$
$$0,0,1,-1,0,0,\cdots, -1,1,-1,0,\cdots,$$
as in the left two pictures in Figure~\ref{dd1d3}.
This is contradiction to Lemma~\ref{0number}.
Thus we have $g\le 5$.
From the classification of Theorem~\ref{dle5}, the Alexander polynomials are $\Delta_{Pr(-2,3,7)}(t)$
or $\Delta_{T(3,4)}(t)$.
The surgery parameters are $(11,3,4)$ or $(19,7,8)$.
\end{proof}
Here we prove Theorem~\ref{234class}.
\begin{proof}
From Proposition \ref{2g2cl}, \ref{2gk-3k2} and \ref{2gk-4k2} the assertion of this theorem follows.
\end{proof}

Finally, we describe the following tables.
\begin{example}
In Table~\ref{nonzeroL1} and \ref{nonzeroL2} we list $K_{p,k}$ with $g(K_{p,k})\le 30$ in non-L-space homology spheres $Y_{p,k}$.
These tables are obtained by the pillowcase method in \cite{tan4}.
\end{example}
\begin{table}[htbp]
\begin{center}
$\begin{array}{|c|c|c|c|c|c|}\hline
g&p&k&k_2&NS_h,AS&\alpha\\\hline
6&10&3&3&(6,5,3,2,0)&6\\
&&&&(6,3,0)&\\\hline
10&17&3&6&(10,9,7,6,4,3,1,0)&11\\
&&&&(10,7,4,1,-1)&\\\hline
12&12&5&5&(12,11,7,6,5,4,2,1,0)&14\\
&&&&(12,7,5,2,0,-2)&\\\hline
12&17&5&7&(12,11,7,6,5,4,2,1,0)&14\\
&&&&(12,7,5,2,0,-2)&\\\hline
12&19&3&6&(12,11,9,8,6,5,3,2,0)&12\\
&&&&(12,9,6,3,0)&\\\hline
13&23&7&10&(13,12,10,9,6,5,3,2,0)&13\\
&&&&(13,10,6,5,3,0)&\\\hline
14&13&5&5&(14,13,9,8,6,5,4,3,1,0)&15\\
&&&&(14,9,6,4,1,-1)&\\\hline
15&15&4&4&(15,14,11,10,7,6,4,2,0)&11\\
&&&&(15,11,7,4)&\\\hline
16&23&5&9&(16,15,11,10,7,5,2,0)&9\\
&&&&(16,11,7)&\\\hline
16&26&3&9&(16,15,13,12,10,9,7,6,4,3,1,0)&17\\
&&&&(16,13,10,7,4,1,-1)&\\\hline
16&26&7&11&(16,15,12,11,9,8,5,4,2,0)&14\\
&&&&(16,12,9,5,2)&\\\hline
16&29&8&11&(16,15,13,12,8,7,5,4,2,1,0)&18\\
&&&&(16,13,8,5,2,0,-2)&\\\hline
18&17&4&4&(18,17,14,13,10,9,6,4,2,0)&12\\
&&&&(18,14,10,6)&\\\hline
18&25&9&11&(18,17,9,8,7,6,4,3,2,1,0)&22\\
&&&&(18,9,7,4,2,0,-2,-4)&\\\hline
18&28&3&9&(18,17,15,14,12,11,9,8,6,5,3,2,0)&18\\
&&&&(18,15,12,9,6,3,0)&\\\hline
19&29&9&13&(19,18,12,11,10,9,6,5,3,2,1,0)&22\\
&&&&(19,12,10,6,3,1,-1,-3)&\\\hline
20&27&5&11&(20,19,15,14,10,8,5,3,0)&10\\
&&&&(20,15,10)&\\\hline
21&35&8&13&(21,20,16,15,13,12,8,7,5,4,3,2,0)&21\\
&&&&(21,16,13,8,5,3,0)&\\\hline
21&38&9&17&(21,20,17,16,12,11,8,7,4,2,0)&17\\
&&&&(21,17,12,8,4)&\\\hline
\end{array}$
\caption{The list of $K_{p,k}$ with $p\le 21$.}
\label{nonzeroL1}
\end{center}
\end{table}
\begin{table}[htbp]
\begin{center}
$\begin{array}{|c|c|c|c|c|c|c|}\hline
g&p&k&k_2&NS_h&\alpha\\\hline
22&35&3&12&(22,21,19,18,16,15,13,12,10,9,7,6,4,3,1,0)&23\\\hline
24&16&7&7&(24,23,17,16,15,14,10,9,8,7,6,5,3,2,1,0)&27\\\hline
24&32&7&9&(24,23,17,16,15,14,10,9,8,7,6,5,3,2,1,0)&27\\\hline
24&33&5&13&(24,23,19,18,14,13,11,10,9,8,6,5,4,3,1,0)&25\\\hline
24&35&11&16&(24,23,13,12,11,10,8,7,5,4,2,1,0)&26\\\hline
24&37&3&12&(24,23,21,20,18,17,15,14,12,11,9,&24\\
&&&&8,6,5,3,2,0)&\\\hline
25&37&13&17&(25,24,14,13,12,11,8,7,5,4,3,2,1,0)&30\\\hline
25&43&9&19&(25,24,20,19,16,15,11,10,7,5,2,0)&18\\\hline
26&42&11&19&(26,25,22,21,18,17,15,14,11,10,7,6,4,2,0)&22\\\hline
26&47&5&19&(26,25,21,20,16,15,11,10,7,5,2,0)&19\\\hline
28&44&3&15&(28,27,25,24,22,21,19,18,16,15,13,12,10,9,&29\\
&&&&7,6,4,3,1,0)&\\\hline
28&44&7&19&(28,27,21,20,16,15,14,13,9,8,&25\\
&&&&7,6,3,1,0)&\\\hline
28&44&13&17&(28,27,18,17,15,14,11,10,8,&26\\
&&&&7,5,4,2,0)&\\\hline
29&45&7&13&(29,28,22,21,16,14,9,7,3,0)&13\\\hline
29&55&16&24&(29,28,22,21,15,14,13,12,8,7,6,4,1,0)&23\\\hline
30&39&7&11&(30,29,23,22,19,18,16,15,12,11,9,7,5,4,2,0)&21\\\hline
30&43&15&20&(30,29,15,14,13,12,10,9,7,6,4,3,2,1,0)&34\\\hline
30&46&3&15&(30,29,27,26,24,23,21,20,18,&30\\
&&&&17,15,14,12,11,9,8,6,5,3,2,0)&\\\hline
30&53&3&18&(30,29,25,24,20,19,15,14,10,8,5,3,0)&20\\\hline
30&58&7&25&(30,29,23,22,16,15,14,13,9,8,7,6,5,4,2,1,0)&32\\\hline
\end{array}$
\caption{Non-zero sequences of $K_{p,k}$ in non-L-space homology spheres up to $22\le p\le 30$.}
\label{nonzeroL2}
\end{center}
\end{table}

\noindent
Motoo Tange\\
University of Tsukuba, \\
Tennodai 1-1-1 Tsukuba, Ibaraki 305-8571, Japan. \\
tange@math.tsukuba.ac.jp

\end{document}